\documentclass[12pt]{amsart}
\usepackage{latexsym}
\usepackage[T2A,T1]{fontenc}
\usepackage[utf8]{inputenc}
\usepackage[english]{babel}
\usepackage{a4}
\usepackage{amsmath}
\usepackage{amssymb}
\usepackage{amsthm}
\usepackage{mathtools}%
\usepackage{bm}%
\usepackage{xcolor}
\usepackage{tikz}
\usetikzlibrary{shapes,calc,patterns}
\usepackage{url}
\usepackage{enumerate}
\usepackage{mdwlist}%
\usepackage{braket}
\usepackage{mleftright}%
\usepackage{array}
\usepackage{tabularx}%
\usepackage{booktabs}%
\usepackage[bookmarks]{hyperref}

\theoremstyle{definition}
\newtheorem{definition}{Definition}[section]
\theoremstyle{plain}
\newtheorem{theorem}[definition]{Theorem}
\theoremstyle{plain}
\newtheorem{thm}[definition]{Theorem}
\theoremstyle{plain}
\newtheorem{proposition}[definition]{Proposition}
\theoremstyle{plain}
\newtheorem{lemma}[definition]{Lemma}
\theoremstyle{plain}
\newtheorem{lem}[definition]{Lemma}
\theoremstyle{plain}
\newtheorem{corollary}[definition]{Corollary}
\theoremstyle{plain}
\newtheorem{cor}[definition]{Corollary}
\theoremstyle{definition}

\theoremstyle{remark}

\theoremstyle{remark}

\theoremstyle{definition}

\theoremstyle{definition}

\numberwithin{equation}{section}

\DeclareMathOperator{\Hh}{\textsl{H}}%
\DeclareMathOperator{\Inv}{\mathsf{Inv}}%
\DeclareMathOperator{\POL}{Pol}%
\DeclareMathOperator{\Con}{Con}%
\newcommand{\GF}[1]{\mathsf{GF}(#1)}%
\newcommand{\ab}[1]{{\mathbf{#1}}}%
\newcommand{\type}[2]{\mathtt{typ}(#1,#2)}%
\newcommand{\Mtct}[3]{\mathtt{M}_{\ab{#1}}(#2,#3)}%
\newcommand{\quotienttct}[2]{\Braket{#1,#2}}
\newcommand{\interval}[2]{\mathtt{I}[#1,#2]}%
\newcommand{\Interval}[2]{\mathbf{I}[#1,#2]}
\newcommand{\bottom}[1]{0_{#1}}%
\newcommand{\uno}[1]{1_{#1}}%
\newcommand{\Cg}[2][\ab{A}]{\text{Cg}_{#1}(#2)}%
\newcommand{\Sym}[1]{\textsf{Sym}(#1)}
\newcommand{\numberset}{\mathbb}
\newcommand{\N}{\numberset{N}}

\newcommand{\Z}{\numberset{Z}}

\newcommand{\finset}[1]{\{1,\dots, #1\}}%
\DeclareMathAlphabet{\mathbfsl}{OT1}{cmr}{bx}{it}

\renewcommand{\vec}[1]{{\boldsymbol{#1}}}%
\let\tp=\vec%
\newcommand{\ari}[1]{_{#1}}%
\newcommand{\supi}[1]{^{(#1)}}%
\newcommand{\strset}[1]{\mathcal{#1}}%
\newcommand{\restrict}[1]{\rvert_{#1}}%
\DeclareMathAlphabet{\mathsc}{OT1}{cmr}{m}{sc}
\DeclareMathAlphabet{\mathbfsl}{OT1}{cmr}{bx}{it}
\newcommand{\vb}[1]{\mathbfsl{#1}} 
\newcommand{\sottospazio}[2]{\text{Span}_{#1}(#2)}
\newcommand{\submodule}[1]{\Braket{#1}}
\newcommand{\card}[1]{\mleft\lvert#1\mright\rvert}%

\newcommand{\algop}[2]{(#1,#2)}
\newcommand{\lcover}{\prec}
\DeclareMathOperator{\dom}{\texttt{dom}}
\newcommand{\matrixvec}[1]{\overline{\vec{#1}}}
\newcommand{\ringmod}[2]{\mbox{}_{#1} #2}
\newcolumntype{C}[1]{>{\centering}p{#1}} %
\title[Polynomial interpolation of partial functions]{Polynomial interpolation of partial functions in finite algebras with a Mal'cev term}
\author{Erhard Aichinger}
\address[Erhard Aichinger]{
Institut für Algebra,
Johannes Kepler Universität Linz, Altenberger Straße 69, 4040 Linz, Austria}
\email{erhard@algebra.uni-linz.ac.at}

\author{Mario Kapl}
\address[Mario Kapl]{Department of Engineering \& IT, Carinthia University of Applied Sciences, Europastraße 4, 9524 Villach, Austria}
\email{m.kapl@fh-kaernten.at}

\author{Bernardo Rossi}
\address[Bernardo Rossi]{Department of Algebra, Faculty of Mathematics and Physics, Charles University, Sokolovsk\'{a} 49/83, 186 75 Praha 8, Czechia.}
\email{bernardo.rossi96@gmail.com}

\subjclass[2010]{08A05, 08A40}
\keywords{Polynomial interpolation, Mal'cev algebras, Congruence identities}
\thanks{Supported by the Austrian Science Fund (FWF):~P33878, and
by Charles University under grant no. PRIMUS/24/SCI/008. }
\begin{document}
\date{\today}
\begin{abstract}
  We provide polynomial completeness results for finite algebras
  in congruence permutable varieties.
  In 2001, Idziak and S{\l}omczy{\'n}ska introduced the completeness
  concept of being \emph{polynomially rich}: a finite algebra is polynomially
  rich if every function preserving congruences and the Tame Congruence
  Theory labelling of prime quotients in the congruence lattice is
  a polynomial function of the algebra. We call a finite algebra
  \emph{strictly polynomially rich} if every partial congruence and type
  preserving function is a polynomial function, and we describe
  strictly polynomially rich algebras in congruence permutable varieties.
\end{abstract}  
\maketitle

\section{Introduction}

The guiding question of this paper is the characterization of those partial
functions on an algebraic structure that can be interpolated by polynomials.
We seek for a description of those algebras in which the interpolable functions
admit a concise description.
Here, we adopt the viewpoint of \emph{universal algebra} \cite{BurSan81, MMT:ALVV}
and consider an algebraic structure as a universe together
with a set of operations.
For instance, the symmetric group $\ab{S}_3 := (S_3, \circ)$ is an algebraic structure.
Its polynomial functions are the functions of the form
$(x_1, \ldots, x_n) \mapsto t(x_1,\ldots, x_n, s_1, \ldots, s_n)$, where $n \in \N$ and
$t(x_1,\ldots, x_n, s_1, \ldots, s_n)$ is a well formed expression using
the operations of the algebra and constants $s_1, \ldots, s_n$ taken
from the algebra. For example, $f(x_1, x_2) := x_1 \circ ((1\,3\,2) \circ x_2)$
is a binary polynomial function of $\ab{S}_3$.

Let $\ab{A}$ be any algebraic structure, let $k \in \N$, let
$T$ be a subset of $A^k$, and suppose that a function
$f : T \to A$ is given by a table of its values. We would like to determine
whether $f$ can be interpolated by a polynomial function of $\ab{A}$; in
other words, we ask whether there exists a $k$-ary polynomial function
$p$ of $\ab{A}$ such that the restriction $p|_T$ of $p$ to $T$ satisfies
$p|_T = f$. When the answer is ``yes'', an expression representing $p$
witnesses this answer. But how could we witness the answer ``no''?
Here, a tool comes from clone theory \cite{PosKal79}. Suppose that
$B$ is a subalgebra of $\ab{A}^n$ with $\{(a,\ldots,a) \mid a \in A\} \subseteq B$
(we will call such such subalgebras \emph{diagonal}). Then if
$\vb{b}\supi{1} = (b\supi{1}_1, \ldots, b\supi{1}_n) \in B$, \ldots,
$\vb{b}\supi{k} = (b\supi{k}_1, \ldots, b\supi{k}_n) \in B$,
$p$ is
a $k$-ary polynomial function of $\ab{A}$, the fact that $B$ is a diagonal
subalgebra of $\ab{A}^n$ implies that
$(p (b\supi{1}_1, \ldots, b\supi{k}_1), \ldots,
  p (b\supi{1}_n, \ldots, b\supi{k}_n)) \in B$.
With this tool, we can prove that there is no unary polynomial $f$ of
$\ab{A} := (\Z_4, +)$ with $f(0) = 0$ and $f(2) = 1$ (set
$B = \{ (x,y) \in A^2 \mid x - y = 2 \}$)
and no binary polynomial with $f(0,0)=f(1,0)=f(0,1)=0,f(1,1) = 1$
(set $B = \{ (x_1, \ldots, x_4) \mid x_1 - x_2 = x_3 - x_4 \}$
and $\vb{b}\supi{1} = (0,0,1,1)$, $\vb{b}\supi{2} = (0,1,0,1)$ and
observe that $(0,0,0,1) \not\in B$).
This leads to the question on which $B$'s we should test a partial function
when hoping to disprove its polynomiality. As a first task, we need
to locate diagonal subalgebras of $\ab{A}^n$. A source of diagonal subalgebras
of $\ab{A}^2$
are kernels of homomorphisms: if $\varphi$ is a homomorphism from $\ab{A}$ to
an algebra $\ab{C}$, then $B := \{(x_1, x_2) \mid \varphi(x_1) = \varphi (x_2)\}$ is
a diagonal subalgebra of $\ab{A}^2$; binary
relations arising in this way as kernels of homomorpisms are called
\emph{congruences} in universal algebra \cite[Ch.II,\S 5]{BurSan81}.
For some algebras, testing partial functions with these congruences provides
a criterion for polynomial interpolability.
Following \cite{We:PVKU}, we call an algebra $\ab{A}$ \emph{affine complete}
if for every $k \in \N$ every (total) function $f : A^k \to A$
that preserves all congruences of $\ab{A}$ is a polynomial function.
There has been research to characterize affine complete
groups \cite{KK:AfCG,Ec:ACOG,MM:ROCP}, but interestingly enough,
it is still not known whether the problem
\begin{quote}
\emph{Given:} a finite group $\ab{G}$. \emph{Asked:} Is $\ab{G}$ affine
complete?
\end{quote}
is algorithmically decidable.
The situation becomes more regular when we consider the stronger concept
of being \emph{strictly affine complete}: a finite algebra $\ab{A}$ is strictly
affine complete if for every $k \in \N$, every subset $T$ of $A^k$,
and every $k$-ary congruence preserving function $f:T \to A$, there is
a polynomial function $p$ interpolating $f$ on $T$. This property
is stronger than affine completeness, which only considers \emph{total}
congruence preserving functions. One reason is that an
algebra $\ab{A}$ may possess a \emph{partial}
congruence preserving
function $f$ defined on a subset $T$ of $A^k$ that cannot be extended to
a total congruence preserving
function on $A^k$: such an $f$ cannot be interpolated by a polynomial
function and therefore excludes that $\ab{A}$ is strictly affine complete.
Indeed, from \cite{KK:CFEP}, it follows that such nonextendible congruence preserving
functions exist on all algebras whose congruence lattice is not distributive.
J.\ Hagemann an C.\ Herrmann \cite{HH:ALEC} showed that a finite algebra is strictly affine complete
if it has a ternary Mal'cev polynomial (this is a polynomial satisfying
$p(a,b,b)=p(b,b,a)=a$ for all $a,b \in A)$ and it is
\emph{neutral}; for finite algebras with a Mal'cev polynomial,
neutrality is a concept that can best be expressed
using commutator theory \cite{FM:CTFC}, where an algebra is defined to be
neutral if $[\alpha, \alpha] = \alpha$ for all congruences $\alpha$,
or Tame Congruence Theory \cite{HM:TSOF}, in which we express neutrality
by claiming that all prime quotients in the congruence lattice are
of TCT-type $\mathbf{3}$.
An algebra $\ab{A}$ is called a \emph{Mal'cev algebra} if it has a ternary term operation $d$ satisfying $d(a,b,b) = d(b,b,a) = a$ for all $a,b \in A$;
by a fundamental result of A.I.Mal'cev \cite{Ma:OTGT} these are the
algebras that generate congruence permutable varieties.
In \cite[Theorem~5.1]{Ai:OHAH}, it is proved that
a finite algebra is strictly affine complete if
every $3$-ary partial congruence preserving
function is polynomial. Note that for $m \in \N$, the group $(\Z_2^m,+)$ has the property
that every \emph{unary} partial congruence preserving function can be interpolated
by a polynomial function \cite[Corollary~11.2]{AI:PIIE}, but $(\Z_2^m, +)$ is not strictly affine complete.

Another completeness property has been introduced in \cite{IS:PRA}.
It starts from the observation that on an algebra $\ab{A} = \algop{A}{F}$, a function
$f : A^k \to A$ is congruence preserving if the expanded algebra
$\ab{A} + f := \algop{A}{F \cup\{f\}}$ has the same congruence relations
as $\ab{A}$. Denoting the congruences of $\ab{A}$
by $\Con (\ab{A})$,  we can express that $\ab{A}$ is affine
complete by saying that every function $f$  with
$\Con (\ab{A} + f) = \Con (\ab{A})$ is 
polynomial. In the 1980's, D.\ Hobby and R.\ McKenzie provided
another invariant of polynomial functions: a pair
of congruences $(\alpha, \beta) \in \Con (\ab{A})$ with
$\alpha < \beta$ and no $\gamma \in \Con (\ab{A})$ with 
$\alpha < \gamma < \beta$ is called a \emph{prime section}
of $\Con(\ab{A})$, and we write $\alpha \lcover \beta$.
With each prime section, \cite{HM:TSOF}
associates a
\emph{type} which, sloppily put, describes how polynomial functions behave on
$\beta$-classes modulo $\alpha$. The \emph{labelled congruence lattice}
of $\ab{A}$ is denoted by $\Con^* (\ab{A})$: It is the digraph
with vertex set $V := \Con(\ab{A})$, edges $E := \{(\alpha, \beta) \in V^2 \mid
\alpha \lcover \beta\}$, together with a labelling function
$t$ with domain $E$ that associates  to each edge its type.
P.~Idziak and K.\ S\l omczy\'nska \cite{IS:PRA} addressed the following
question: which algebras have the property that every function
$f: A^k \to A$ with $\Con^* (\ab{A}) = \Con^* (\ab{A} + f)$ is
polynomial? They called such algebras \emph{polynomially rich}
and characterized those finite algebras in congruence permutable varieties
that have the property that all their homomorphic images are
polynomially rich \cite[Theorem~24]{IS:PRA}.

At first glance, it is not obvious how to define polynomial richness
for \emph{partial functions}. For a partial function
$f$, the structure $\ab{A} + f$ is an algebraic structure with
partial operations and hence notions from TCT are not available.
However, calling a function \emph{type-preserving} if $\Con^* (\ab{A}) = \Con^* (\ab{A} + f)$, there is a way of describing type preserving functions
that can be extended to partial functions.
First, we observe that if $\ab{A}$ is a finite algebra
with a Mal'cev polynomial $d$, then the type of
$\alpha \lcover \beta$ is $\mathbf{2}$ if
the commutator $[\beta, \beta] \le \alpha$,
and $\mathbf{3}$ otherwise. The commutator inequality
$[\beta, \beta] \le \alpha$ is equivalent to
the condition that $\rho (\alpha, \beta) =
  \{(x_1,x_2,x_3,x_4) \in A^4 \mid
    x_1 \, \beta \, x_2 \, \beta \, x_3,
    d(x_1,x_2,x_3) \, \alpha \, x_4 \}$
    is a subalgebra of $\ab{A}^4$
    (cf. \cite[Lemma~2.4]{AM:PCOG})\footnote{When $[\beta, \beta] \le \alpha$,
    then $\rho (\alpha, \beta)$ does not depend
    on the choice of the Mal'cev polynomial $d$.
    The reason is that $( (x_3/\beta)/(\alpha/\beta), d)$ is    an abelian Mal'cev algebra, whose Mal'cev term
    is unique.}.
Now a  total function $f : A^k \to A$ is
type-preserving if it preserves all congruences
of $\ab{A}$, and for each prime section $\alpha \lcover \beta$
with $[\beta, \beta] \le \alpha$,
$\rho (\alpha, \beta)$ is a subalgebra of
$(\ab{A} + f)^4$. Using the concept of \emph{preserving
a relation}, we can formulate type preservation for
partial functions.
For an algebra $\ab{A}$, $k,n \in \N$,
$B \subseteq A^n$, $T \subseteq A^k$ and $f : T \to A$, we
call $B$ a \emph{relation} on $A$ and 
say that $f$ \emph{preserves} $B$ if for all
$\vb{b}\supi{1} = (b\supi{1}_1, \ldots, b\supi{1}_n) \in B$, \ldots,
$\vb{b}\supi{k} = (b\supi{k}_1, \ldots, b\supi{k}_n) \in B$
with $(b\supi{1}_1, \ldots, b\supi{k}_1) \in T$, \ldots,
$(b\supi{1}_n, \ldots, b\supi{k}_n) \in T$, we have
 $(f (b\supi{1}_1, \ldots, b\supi{k}_1), \ldots,
   f (b\supi{1}_n, \ldots, b\supi{k}_n)) \in B$.
  We say that a partial function $f$
   on an algebra  $\ab{A}$ with a Mal'cev polynomial
   $d$ is \emph{type preserving} %
   if it preserves all congruences and, for each prime section
   $(\alpha, \beta)$ with $[\beta, \beta] \le \alpha$,
   it preserves $\rho(\alpha, \beta)$.
   We call an algebra \emph{strictly polynomially rich}
   if every partial function that preserves all congruences and
   all relations $\rho (\alpha, \beta)$, where
   $\alpha, \beta \in \Con (\ab{A})$ with $\alpha \lcover \beta$ and $[\beta, \beta] \le \alpha$,
   can be interpolated by a polynomial function.

   However, when dealing with partial functions,
   strict affine completeness and strict polynomial
   richness immediately imply that the algebra has
   a distributive congruence lattice: the reason
   is that the $5$-ary near-unanimity function $u_5$
   (defined only on its ``natural'' domain) preserves
   every $4$-ary relation of $\ab{A}$.
   But if $u_5$ can be interpolated by a polynomial function,
   then $\ab{A}$ has a near unanimity polynomial;
   one consequence of this \cite{Mi:NUIA}, \cite[Lemma~1.2.12]{KP:PCIA}
   is that $\ab{A}$ has a distributive congruence lattice.
   This results in the fact that finite strictly polynomially
   rich Mal'cev algebras are strictly affine complete (Corollary~\ref{cor:general}).

   Considering only functions of bounded arity, 
   an algebra $\ab{A}$ is called \emph{strictly $k$-affine complete}
   if for every finite subset $T$ of $\ab{A}^k$ and for every
   congruence preserving function $f:T \to A$, there is
   a polynomial function of $\ab{A}$ with $p|_T = f$.
   From~\cite[Proposition~5.2]{Ai:OHAH} it follows that every
   strictly $2$-affine complete Mal'cev algebra is strictly
   affine complete.
   We say that a finite algebra
   $\ab{A}$ is \emph{strictly $k$-polynomially rich}
   if for every finite subset $T$ of $\ab{A}^k$ and for every
   type preserving function $f:T \to A$, there is
   a polynomial function of $\ab{A}$ with $p|_T = f$.
   We do not have to dig very deep to find out 
   that every
   finite strictly $5$-polynomially rich Mal'cev algebra
   is strictly polynomially rich -- it is even strictly affine complete. 
   (Corollary~\ref{cor:general}).
   In 
   Theorem~\ref{teor:char_hered_scr_4_p_r}, we show that
   for a finite Mal'cev algebra, it is sufficient for strict polynomial richness
   to have all its homomorphic images strictly $4$-polynomially rich.
   This result is optimal
   since by a result of the second author, the group
   $\algop{\Z_2}{+}$ is strictly $3$-polynomially rich but
   not strictly $4$-polynomially rich (cf.
   Theorem~\ref{teor:result_Kapl_thesis}, \cite{Kap05}). 
   The bulk of the present paper is devoted to the characterization
   of strictly $k$-polynomially rich finite Mal'cev algebras for
   $k \in \{1,2,3\}$, and the main results are stated
   in Section~\ref{sec:content}. For $k \in \{2,3\}$ we characterize those
   algebras for which every homomorphic image is strictly $k$-polynomially rich
   (Theorem~\ref{teor:char_hered_scr_23_p_r}); for $k=1$ we obtain
   a characterization under the additional assumption that $\ab{A}$ is
   congruence regular (Theorem~\ref{teor:char_hered_scr_1_p_r}).

   We first need to study strictly polynomially rich modules over simple
   rings: this has been done in a thesis by the second author \cite{Kap05}
   and is presented in Sections~\ref{sec.preliminary_modules}
   and~\ref{sec:modules}. The description for arbitrary finite Mal'cev algebras
   is a part of the third author's thesis~\cite{Ro:UAGA} and is
   developed here in Sections~\ref{sec:preliminary_on_malcevalgebras}
   to~\ref{sec:proof_main_results}.
   
\section{Completeness and near unanimity functions}

Polynomial completeness properties referring to
partial functions are usually quite strong. One reason lies in the
following lemma about the near-unanimity operation. For a set $A$ and
$k \ge 3$, we let
   $U_k \subseteq A^{k}$ be defined by
   $U_k := \{ (x,\ldots,x) \mid x \in A \} \cup
   \{ (y,x,x\ldots,x) \mid x,y \in A \} \cup
   \{ (x,y,x\ldots,x) \mid x,y \in A \} \cup \dots \cup
   \{ (x,x,x\ldots,x,y) \mid x,y \in A \}$, and we define
    $u_k: U_k \to A$ to be the $k$-ary \emph{near-unanimity} function that maps
   every $\vb{z} \in U_k$ to the element appearing at least $k-1$ times
   in $\vb{z}$.
\begin{lem} \label{lem:nu}
  Let $A$ be a set, let $k \ge 3$, let $n \in \N$ be such that
  $n < k$, and let $B \subseteq A^n$. Then the partial function
  $u_k$ preserves the relation $B$.
\end{lem}
\begin{proof}
 Suppose $(b\supi{1}_1, \ldots, b\supi{k}_1) \in U_k$, \ldots,
   $(b\supi{1}_n, \ldots, b\supi{k}_n) \in U_k$
   with 
$\vb{b}\supi{1} = (b\supi{1}_1, \ldots, b\supi{1}_n) \in B$, \ldots,
 $\vb{b}\supi{k} = (b\supi{k}_1, \ldots, b\supi{k}_n) \in B$.
   Then there are $x_1, \ldots, x_n, y_1, \ldots, y_n \in A$
   such that
   $(b\supi{1}_i, \ldots, b\supi{k}_i) \in
   \{(x_i, x_i, \ldots, x_i)\} \cup
   \{(y_i, x_i, \ldots, x_i)\} \cup \dots \cup
   \{(x_i, x_i, \ldots, y_i)\}$ for all $i \in \{1,\ldots,n\}$.
   Let $j_i \in \{1, 2, \ldots, k\}$ be equal to $1$
   if $(b\supi{1}_i, \ldots, b\supi{k}_i) = (x_i,\ldots, x_i)$;
   otherwise let $j_i$ the index of the
   dissenter $y_i$, i.e., let $j_i$ be such that
   $b\supi{j_i}_i = y_i$. In both cases, we have
      $b\supi{m}_i = x_i$ for $m \in \{1,\ldots,k\} \setminus \{j_i\}$.
   Since $n < k$, there is
   $r \in \{1,\ldots,k\} \setminus \{j_i \mid i \in \{1,\ldots,n\} \}$.
    Then $(b\supi{r}_1, \ldots, b\supi{r}_n) = (x_1, \ldots, x_n)$
   and thus $(x_1, \ldots, x_n) \in B$.
   Since $(x_1, \ldots, x_n)  =
          (u_k ( b\supi{1}_1, \ldots, b\supi{k}_1),\ldots, 
         u_k ( b\supi{1}_n, \ldots, b\supi{k}_n))$,
  we obtain $ (u_k ( b\supi{1}_1, \ldots, b\supi{k}_1),\ldots, 
  u_k ( b\supi{1}_n, \ldots, b\supi{k}_n)) \in B$,
  which concludes the proof that $u_k$ preserves $B$.
\end{proof}
A consequence is that many polynomial completeness concepts involving
partial functions boil down to strict affine completeness.
\begin{thm} \label{thm:general1}
  Let $\ab{A}$ be a finite algebra with a Mal'cev term,
  let $n \ge 2$ and
   let $\mathcal{R}$ be a set of  relations
   on $A$ of arity at most $n$, i.e., $\mathcal{R} \subseteq \bigcup_{m \in \{1,\ldots, n\}}
   \mathcal{P} (A^m)$.
   We assume that every partial $(n+1)$-ary $\mathcal{R}$-preserving
   function can be interpolated by a
   polynomial function of $\ab{A}$.
   Then $\ab{A}$ is strictly affine complete.
\end{thm}
\begin{proof}
   By Lemma~\ref{lem:nu}, $u_{n+1}$ preserves
   every relation $R \in \mathcal{R}$.
    Hence from the assumptions, we have a polynomial
    function $p:A^{n+1} \to A$ with $p|_T = u_{n+1}$. Let $\ab{A}^*$
    be the algebra that we obtain by adjoining to $\ab{A}$ all its
    elements as nullary operations.
   Then $p$ provides a near-unanimity term for $\ab{A}^*$, and thus
   the variety generated by $\ab{A}^*$ is congruence distributive
   \cite{Mi:NUIA}.
  Since $\ab{A}$ has Mal'cev term, this variety is then arithmetical.
  Using the implication (2)$\Rightarrow$(4) of
  \cite[Theorem~5.1]{Ai:OHAH}, we obtain that $\ab{A}$ is strictly
  affine complete.
\end{proof}

\begin{cor} \label{cor:general}
  Let $\ab{A}$ be a finite strictly $5$-polynomially rich
  Mal'cev algebra. Then $\ab{A}$ is strictly affine complete.
\end{cor}
\begin{proof}
  Being type preserving is expressed by the preservation
  of relations of arity at most~$4$.
  Now the claim follows from Theorem~\ref{thm:general1}.
\end{proof}

\section{Characterization of strictly polynomially rich Mal'cev algebras}\label{sec:content}
Our description of strictly polynomially rich algebras will involve
certain concepts from \emph{universal algebra} \cite{BurSan81}.
In particular, we will use congruence relations and their
(binary) commutator as introduced, e.g., in \cite{MMT:ALVV}.
Several useful properties of polynomial functions with respect to commutators
are listed in \cite[Section~2]{Aic06}. In a Mal'cev algebra,
the largest congruence $\gamma$ with $[\gamma, \beta] \le \alpha$ is
called the \emph{centralizer} of $\beta$ modulo $\alpha$ and denoted
by $(\alpha:\beta)$. In \cite{IS:PRA}, the following condition
has been used to describe polynomially rich algebras.
\begin{definition}[{cf.~\cite{IS:PRA, AI:PIIE, AM:TOPC}}]\label{def:SC1}
A Mal'cev algebra $\ab{A}$ satisfies the condition (SC1) 
if for every strictly meet irreducible congruence $\mu$ of $\ab{A}$
we have $(\mu:\mu^+)\leq \mu^+$. 
\end{definition}
We also need the following condition on prime quotients; this condition
is a generalization of the condition (AB2) used in \cite{AI:PIIE}.

\begin{definition}[{cf.~\cite[Definition~1.2]{Ros24}}]\label{def:definition_ABp}
Let $\ab{A}$ be a Mal'cev algebra, let
$p$ be a prime number, and let $\gamma\in \Con\ab{A}$.
We say that $(\ab{A}, \gamma)$ 
has property (AB$p$) if for all $\alpha, \beta\in \interval{\bottom{A}}{\gamma}$
with $\alpha\prec\beta$ and $[\beta, \beta]\leq \alpha$, and 
for each $a\in A$ we have 
that $\card{(a/\alpha)/(\beta/\alpha)}\in\{1, p\} $.
In other words, for each abelian prime quotient 
$\quotienttct{\alpha}{\beta}$
in the interval $\interval{\bottom{A}}{\gamma}$,
each $\beta$-class is the union of either $1$ or $p$ distinct
$\alpha$-classes. 
We say that $\ab{A}$ has property (AB$p$) if $(\ab{A}, \uno{A})$
has property (AB$p$).  
\end{definition}
The notion of type-preserving function was
first introduced in \cite{IS:PRA}. The premise is that polynomial functions 
preserve the labelling of prime quotients with the 
TCT-types $\{\tp{1},\dots \tp{5}\}$: 
Let $\ab{A}$ be a finite algebra and let $f\colon A^k\to A$ be 
a polynomial function. Then the algebra $\ab{A}+f$ 
and the algebra $\ab{A}$ have the same congruence lattice 
and for each prime quotient $\quotienttct{\eta}{\theta}$ the type of $\quotienttct{\eta}{\theta}$ 
is the same in $\ab{A}$ and $\ab{A}+f$. 
Hence for finite algebras, we could define
that $f$ preserves the types of $\ab{A}$
if 
for each prime quotient $\quotienttct{\eta}{\theta}$ 
the type of $\quotienttct{\eta}{\theta}$ 
is the same in $\ab{A}$ and $\ab{A}+f$. 
In the case of Mal'cev algebras, 
the types of prime quotients are determined by commutators, and commutators
can be expressed by the preservation of certain $4$-ary relations
\cite[Lemma~2.4]{AM:PCOG}.
Thus we rather use the following definition
(cf. \cite[Section~5]{AM:TOPC}):
\begin{definition}\label{def:partial_type_preserving}
Let $\ab{A}$ be an algebra with a Mal'cev polynomial $d$.
For each prime quotient $\quotienttct{\alpha}{\beta}$ of 
type $\tp{2}$ we define $\rho(\alpha, \beta)$
as 
	\[
	\rho(\alpha, \beta):=
      \{a_1, a_2, a_3, a_4) 
\in
A^4\mid a_1\mathrel{\beta}a_2\mathrel{\beta}a_3 \text{ and } d(a_1, a_2, a_3)\mathrel{ \alpha } a_4 \}.
	\]
Let $f\colon T\subseteq A^n \to A$. 
We say that $f$ is \emph{type-preserving} if
$f$ is congruence-preserving and $f$ preserves $\rho(\alpha, \beta)$
for each prime quotient $\quotienttct{\alpha}{\beta}$ in $\Con\ab{A}$ of type $\tp{2}$.
We say that $\ab{A}$ is 
\emph{strictly $k$-polynomially
rich} if each $k$-ary partial type-preserving function
defined on $A$ can be interpolated on its domain by
a polynomial function of $\ab{A}$. 
\end{definition}
Our description will use notions from tame congruence theory:
the \emph{subtype} of a tame quotient has been defined  
in \cite{Kea93}; for a prime quotient of type $\mathbf{2}$,
the subtype of $\quotienttct{\alpha}{\beta}$ is the number of $\alpha$-classes
in an $\quotienttct{\alpha}{\beta}$-trace; a different description of
subtypes in Mal'cev-algebras is given in Theorem~\ref{teor:from_coordinatization_to_TCT}
below. The \emph{extended type} has been
defined in \cite{IS:PRA} and consists of the type and subtype together.
Using these subtypes,
we can divide the class of tame quotients into four classes: those that
have completeness type $1$, $2$ or $3$, and those that are
of none of these types. In Theorem~\ref{teor:main_result_partial_type_preserving},
we will show that a finite Mal'cev algebra with (SC1) and
certain quotients of completeness type $k$ will be strictly $k$-polynomially
rich. The dividing line between
the completeness types comes from
the description of strictly $k$-polynomially rich
modules in Theorem~\ref{teor:result_Kapl_thesis}.
\begin{definition}\label{def:ex_main_theorem}
Let $\ab{A}$ be a finite algebra with a Mal'cev polynomial, let
$\alpha, \beta\in \Con\ab{A}$ such that $\quotienttct{\alpha}{\beta}$
is tame, and let $m$ be the height of the interval $\interval{\alpha}{\beta}
= \{ \gamma \in \Con \ab{A} \mid \alpha \le \gamma \le \beta \}$.
We introduce the notion of \emph{completeness type} (short CT):
\begin{enumerate}
\item $\quotienttct{\alpha}{\beta}$ is of \emph{completeness type $1$} (CT1) if 
\begin{enumerate}
			\item $\type{\alpha}{\beta}=\tp{3}$; or \label{item:main_theorem_poly_rich_item_neutral}
			\item $\type{\alpha}{\beta}=\tp{2}$, the subtype of $\quotienttct{\alpha}{\beta}$ is
			$2$, $m=1$, and for each $a/\alpha\in (A/\alpha)/(\beta/\alpha)$ 
			we have $\card{(a/\alpha)/(\beta/\alpha)}\in\{1,2,4,8\}$; or\label{item:main_theorem_poly_rich_item_subtype2}
			\item $\type{\alpha}{\beta}=\tp{2}$, the subtype of $\quotienttct{\alpha}{\beta}$ is
			$3$, $m=1$, and for each $a/\alpha\in (A/\alpha)/(\beta/\alpha)$ 
			we have $\card{(a/\alpha)/(\beta/\alpha)}\in\{1,3,9\}$; or\label{item:main_theorem_poly_rich_item_subtype3}
			\item $\type{\alpha}{\beta}=\tp{2}$, the subtype of $\quotienttct{\alpha}{\beta}$ is
			$p$ with $p\in \{2,3,5\}$, $m\geq 1$, and $(\ab{A}/\alpha,\beta/\alpha)$  
			satisfies (AB$p$);\label{item:main_theorem_poly_rich_item_ABp}
		\end{enumerate}
		\item $\quotienttct{\alpha}{\beta}$ is of \emph{completeness type $2$} (CT2) if 
		\begin{enumerate}
			\item $\type{\alpha}{\beta}=\tp{3}$; or \label{item:main_theorem_poly_rich_item_neutral_2}
			\item $\type{\alpha}{\beta}=\tp{2}$, the subtype of $\quotienttct{\alpha}{\beta}$ is
			$p\in\{2,3\}$, $m=1$, and $(\ab{A}/\alpha,\beta/\alpha)$  
			satisfies (AB$p$); \label{item:main_theorem_poly_rich_item_ABp_2}
		\end{enumerate}
		\item $\quotienttct{\alpha}{\beta}$ is of \emph{completeness type $3$} (CT3) if
		\begin{enumerate}
			\item $\type{\alpha}{\beta}=\tp{3}$; or \label{item:main_theorem_poly_rich_item_neutral_3}
			\item $\type{\alpha}{\beta}=\tp{2}$, the subtype of $\quotienttct{\alpha}{\beta}$ is
			$2$, $m=1$, and $(\ab{A}/\alpha, \beta/\alpha)$
			satisfies (AB$2$).\label{item:main_theorem_poly_rich_item_ABp_3}
		\end{enumerate}
	\end{enumerate}
\end{definition}
Note that for each tame quotient $\quotienttct{\alpha}{\beta}$ of $\Con\ab{A}$ we have that
\[
\type{\alpha}{\beta}=\tp{3}\Longrightarrow \quotienttct{\alpha}{\beta} \text{ is (CT$3$) }\Longrightarrow
\quotienttct{\alpha}{\beta} \text{ is (CT$2$) }\Longrightarrow
\quotienttct{\alpha}{\beta} \text{ is (CT$1$)}.
\]
In \cite{IS:PRA}, it was found that the congruence lattice of algebras
with (SC1) contains certain elements that were called \emph{homogeneous}
in \cite{AI:PIIE}. Properties of these homogeneous elements were studied in
\cite[Section~7]{AM:TOPC}; for instance, every homogeneous element of
an algebraic modular lattice is distributive. 
\begin{definition}[{\cite[Definition~7.2]{AM:TOPC}}]\label{def:homogeneous}
Let $\ab{L}$ be a bounded lattice with smallest element $0$. An element $\mu$ of $L$ is 
	\emph{homogeneous} if 
	\begin{enumerate}
		\item $\mu>0$,
		\item for all $\alpha, \beta, \gamma, \delta\in L$ 
		with $\alpha\prec \beta \leq \mu$ and $\gamma\prec \delta \leq \mu$, 
		we have 
		$\interval{\alpha}{\beta}\leftrightsquigarrow\interval{\gamma}{\delta}$, and
		\item there are no $\alpha, \beta, \gamma,\delta\in L$ such that 
		$\alpha\prec\beta\leq\mu\leq\gamma\prec\delta$, and
		$\interval{\alpha}{\beta}\leftrightsquigarrow\interval{\gamma}{\delta}$. 
	\end{enumerate}
\end{definition}
In a finite Mal'cev algebra with (SC1), we can find a series of
distributive elements in its congruence lattice
\cite[Propositions~8.2 and~8.8]{AM:TOPC}; the core of this fact can be
traced back to the assertions (20) and (21) of \cite{IS:PRA}.
\begin{definition}[{\cite[Definition~7.4]{AM:TOPC}}]\label{def:hom_series}
Let $n\in\N$, let $\ab{L}$ be a finite modular lattice
with smallest element $0$, largest element $1$ and
with $\card{L}>1$.  A sequence $(\alpha_0, \alpha_1, \dots, 
	\alpha_{n})$ is a \emph{homogeneous series} 
	if $\alpha_0=0$, $\alpha_{n}=1$ and 
	for each $i\in \{1,\dots n\}$, $\alpha_i$ is a homogeneous 
		element of the lattice $\Interval{\alpha_{i-1}}{1}$.
\end{definition}
In Section~\ref{sec:homegeneous_seq} we will prove that 
each quotient $\quotienttct{\alpha_{i-1}}{\alpha_i}$
of a homogeneous series in a finite Mal'cev algebra with (SC1) is tame
(Proposition~\ref{prop:homogenuityandcolpementation}\eqref{item:simplecomplmodlatticebelowhom} and Lemma~\ref{lemma:projectivity_andTCT}).
For $k\in\{1,2,3\}$ we say that 
a homogeneous series is (CT$k$) if
for each $i\in\finset{n}$ the quotient $\quotienttct{\alpha_{i-1}}{\alpha_i}$
is (CT$k$). In Section~\ref{sec:proof_main_results}
we will prove the following theorems:
\begin{theorem}\label{teor:main_result_partial_type_preserving}
	Let $\ab{A}$ be a finite algebra with a Mal'cev polynomial that satisfies (SC1),
	let  $k\in\{1,2,3\}$,
	and let $(\mu_0, \dots, \mu_{n})$ be a homogeneous series
	of $\Con\ab{A}$.
	If $(\mu_0, \dots, \mu_{n})$
	is (CT$k$),
	then $\ab{A}$ is strictly $k$-polynomially rich.
\end{theorem}
For proving converses of this result, we had to assume that not only $\ab{A}$,
but even each of its homomorphic images is strictly $k$-polynomially rich;
the class of homomorphic images of $\ab{A}$ is abbreviated by $\Hh\ab{A}$.
In addition, at one place we assume that $\ab{A}$ is \emph{congruence regular},
which means that for all congruences $\alpha, \beta$ of $\ab{A}$ and all
$a \in A$ with $a/\alpha = a/\beta$, we have $\alpha = \beta$.
\begin{theorem}\label{teor:char_hered_scr_1_p_r}
	Let $\ab{A}$ be a finite congruence regular
	 algebra with a Mal'cev polynomial. 
	Then each algebra in $\Hh\ab{A}$ is strictly $1$-polynomially rich
	if and only if 
	$\ab{A}$ satisfies (SC1) and 
	each homogeneous series of $\Con\ab{A}$ is (CT1). 
\end{theorem}
\begin{theorem}\label{teor:char_hered_scr_23_p_r}
	Let $\ab{A}$ be a finite algebra with a Mal'cev polynomial, and let $k\in\{2, 3\}$. 
	Then the 
	following are equivalent:
	\begin{enumerate}
		\item Each algebra in $\Hh\ab{A}$ is strictly $k$-polynomially rich;\label{item:HA_in_char_hered_scr_23_p_r}
		\item $\ab{A}$ satisfies (SC1) and \label{item:SC1_condizioni_mu_in_char_hered_scr_23_p_r}
		each homogeneous series $(\mu_0\dots, \mu_{n})$ of $\Con\ab{A}$
		is (CT$k$). 
	\end{enumerate}
\end{theorem}
\begin{theorem}\label{teor:char_hered_scr_4_p_r}
	Let $\ab{A}$ be a finite algebra with a Mal'cev polynomial, 
	and 
	let $k\in\N\setminus\{1,2,3\}$.
	Then the following are equivalent
	\begin{enumerate}
		\item $\ab{A}$ is strictly $2$-affine complete;\label{item:2-affine_in_char_hered_scr_4_p_r}
		\item $\ab{A}$ is strictly $l$-affine complete for each $l\in\N\setminus\{1\}$;\label{item:l-affine_in_char_hered_scr_4_p_r}
		\item each algebra in $\Hh\ab{A}$ is strictly $k$-polynomially rich;\label{item:HAk_in_char_hered_scr_4_p_r}
		\item each algebra in $\Hh\ab{A}$ is strictly 4-polynomially rich;\label{item:HA_in_char_hered_scr_4_p_r}
		\item $\ab{A}$ is congruence neutral. \label{item:neutral_in_char_hered_scr_4_p_r}
	\end{enumerate}
\end{theorem}
The proofs will be given in
Section~\ref{sec:proof_main_results}.

\section{Preliminaries on modules over simple rings}\label{sec.preliminary_modules}
In this section we collect some 
preliminary results on modules over simple rings.
We start with the following consequence of Lemma~2.6 and Lemma~2.9
from \cite{AM:PCOG}:
\begin{lemma}\label{lemma:on_notions_of_type_preserving_and-commutatora_and_rho_for_modules}
	Let $\ab{V}$ be a module over a ring $\ab{R}$, 
	let $k\in \N$, let
	$T\subseteq V^k$, and let $f\colon T\to V$. 
	Then $f$ is type-preserving according to Definition~\ref{def:partial_type_preserving}
	if and only if $f$ preserves the congruences of $\ab{V}$ and for each 
	pair of  $\ab{R}$-submodules $A,B$ with $A\prec B$,
	the function $f$ preserves the relation
	\[
	\rho(A,B):=\{(x_1, x_2, x_3, x_4)\in V^4\mid x_1-x_2+x_3-x_4\in A,\,
	x_1-x_2,\, x_2-x_3\in B\}.
	\]
\end{lemma}
We write $\GF{q}$ for the finite field with $q$ elements.
For a finite field $\ab{D}$ and 
$n,m\in\N$, we let $\ab{M}_n(\ab{D})$ be the 
ring of $n\times n$-matrices with entries in $D$,
and we let $\ab{D}^{(n\times m)}$ 
be the $\ab{M}_n(\ab{D})$-module of 
$n\times m$-matrices with entries in $D$.
For a set $T\subseteq D^{(n\times m)}$
we let $\sottospazio{}{T}$ 
be the $\ab{D}$-subspace of $D^{(n\times m)}$
generated by $T$, and we let
$\submodule{T}$ be the $\ab{M}_n(\ab{D})$-submodule
of $\ab{D}^{(n\times m)}$ generated 
by $T$. 
For $i\in\finset{n}$, $j\in \finset{m}$,
$\vec{a}\in D^n$, we let:
$\vec{a}^T$ be the $1\times n$-matrix whose columns 
are the components of $\vec{a}$; 
$\vec{e}_i^n$ be the element of $D^n$ that is $1$ in the 
$i$-th component and zero elsewhere;
$\vec{1}_n$ be the vector of $D^n$ whose entries are all $1$;
$\vec{0}_n$ be the vector of $D^n$ whose entries are all $0$;
$\vec{0}_{n\times m}$ be the matrix in $D^{(n\times m)}$ whose entries are all zero;
$\vec{0}$ be the zero of the ring $\ab{M}_n(\ab{D})$;
$E_{i,j}$ be the matrix in $D^{(n\times m)}$ whose entry $(i,j)$ is $1$
and whose all other entries are $0$;
$M_{i,j}$ be the matrix in $\ab{M}_n(\ab{D})$ whose entry $(i,j)$ is $1$
and whose all other entries are $0$;
$E_{\vec{a}}^j$ be the matrix from $D^{(n\times m)}$ 
whose $j$-th column is the vector $\vec{a}$ and whose all other entries are $0$;
$M_{\vec{a}}^j$ be the matrix from $\ab{M}_n(\ab{D})$ 
whose $j$-th column is the vector $\vec{a}$ and whose all other entries are $0$.
Finally, for $k\in\N$, $l\in \finset{k}$, and 
$\matrixvec{x}\in (D^{(n\times m)})^{k}$, we let $\matrixvec{x}(l)$ be 
the $l$-th component of $\matrixvec{x}$.  

Next we state and prove a series of lemmata originally proved in~\cite{Kap05}.
The first lemma tells that we can restrict ourselves 
to functions that preserve $0$. 
\begin{lemma}\label{lemma:reducing_to_zero_preserving}
	Let $\ab{D}$ be a finite field, let $n,m, k\in \N$, and let us
	assume that each partial $k$-ary type-preserving function $f$
	defined on $(D^{(n\times m)})^k$ whose domain 
	contains the $k$-tuple
	$\matrixvec{0}_{n\times m}=(\vec{0}_{n\times m}, \dots, 
	\vec{0}_{n\times m})$ and that satisfies 
	$f(\matrixvec{0}_{n\times m})=\vec{0}_{n\times m}$
	can be interpolated by a polynomial function of 
	$\ab{D}^{(n\times m)}$.
	Then $\ab{D}^{n\times m}$ is strictly $k$-polynomially rich.
\end{lemma}
\begin{proof}
	Let $T\subseteq (D^{(n\times m)})^k$, and let
	$f\colon T\to D^{(n\times m)}$ be a type-preserving function.
	If $\matrixvec{0}_{n\times m}\in T$ and 
	$f(\matrixvec{0}_{n\times m})=\vec{0}_{n\times m}$
	then by assumption, $f$ can be interpolated by a polynomial function.
	Let us now consider the case that 
	$\matrixvec{0}_{n\times m}\in T$ and 
	$f(\matrixvec{0}_{n\times m})\neq \vec{0}_{n\times m}$.
	Let $M=f(\matrixvec{0}_{n\times m})$. 
	Then the map $\matrixvec{x}\mapsto f(\matrixvec{x})-M$
	is type-preserving, and zero-preserving, 
	hence there exists $p\in\POL\ari{k}\ab{D}^{(n\times m)}$
	such that for all $\matrixvec{x}\in T$ we have
	$p(\matrixvec{x})=f(\matrixvec{x})-M$.
	Thus, $f$ can be interpolated by the polynomial 
	function $\matrixvec{x}\mapsto p(\matrixvec{x}) + M$. 
	Let us now consider the case that
	$\matrixvec{0}_{n\times m}\notin T$.
	Let $\matrixvec{m}\in T$, and let us define 
	$T'=\{(\matrixvec{x}(1)-\matrixvec{m}(1), \dots, 
	\matrixvec{x}(k)-\matrixvec{m}(k))\mid \matrixvec{x}\in T\}$.
	Then $\matrixvec{0}_{n\times m}\in T'$.
	Let us define $g\colon T'\to D^{(n\times m)}$ by
	$g(\matrixvec{y})=f(\matrixvec{y}+\matrixvec{m})$,
	where the $+$ is now defined component-wise.
	Since $g$ is a composition of type-preserving functions
	it is a type-preserving function and
	furthermore, its domain contains $\matrixvec{0}_{n\times m}$.
	Thus, we can resort to the already discussed case
	and construct a $k$-ary polynomial function $p$
	of $\ab{D}^{(n\times m)}$ that interpolates 
	$g$ on $T'$. 
	Hence setting $q$ to be the polynomial function
	$\matrixvec{x}\mapsto p(\matrixvec{x}-\matrixvec{m})$, 
	for all $\matrixvec{x}\in T$ we have that  
	$q(\matrixvec{x})=p(\matrixvec{x}-\matrixvec{m})=
	g(\matrixvec{x}-\matrixvec{m})=f(\matrixvec{x}-\matrixvec{m}
	+\matrixvec{m})=f(\matrixvec{x})$.
\end{proof}

\begin{lemma}\label{lemma:technical_lemma_kapl}
	Let $n, q\in \N$ be such that 
	\[(n\in\{1,2,3\}\text{ and }q=2) \text{ or } (n\in\{1,2\}\text{ and } q=3)\]
	Then for all $T\subseteq (\GF{q})^n$, there exists $S\subseteq T$ that consists of linearly independent
	vectors such that 
	\[\forall t\in T\, \exists x_1, x_2, x_3\in (S\cup\{\vec{0}_n\})^3
	\colon t=x_1-x_2+x_3.\]
\end{lemma}
\begin{proof}
	In the case that $q=2$ we have that $x_1-x_2+x_3=x_1+ x_2+x_3$, 
	and the 
	result follows from the fact that the dimension of the
	space generated by $T$ is at most $3$.
	If $q=3$, the result follows from the
	fact that for all $x,y\in  (\GF{3})^n$ we have $2x=x-\vec{0}_n+x$,
	$2x+2y=x-y+x$, and $2x+y=y-x+\vec{0}_n$,
	and from the fact that the dimension of the space 
	generated by $T$ is at most $2$.
\end{proof}
The next lemma states that the module
$\{ E_\vec{a}^1  \mid \vec{a} \in \ab{D}^n\}$ is a primitive
$\ab{M}_n(\ab{D})$-submodule of $\ab{D}^{n \times m}$.
\begin{lemma}\label{lemmasm}
	Let $\ab{D}$ be a field, let $n,m\in\N$, let $N$ be a $\ab{M}_n(\ab{D})$-submodule
	of the module $\ab{D}^{(n\times m)}$, and let $\vec{a}\in D^n\setminus\{0\}$.
	If $E_\vec{a}^1 \in N$, then for each $\vec{b}\in D^n$ we have $E_\vec{b}^1\in N$.
\end{lemma}
\begin{proof}
	Let $\vec{b}\in D^n$, let $i\in \finset{n}$ be such that
	$a_i\neq 0$, and let
	$\vec{a}':=(b_1\cdot a_i^{-1}, \dots, b_n\cdot a_i^{-1})$. 
	Then $M_{\vec{a}'}^i\in \ab{M}_n(\ab{D})$, and therefore, 
	$E_{\vec{b}}^1=M_{\vec{a}'}^i\cdot E_{\vec{a}}^1\in N$. 
\end{proof} 
\begin{lemma}\label{lemmaeq}
	Let $\ab{D}$ be a field, let $n,m\in\N$, 
	let $A$ be a $\ab{M}_n(\ab{D})$-submodule
	of the module $\ab{D}^{(n\times m)}$,
	let $b\in D\setminus\{0\}$, let $\vec{b}:=b\cdot \vec{1}_n$, and let
	$T=\{\vec{0}_{n\times m}\}\cup \{E_{i,1}\mid i\in \finset{n}\}\cup \{E_\vec{b}^1\}$.
	If there exists a pair of distinct elements $M_1, M_2$ of $T$ 
	with $M_1\equiv M_2\bmod A$, then $E_{1,1}\in A$.
\end{lemma}
\begin{proof}
	Let us suppose that 
	there are $M_1, M_2\in T$ such that $M_1\neq M_2$ 
	and $M_1\equiv M_2\bmod A$. 
	Then $M_1-M_2\in A$, and by its definition, there exists 
	$i\in\finset{n}$ such that the $(i,1)$-th component of 
	$M_1-M_2$ is not zero. Thus, Lemma~\ref{lemmasm} 
	implies that for all $\vec{c}\in D^n$ the matrix
	$E_\vec{c}^1\in A$, and therefore,
	$E_{1,1}\in A$. 
\end{proof}
The proofs of the next five lemmata share
a common structure and are extremely elementary. 
We first state the five lemmata and then provide a
proof scheme that works for each of them. 
\begin{lemma}\label{lemmaeqq2}
	Let $\ab{D}=\GF{2}$, let $n\in\N\setminus\{1,2,3\}$, let $m\in \N$,
	let $T=\{\vec{0}_{n\times m}\}\cup \{E_{i,1}\mid 
	i\in \finset{n}\}\cup \{E_{\vec{1}_n}^1\}$,
	let $A$ be a $\ab{M}_n(\ab{D})$-submodule of $D^{(n\times m)}$.
	If there exist pairwise distinct $M_1, M_2, M_3\in T\setminus\{E_{\vec{1}_n}^1\}$
	with $M_1-M_2+M_3 \equiv E_{\vec{1}_n}^1\bmod A$
	then $E_{\vec{1}_n}^1\in A$.
\end{lemma}
\begin{lemma} \label{lemmaeqq3}
Let $\ab{D}=\GF{3}$, let $n\in\N\setminus\{1,2\}$, let $m\in \N$,
let $T=\{\vec{0}_{n\times m}\}\cup \{E_{i,1}\mid 
i\in \finset{n}\}\cup \{2\cdot E_{\vec{1}_n}^1\}$,
and
let $A$ be a $\ab{M}_n(\ab{D})$-submodule of $D^{(n\times m)}$.
If there exists pairwise distinct $M_1, M_2, M_3\in T\setminus\{2\cdot E_{\vec{1}_n}^1\}$
with $M_1-M_2+M_3 \equiv 2\cdot E_{\vec{1}_n}^1\bmod A$, or
$M_1+M_2\equiv E_{\vec{1}_n}^1\bmod A$, or
$2\cdot M_1-M_2\equiv 2\cdot E_{\vec{1}_n}^1\bmod A$, then $E_{1,1}\in A$.
\end{lemma}
\begin{lemma} \label{lemmaeqq5}
Let $\ab{D}=\GF{5}$, let $n\in\N\setminus\{1\}$, let $m\in \N$,
let $T=\{\vec{0}_{n\times m}\}\cup \{E_{i,1}\mid 
i\in \finset{n}\}\cup \{4\cdot E_{\vec{1}_n}^1\}$, 
and
let $A$ be a $\ab{M}_n(\ab{D})$-submodule of $D^{(n\times m)}$.  
If there exist pairwise distinct $M_1, M_2, M_3\in T\setminus\{4\cdot E_{\vec{1}_n}^1\}$
	with $M_1-M_2+M_3 \equiv 4\cdot E_{\vec{1}_n}^1\bmod A$, or
	$M_1+M_2\equiv 3\cdot E_{\vec{1}_n}^1\bmod A$, or
	$2\cdot M_1-M_2\equiv 4\cdot E_{\vec{1}_n}^1\bmod A$, then $E_{1,1}\in A$.
\end{lemma}
\begin{lemma} \label{lemmaeqq7}
	Let $p$ be a prime greater than or equal to $7$, 
	let $\ab{D}=\GF{p}$,
	let $n,m\in\N$, let $b=p-2$, let 
$T=\{\vec{0}_{n\times m}\}\cup \{E_{i,1}\mid i\in \finset{n}\}\cup \{E_{b\cdot \vec{1}_n}^1\}$,
	and let $A$ be a $\ab{M}_n(\ab{D})$-submodule 
	of $D^{(n\times m)}$.
If there exist pairwise distinct $M_1, M_2, M_3\in T\setminus\{E_{b\cdot \vec{1}_n}^1\}$
	with $M_1-M_2+M_3 \equiv E_{b\cdot \vec{1}_n}^1\bmod A$, or
	$M_1+M_2\equiv 2\cdot E_{b\cdot \vec{1}_n}^1\bmod A$, or
	$2\cdot M_1-M_2\equiv E_{b\cdot \vec{1}_n}^1\bmod A$, then $E_{1,1}\in A$.
\end{lemma}
\begin{lemma} \label{lemmaeqqp}
	Let $p$ be a prime, let $\alpha\in\N\setminus\{1\}$,
	let $\ab{D}=\GF{p^\alpha}$,
	let $n,m\in\N$, let $b\notin\{0, \dots ,p-1\}$ let 
	$T=\{\vec{0}_{n\times m}\}\cup \{E_{i,1}\mid i\in 
	\finset{n}\}\cup \{E_{b\cdot \vec{1}_n}^1\}$,
	and let $A$ be a $\ab{M}_n(\ab{D})$-submodule 
	of $D^{(n\times m)}$. %
If there exist pairwise distinct $M_1, M_2, M_3\in T\setminus\{E_{b\cdot \vec{1}_n}^1\}$
	with $M_1-M_2+M_3 \equiv E_{b\cdot \vec{1}_n}^1\bmod A$, then 
	$E_{\vec{1}_n}^1\in A$.
	Furthermore, if $p\neq 2$, and there exist distinct
	$M_1, M_2\in T\setminus\{E_{b\cdot \vec{1}_n}^1\}$ with
	$M_1+M_2\equiv 2\cdot E_{b\cdot \vec{1}_n}^1\bmod A$, or
	$2\cdot M_1-M_2\equiv E_{b\cdot \vec{1}_n}^1\bmod A$, then $E_{1,1}\in A$.
\end{lemma}
\begin{proof}[Proof of Lemmata~\ref{lemmaeqq2}-\ref{lemmaeqqp}]
We describe
the proof scheme of the above reported five lemmata.
To this end, we observe that in each case the set 
$T$ consists of $n+2$ elements and 
that the set
$T':=\{\vec{0}_{n\times m}\}\cup \{E_{i,1}\mid i\in \finset{n}\}$ 
is always a proper subset of $T$ of cardinality $n+1$. 
Let us denote by $t_{n+2}$ the unique element of $T\setminus T'$. 

We suppose that there exist 
pairwise distinct $M_1, M_2,M_3\in T'$ such that 
$M_1-M_2+M_3-t_{n+2}\in A$. 
Then there exists $\vec{g}\in D^n\setminus\{\vec{0}_n\}$ 
such that $M_1-M_2+M_3-t_{n+2}=E_{\vec{g}}^1$.
The proof of the fact that $\vec{g}$ is not 
zero depends on the assumptions of each lemma, 
and it is in each case elementary.
Hence Lemma~\ref{lemmasm} implies that 
$E_{1,1}^1\in A$.
	
For proving the second part of 
Lemmata~\ref{lemmaeqq3}-\ref{lemmaeqqp},
let $|\ab{D}|=q^{\beta}$ for a prime number $q$ 
and $\beta\in \N$.
Then we have to distinguish two cases: 
If $q=2$, then trivially $2\cdot t_{n+2} \equiv 0 \bmod A$.
Let us now consider the case that $2<q$. 
To this end, let us observe that
the claim of each lemma can be phrased as follows:
There exist distinct $M_1, M_2\in T'$ with $M_1+M_2\equiv 2\cdot t_{n+2}\bmod A$
or $2\cdot M_1-M_2\equiv t_{n+2}\bmod A$.
Then we can argue as above and prove that 
there exists $\vec{f}\in D^n\setminus \{\vec{0}_n\}$ such that
$E_{\vec{f}}^1=M_1+M_2-2\cdot t_{n+2}\in A$, 
or there exists  $\vec{h}\in D^n\setminus \{\vec{0}_n\}$ such that
$E_{\vec{h}}^1=2\cdot M_1-M_2- t_{n+2}\in A$,
and then infer from Lemma~\ref{lemmasm} that
$E_{1,1}^1\in A$.  
\end{proof}
\begin{lemma} \label{lemmaeqtwo}
Let $q\in\{2,3\}$, let $n=2$, let $\ab{D}=\GF{q}$, $m \in \N\setminus\{1\}$, 
let $T= \{\vec{0}_{2\times m}, E_{1,1}, E_{2,2}, E_{\vec{1}_2}^1+E_{\vec{1}_2}^2\}$, 
and let $T':=\{E_{1,1}, E_{2,2},E_{\vec{1}_2}^1+E_{\vec{1}_2}^2\}$.
Let $A$ be a $\ab{M}_2(\ab{D})$-submodule of $\ab{D}^{(2 \times m)}$, and let us assume 
that $A\cap T'=\emptyset$. 
Then we have
\begin{enumerate}
\item $a\cdot E_{1,1} \not\equiv b \cdot E_{2,2} \bmod A$ 
for all $(a,b)\in D^2\setminus\{(0,0)\}$.\label{item_item1oflemmaeqtwo}
\item $a \cdot E_{1,1}+b \cdot E_{2,2} \not\equiv E_{\vec{1}_2}^1+E_{\vec{1}_2}^2 \bmod A$ 
for all $(a,b) \in D^2$.\label{item_item2oflemmaeqtwo}
\end{enumerate}
\end{lemma}
\begin{proof}
We first prove \eqref{item_item1oflemmaeqtwo}.
Seeking a contradiction, let $a,b\in D$ not both zero, and let us
assume that $a\cdot E_{1,1}\equiv b\cdot E_{2,2}\bmod A$.
	Hence the matrix $a\cdot E_{1,1}-b\cdot E_{2,2}\in A$.
	If $a\neq 0$, then 
	$E_{1,1}=a^{-1}\cdot M_{1,1}\cdot (a\cdot E_{1,1}-b\cdot E_{2,2})\in A$
	in contradiction with the 
	assumptions. 
	If $b\neq 0$, then 
	$E_{2,2}=(-b)^{-1}\cdot M_{2,2}\cdot (a\cdot E_{1,1}-b\cdot E_{2,2})\in A$
	in contradiction with the 
	assumptions.

Next, we show \eqref{item_item2oflemmaeqtwo}.
	Seeking a contradiction, let $a,b\in D$ 
	and let us suppose that 
	$a\cdot E_{1,1}+b\cdot E_{2,2}\equiv E_{\vec{1}_2}^1+E_{\vec{1}_2}^2 \bmod A$.
	Hence we have that
	\[E_{(a-1, -1)}^1 +E_{(-1, b-1)}^2=
	a\cdot E_{1,1}+b\cdot E_{2,2}-E_{\vec{1}_2}^1-E_{\vec{1}_2}^2\in A.\]
	It is easy to verify that this implies that
$E_{\vec{1}_2}^1+E_{\vec{1}_2}^2\in A$, which is a contradiction to the assumptions.
\end{proof}
\begin{lemma} \label{lemmaeqthree}
	Let $q\in\{2,3\}$, let $\ab{D}=\GF{q}$, $m \in \N\setminus\{1\}$, 
	and let 
	\[T= \{\vec{0}_{3\times m}, E_{1,1}, E_{2,1}, E_{3,2},E_{\vec{1}_3}^1+E_{\vec{1}_3}^2\}.\]
	Let $T'=T\setminus \{\vec{0}_{3\times m}\}$ and let 
	$A$ be a $\ab{M}_3(\ab{D})$-submodule of $\ab{D}^{(3 \times m)}$ such 
	that $T'\cap A=\emptyset$. 
	Then we have
	\begin{enumerate}
		\item \label{lemmaeq3punto1} For all $M_1, M_2$ distinct elements of $T'$ and for all
		$(a,b)\in D^2\setminus\{(0,0)\}$ we have $a\cdot M_1\not \equiv b\cdot M_2\bmod A$;
		\item \label{lemmaeq3punto2} for all $a,b,c \in D$, we have $a\cdot E_{1,1}+b\cdot E_{2,1}+ c\cdot E_{2,3}
		\not \equiv E_{\vec{1}_3}^1+E_{\vec{1}_3}^2\bmod A$.
	\end{enumerate}
\end{lemma}
\begin{proof}
We first prove \eqref{lemmaeq3punto1}.
Seeking a contradiction, let $a,b\in D$ not both zero, and let us
assume that there exist distinct pairs $(i,j),(k,l)\in \{(1,1),(2,1),(3,2)\}$
such that $a\cdot E_{i,j}+b\cdot E_{k,l}\in A$. 
If $a\neq 0$, then 
$E_{i,j}=a^{-1}\cdot M_{i,j}\cdot (a\cdot E_{i,j}+b\cdot E_{k,l})\in A$
in contradiction with the 
assumption that $T'\cap A=\emptyset$. 
If $b\neq 0$, then 
$E_{k,l}=(-b)^{-1}\cdot M_{k,l}\cdot (a\cdot E_{i,j}+b\cdot E_{k,l})\in A$
in contradiction with the 
assumption that $T'\cap A=\emptyset$.

Next, we show \eqref{lemmaeq3punto2}.
Seeking a contradiction, let $a,b,c\in D$ 
and let us suppose that 
$a\cdot E_{1,1}+b\cdot E_{1,2}+c\cdot E_{2,3}\equiv E_{\vec{1}_3}^1+E_{\vec{1}_3}^2 \bmod A$.
Hence we have that
\[E_{(a-1, b-1, -1)}^1 +E_{(-1, -1, c-1)}^2=
a\cdot E_{1,1}+b\cdot E_{1,2}+c\cdot E_{2,3}- E_{\vec{1}_3}^1-E_{\vec{1}_3}^2\in A.\]
It is easy to verify that this implies that
$E_{\vec{1}_3}^1+E_{\vec{1}_3}^2\in A$, which is a contradiction to the assumption 
that $T'\cap A=\emptyset$.
\end{proof}
\begin{lemma} \label{lemmar}
	Let $q\in \{3,5\}$, $\ab{M}= \Z ^{(1 \times m)}_{q}$,
	$\ab{R}=\Z_{q}$, let $r_{1},r_{2} \in R$,
	and let $T \subseteq M$ with $\vec{0}_m \in T$. 
	Let  $a,b \in T\setminus\{\vec{0}_m\}$,
	and let $c:T \to M$ be a type-preserving function that satisfies 
	$c(\vec{0}_m)=\vec{0}_m$, $c(a)=r_{1}\cdot a$ and $c(b)=r_{2}\cdot b$.
	Then $r_1=r_2$. 
\end{lemma}
\begin{proof}
	We split the proof into several cases.\\
	\textbf{Case 1}: \emph{$a=b$}: This case is obvious.\\
\textbf{Case 2}: \emph{$a \in \submodule{b}$ and $a \not= b$ and $q=3$}:
Let $A$ be the trivial submodule $\{\vec{0}_m\}$
and $B$ be the submodule generated by $b$, that is $B=\{\vec{0}_m,b,2b\}$.
The case assumption implies that
$b=2\cdot a$ which is equivalent to $(a,\vec{0}_m,a,b) \in \rho(A,B)$.
Since $c$ preserves types, $c$ preserves $\rho(A,B)$ and therefore,
\[(r_1\cdot a,\vec{0}_m,r_1\cdot a,r_2\cdot b)=
(c(a),c(\vec{0}_m),c(a),c(b)) \in \rho(A,B).\]
Thus, $c(b) = 2\cdot c(a)$. 
Hence $r_{2}\cdot b = 2 r_{1}\cdot a$,
and since $b=2\cdot a$, we have that $ 2 r_{2}\cdot a = 2 r_{1}\cdot a $,
and therefore, $ r_{1}=r_{2}$.\\
	\textbf{Case 3}: \emph{$a \in \submodule{b}$ and $a \not= b$ and $q=5$}:
	Let $A=\{ \vec{0}_m\}$ and $B= \{\vec{0}_m,b,2\cdot b,3\cdot b,4\cdot b \}$.
	We split the proof into three subcases:\\
	\textbf{Case 3.1}: \emph{$b=2 a$}: Thus,
	$(a,\vec{0}_m,a,b) \in \rho(A,B)$, and since $c$ preserves types 
	\begin{equation*}
		(c(a),\vec{0}_m,c(a),c(b)) \in \rho(A,B) 
		\Longleftrightarrow c(b)=2c(a).
	\end{equation*}
	\textbf{Case 3.2}: \emph{$b=3a$}: Thus,
	$(b,a,b,\vec{0}_m) \in \rho(A,B)$, and since $c$ preserves types 
	\begin{equation*}
		(c(b),c(a),c(b),\vec{0}_m)\in \rho(A,B) 
		\Longleftrightarrow c(b)=3c(a).
	\end{equation*}
	\textbf{Case 3.3}: \emph{$b=4a$}: Thus, 
	$(\vec{0}_m,a,\vec{0}_m,b)\in \rho(A,B)$, and since $c$ preserves types 
	\begin{equation*}
		(\vec{0}_m,c(a),\vec{0}_m,c(b)) \in \rho(A,B) \Longleftrightarrow c(b)=4c(a).
	\end{equation*}
Thus, there exists $t\in \{2,3,4 \}$ such that $c(b)=tc(a)$ and $b=ta$.
Hence $r_{2}b = t r_{1}a$, and therefore, 
$r_{2} t a = t r_{1}a$, which
implies $r_{1}=r_{2}$.\\
\textbf{Case 4}: \emph{$a \notin\submodule{b}$}:
The case assumption implies that $a$ and $b$
are linearly independent. 
Therefore, 
\[
r_1 a- r_2 b\in \{ s(a-b) \mid s\in \GF{q}\}.
\]
Hence there exists $s\in \GF{q}$ such that
$r_1a- r_2 b=sa -sb$. 
Thus, since $a$ and $b$ are linearly independent 
we have that $r_1=s=r_2$. 
\end{proof}
\section{Strictly polynomially rich modules over simple rings}\label{sec:modules}
In this section we prove the following 
theorem:
\begin{theorem}[{\cite[Theorem~3.31]{Kap05}}]\label{teor:result_Kapl_thesis}
	Let $m,n\in\N$, let $\ab{D}$ be a finite field, and let 
	$\ab{M}_n(\ab{D})$ be the ring of $n\times n$-matrices 
	over $\ab{D}$, let $\ab{D}^{(n\times m)}$ denote the
	$\ab{M}_n(\ab{D})$-module
	of all $n\times m$-matrices with entries from $\ab{D}$.
	Let $k\in\N$ and $q=\card{D}$. 
	Then $\ab{D}^{(n\times m)}$ is strictly $k$-polynomially rich
	if and only if 
	\begin{enumerate}
		\item $m\in\N$, $q\in\{2,3,5\}$, $n=1$, and $k=1$; or\label{item:kapl_m23511}
		\item $m=1$, $q=2$, $n\in\{2,3\}$, and $k=1$; or\label{item:kapl_12231}
		\item $m=1$, $q=3$, $n=2$, and $k=1$; or\label{item:kapl_1321}
		\item $m=1$, $q=2$, $n=1$, and $k\in\{2,3\}$, or\label{item:kapl_12123}
		\item $m=1$, $q=3$, $n=1$, and $k=2$. \label{item:kapl_1312}
	\end{enumerate}
\end{theorem}
  Stated differently, Theorem~\ref{teor:result_Kapl_thesis} tells
  that a nontrivial module over a finite simple ring is
  never strictly $4$-polynomially rich; $\ringmod{\Z_2}{\Z_2}$ is strictly $3$-polynomially
  rich, $\ringmod{\Z_3}{\Z_3}$ is $2$-polynomially rich, but not $3$-polynomially rich,
  and the following is a complete list of finite modules over finite
  simple rings that are strictly $1$-polynomially rich, but not
  strictly $2$-polynomially rich: $\ringmod{\Z_2}{(\Z_2^m)}$ and
  $\ringmod{\Z_3}{(\Z_3^m)}$ for $m \ge 2$, $\ringmod{\Z_5}{(\Z_5^m)}$
  for $m \ge 1$,   and the three modules $\ringmod{\ab{M}_2 (\Z_2)}{(\Z_2^2)}$,
  $\ringmod{\ab{M}_3 (\Z_2)}{(\Z_2^3)}$ and $\ringmod{\ab{M}_2 (\Z_3)}{(\Z_3^2)}$.
\begin{proof}
We first prove the ``if''-direction.
We split the proof into four cases.
	
\textbf{Case 1}: \emph{$m \geq 1, n=1, q=2, k=1$}: 
In this case $\ab{D}^{(n\times m)}$ is strictly $1$-affine complete 
by~\cite[Theorem~1.3]{AI:PIIE}.
Therefore, $\ab{D}^{(n\times m)}$ is also strictly $1$-polynomially rich.
	
\textbf{Case 2}: \emph{$m\geq 1,n=1,q\in\{3,5\},k=1$}:
Let $T \subseteq D^{(n\times m)}$ and $f:T\rightarrow D^{(n\times m)}$ be a 
type-preserving function.
We assume $\vec{0}_{1\times m}\in T$ and that
$f$ is zero-preserving. 
It is well known (cf., e.g., \cite[page~233]{KP:PCIA}) 
that every unary zero preserving congruence-preserving function $c$
on the $\ab{M}_n(\ab{D}) $-module $\ab{D}^{(n\times m)}$ is of the form
$m\mapsto r(m)\cdot m$ with $r(m) \in D^{(n\times n)} $. 
Lemma~\ref{lemmar} implies that if $f$ is 
also type-preserving, then $r(m)$ is constant, and hence 
$f$ can be interpolated on its domain by a polynomial.
Thus, Lemma~\ref{lemma:reducing_to_zero_preserving}
implies that $\ab{D}^{(n\times m)}$ is strictly $1$-polynomially rich.
	
\textbf{Case 3}: \emph{$(m=1,q=2, n\in\{2,3\},k=1)$ or $(m=1,q=3,n=2,k=1)$}:
Let $T \subseteq D^{(n\times 1)}$, and 
$c:T \rightarrow D^{(n\times 1)}$ a type-preserving function.
The case assumptions implies 
that $\ab{D}^{(n\times m)}=\ab{D}^{(n\times 1)}$ is simple,
hence setting 
$A=\{ \vec{0}_{n\times 1}\}$ and 
$B=D^{(n\times 1)}$, we have that $A \prec B$. 
We assume $\vec{0}_{n\times 1} \in T$ 
and $c(\vec{0}_{n\times 1})=\vec{0}_{n\times 1}$.
Let $S$ be the basis of $\sottospazio{}{T}$ with $S\subseteq T$ constructed in 
Lemma~\ref{lemma:technical_lemma_kapl}. Then 
\begin{equation}\label{eq:conseguenzaLemmatennicoKaplincase3}
\forall t \in T \,\exists x,y,z \in S \cup \{\vec{0}_{n\times 1} \}\colon t=x-y+z.
\end{equation}
Since each zero-preserving unary polynomial 
function of $\ab{D}^{(n\times m)}$ is of the form
$p(x)=R \cdot x \text{ for }R\in D^{(n\times n)} $,
we have to show there exists $R \in D^{(n\times n)}$ 
such that $c(x)=R\cdot x$ for all $x \in T$.
By solving a linear system of equations 
one can compute $R\in D^{(n\times n)}$ such that 
$c(x)=R\cdot x$ for all $x\in S\cup 
\{\vec{0}_{n\times 1}\}$.
Note that the system has a solution since we are assuming that 
$S$ consists of linearly independent vectors.
Next we show that $R\cdot x$ interpolates $c$ on $T$. 
To this end, let us fix $t\in T$.
By \eqref{eq:conseguenzaLemmatennicoKaplincase3}
there exist $x,y,z \in S \cup \{\vec{0}_{n\times 1}\}$ 
such that
$(x,y,z,t) \in \rho(A,B)$.
Since $c$ is type-preserving, we have
$(c(x),c(y),c(z),c(t))\in \rho(A,B)$,
which is equivalent to
$c(t)=c(x)-c(y)+c(z)$.
Hence $c(t)=c(x)-c(y)+c(z)=Rx-Ry+Rz=R(x-y+z)=Rt.$ 
Thus, 
by Lemma~\ref{lemma:reducing_to_zero_preserving},
$\ab{D}^{(n\times m)}$ is strictly $1$-polynomially rich.

\textbf{Case 4}: \emph{$(m=1,q=2,n=1, k\in\{1,2,3\})$ 
or $(m=1,q=3,n=1,k=2)$}:	
In this case $\ab{D}^{(n\times m)}=\ab{D}$ is simple 
and $A:=\{ 0 \}\prec B:=D^{(n\times m)}=D$. 
Let $T \subseteq D^{k}$, and let 
$c:T \rightarrow D$ be a type-preserving function.
We assume $\vec{0}_k \in T$ and $c(\vec{0}_k)=0$. 
Let $S$ be the basis of $\sottospazio{}{T}$ with $S \subseteq T$
constructed in  
Lemma~\ref{lemma:technical_lemma_kapl}. Then we have
\begin{equation}\label{eq:equzione_che_vien_dal_lemma_tennico_kapl}
		\forall t \in T\, \exists x,y,z \in S \cup \{\vec{0}_k\}\colon
		t=x-y+z.
	\end{equation}
	Next, we show that there are $r_{1},\ldots,r_{k} \in D $ such that
	\[c(x_{1},\ldots,x_{k})=r_{1}x_{1}+\cdots +r_{k}x_{k} \text{  for
		all } (x_{1},\ldots,x_{k})\in T.\]
	Solving a linear system, one computes $r_{1},\ldots,r_{k}\in D $
	such that
	\begin{equation*}
		c(x_{1},\ldots,x_{k})=r_{1}x_{1}+\cdots+r_{k}x_{k}, \,
		\text{ for all } (x_1, \dots, x_k) \in S \cup \{ \vec{0}_k\}.
	\end{equation*} 
	Note that the system has a solution because the elements of $S$ are 
	linearly independent. 
	Next, we prove that the map $p$ defined by 
	$(x_1, \dots, x_k)\mapsto r_{1}x_{1}+\cdots+r_{k}x_{k}$
	for all $x_1, \dots, x_k\in T$ interpolates $c$ on $T$. 
	To this end, let $t\in T$.
	By \eqref{eq:equzione_che_vien_dal_lemma_tennico_kapl}
	there are $(x_1, \dots x_k), (y_1, \dots, y_k),
	(z_1, \dots, z_k)\in S\cup\{\vec{0}_k\}$
	such that for each $i\in\finset{k}$ 
	we have 
	$(x_i,y_i, z_i, t_i)\in \rho(A,B)$. 
	Since $c$ is type-preserving, we have
	\begin{equation*}
		(c(x_{1},\ldots,x_{k}),c(y_{1},\ldots,y_{k}),c(z_{1},\ldots,z_{k}),
		c(t_{1},\ldots,t_{k})) \in \rho(A,B),
	\end{equation*}
	which is equivalent to
	\begin{equation*}
		c(t_{1},\ldots,t_{k})=c(x_{1},\ldots,x_{k})-c(y_{1},\ldots,y_{k})+
		c(z_{1},\ldots,z_{k}).
	\end{equation*}
	Hence 
	\[
	\begin{split}
		c(t_{1},\ldots,t_{k})=&c(x_{1},\ldots,x_{k})-
		c(y_{1},\ldots,y_{k})+c(z_{1},\ldots,z_{k})=\\
		&r_{1}x_{1}+\cdots+r_{k}x_{k}-
		r_{1}y_{1}-\cdots-r_{k}y_{k}+r_{1}z_{1}+\cdots+r_{k}z_{k}=\\
		&r_{1}(x_{1}-y_{1}+z_{1})+\cdots+r_{k}(x_{1}-y_{k}+z_{k})=\\
		&r_{1}t_{1}+\cdots+
		r_{k}t_{k}.
	\end{split}
	\]
	Therefore $c$ is the restriction of a $k$-ary
	polynomial function and 
	by Lemma~\ref{lemma:reducing_to_zero_preserving},
	$\ab{D}^{(n\times m)}$ is strictly $k$-polynomially
	rich.
	
	Next, we prove the ``only if''-direction. We split the
	proof into eight cases. First, 
	we need to prove that for each $q\notin \{2,3,5\}$ and for all $n,m\in \N$, 
	the $\mathbf{M}_n(\GF{q})$-module $\GF{q}^{n\times m}$ is 
	not strictly $1$-polynomially rich, and hence not even strictly $k$-polynomially 
	rich for any $k\in \N$. This is shown in Case 1.
 Let $q\in \{2,3,5\}$. 
Since there are some combinations of $k,n,m$ for which 
the $\mathbf{M}_n(\GF{q})$-module
$\GF{q}^{n\times m}$ is strictly $k$-polynomially rich, the case distinction must be finer. 
	In order to help the reader understand that these eight cases suffice, we summarise 
	the combinations for which  the $\mathbf{M}_n(\GF{q})$-module
	$\GF{q}^{n\times m}$ is not strictly $k$-polynomially rich 
	in Tables \ref{tabella:quguale5}, \ref{tabella:quguale3}, and \ref{tabella:quguale2}.
	\begin{table}
	\begin{tabular}{C{15mm}C{25mm}C{25mm}}
		\toprule
		 & $n=1$ & $n\geq 2$ 
		\tabularnewline
		\midrule
		$m\geq 1$ &$k\geq 2$, Case 8 &  $k\geq 1$, Case 1
		\tabularnewline
		\bottomrule
	\end{tabular}
	\caption{The combinations of $k,n,m$ for which
	the $\ab{M}_n(\GF{5})$-module $\GF{5}^{n\times m}$ is NOT strictly 
	$k$-polynomially rich.}\label{tabella:quguale5}
\end{table}
\begin{table}
	\begin{tabular}{C{15mm}C{25mm}C{25mm}C{25mm}}
		\toprule
		 & $n=1$ & $n=2$ & $n\geq 3$ 
		\tabularnewline
		\midrule
		$m=1$ &$k\geq 3$, Case 7 & $k\geq 2$, Case 4 & $k\geq 1$, Case 1
		\tabularnewline
		$m\geq 2$ &$k\geq 2$, Case 5 & $k\geq 1$, Case 2 & $k\geq 1$, Case 1
		\tabularnewline
		\bottomrule
	\end{tabular}
	\caption{The combinations of $k,n,m$ for which
	the $\ab{M}_n(\GF{3})$-module
	$\GF{3}^{n\times m}$ 
	is NOT strictly $k$-polynomially rich.}\label{tabella:quguale3}
\end{table}
\begin{table}
	\begin{tabular}{C{15mm}C{25mm}C{25mm}C{25mm}C{25mm}}
		\toprule
		 & $n=1$ & $n=2$ & $n=3$ & $n\geq 4$ 
		\tabularnewline
		\midrule
		$m=1$ &$k\geq 4$, Case 6 & $k\geq 2$, Case 4 & $k\geq 2$, Case 4 & $k\geq 1$, Case 1
		\tabularnewline
		$m\geq 2$ &$k\geq 2$, Case 5 & $k\geq 1$, Case 2 & $k\geq 1$, Case 3 & $k\geq 1$, Case 1
		\tabularnewline
		\bottomrule
	\end{tabular}
	\caption{The combinations of $k,n,m$ for which
	the $\ab{M}_n(\GF{2})$-module $\GF{2}^{n\times m}$ is NOT strictly 
	$k$-polynomially rich.}\label{tabella:quguale2}
\end{table}
	
	\textbf{Case 1}: \emph{$(m \geq 1,k=1,n \geq 4, q = 2)$ or 
		$(m \geq 1,k=1,n \geq 3, q = 3)$ or $(m \geq 1, k=1, n\geq 2, q = 5)$
		or $(m \geq 1,k=1, n\geq 1, q\notin\{2,3,5\}$)}:
	Clearly, there are a prime $p$ and a natural number $\alpha$
	such that $q=p^\alpha$. We show that $\ab{D}^{(n\times m)}$ is 
	not strictly $1$-polynomially rich. To this end,
	we define a partial function $c$,
	depending on the different subcases. Let $b \in D$ with
	\begin{align}
		(m \geq 1,k=1,n \geq 4, q = 2)&\Rightarrow b=1\label{subcase:b=1}\\
		(m \geq 1,k=1,n \geq 3,q=3)&\Rightarrow b=2\label{subcase:b=2}\\
		(m \geq 1,k=1,n \geq 2, q=5)&\Rightarrow b=4\label{subcase:b=4}\\
		(m \geq 1,k=1,n \geq 1,p >5,\alpha=1)&\Rightarrow b=p-2\label{subcase:b=p-2}\\
		(m \geq 1,k=1,n \geq 1,\alpha > 1)&\Rightarrow b \notin \{0,1,\ldots,p-1\}.\label{subcase:bless_p}
	\end{align}
Let $T=\{\vec{0}_{n\times m}\}\cup \{E_{i,1}\mid i\in\finset{n}\}\cup
\{E_{b\cdot \vec{1}_n}^1\}$, and let $c\colon T\to \{\vec{0}_{n\times m},
E_{b\cdot \vec{1}_n}^1\}$
be the identity on $\{\vec{0}_{n\times m},
E_{b\cdot \vec{1}_n}^1\}$
and the constant function with value 
$\vec{0}_{n\times m}$ on
$\{ E_{i,1}\mid i\in\finset{n}\}$.
Lemma~\ref{lemmasm} implies that $c$ is a congruence
preserving function.
Next, let $A,B$ be two submodules of 
$\ab{D}^{(n \times m)}$ with $A \prec B$;
we show that 
	for all $x_{1},x_{2},x_{3},x_{4} \in T$,
	\begin{equation} \label{eqrho}
		(x_{1},x_{2},x_{3},x_{4})\in \rho(A,B) \Rightarrow
		(c(x_{1}),c(x_{2}),c(x_{3}),c(x_{4})) \in \rho(A,B).
	\end{equation}
	Let $x_{1},x_{2},x_{3},x_{4} \in T$.
	If $x_{1}=x_{2}=x_{3}=x_{4}=\vec{0}_{n\times m}$, 
	or $E_{b\cdot \vec{1}_n}^1\notin \{x_{1},x_{2},x_{3},x_{4}\}$,
	or $x_{1}=x_{2}=x_{3}=x_{4}=E_{b\cdot \vec{1}_n}^1$,
	then $(c(x_1), \dots, c(x_4))$ is a diagonal tuple 
	and therefore, it belongs to $\rho(A,B)$.
	Let us now assume that at least one but not all of
	$x_{1},x_{2},x_{3},x_{4}$ is equal to $E_{b\cdot \vec{1}_n}^1$.
	If $E_{1,1}\in A$, then Lemma~\ref{lemmasm}
	implies that $T\subseteq A$ and therefore,
	\eqref{eqrho} follows from the fact that $c(T)\subseteq A$.
	Let us now assume that $E_{1,1}\notin A$.
	Lemma~\ref{lemmaeq} implies that for each 
	pair of distinct elements $x, y$ from $T$ we have 
	$ x\not\equiv y\bmod A$. We show that this cannot 
	occur under each of the assumptions 
	\eqref{subcase:b=1}-\eqref{subcase:bless_p}.
	Since $x_1-x_2+x_3-x_4\in A$, and $x_1, x_2, x_3, x_4\in T$,
	one of the following conditions must be satisfied according to how many 
	distinct elements are there in $\set{x_1, x_2, x_3, x_4}$ and in which 
	order they appear in $(x_1, x_2, x_3, x_4)$:
	\begin{align}
		&\exists \text{ distinct }M_1, M_2, M_3\in 
		T\setminus\{E_{b\cdot \vec{1}_n}^1\}\colon 
		M_1-M_2+M_3\equiv E_{b\cdot \vec{1}_n}^1 \bmod A\label{eq:ex1pagina13}\\ 
		&\exists \text{ distinct }M_1, M_2\in T\setminus\{E_{b\cdot \vec{1}_n}^1\}
		\colon 2\cdot M_1-M_2\equiv   E_{b\cdot \vec{1}_n}^1 \bmod A\label{eq_nuova_strana_col_2}\\
		&\exists \text{ distinct }M_1, M_2\in T\setminus\{E_{b\cdot \vec{1}_n}^1\}
		\colon M_1+M_2 \equiv 2 \cdot E_{b\cdot \vec{1}_n}^1 \bmod A\label{eq:ex2Bpagina13}\\
		&\exists \text{ distinct }M_1, M_2\in T\setminus\{E_{b\cdot \vec{1}_n}^1\}
		\colon M_1-M_2\equiv \vec{0}_{n\times m} \bmod A.\label{eq:ex3pagina13}\\
		&\exists M_1 \in T\setminus\{E_{b\cdot \vec{1}_n}^1\}
		\colon M_1\equiv E_{b\cdot \vec{1}_n}^1 \bmod A.\label{eq:singoletto}
	\end{align}
	In fact if $\card{\set{x_1, x_2, x_3, x_4}}=4$, 
	then we are in case~\eqref{eq:ex1pagina13}; 
	if $\card{\set{x_1, x_2, x_3, x_4}}=3$ and 
	$E_{b\cdot \vec{1}_n}^1$ appears once in $(x_1, x_2, x_3, x_4)$
	then we are in case~\eqref{eq_nuova_strana_col_2}; 
	if $E_{b\cdot \vec{1}_n}^1$ appears twice in $(x_1, x_2, x_3, x_4)$
	both times in an odd or even position, then we are in case~\eqref{eq:ex2Bpagina13};
	if $E_{b\cdot \vec{1}_n}^1$ appears twice in $(x_1, x_2, x_3, x_4)$ 
	once in an odd position and once in an even position, then 
	we are in case~\eqref{eq:ex3pagina13}; 
	and if $\card{\set{x_1, x_2, x_3, x_4}}=2$, then we are in case~\eqref{eq:singoletto}.
	
Conditions~\eqref{eq:ex3pagina13} and~\eqref{eq:singoletto} are clearly in 
contradiction with the already proved fact that
for each pair $(x,y)$ of distinct elements from $T$ 
we have $x\not\equiv y\mod A$. 
In order to derive a contradiction 
from \eqref{eq:ex1pagina13}-\eqref{eq:ex2Bpagina13} we distinguish 
cases according to which among the 
left-hand side of \eqref{subcase:b=1}-\eqref{subcase:bless_p}
is satisfied.
Under \eqref{subcase:b=1}, \eqref{eq:ex1pagina13}
contradicts Lemma~\ref{lemmaeqq2}, and since
$2\cdot E_{b\cdot\vec{1}_n}^1=\vec{0}_{n\times m}$,
\eqref{eq_nuova_strana_col_2}
and \eqref{eq:ex2Bpagina13} are equivalent to~\eqref{eq:ex3pagina13} 
and~\eqref{eq:singoletto}.
Under \eqref{subcase:b=2}, each of \eqref{eq:ex1pagina13}, 
\eqref{eq_nuova_strana_col_2} and \eqref{eq:ex2Bpagina13}
contradict Lemma~\ref{lemmaeqq3}.
Under \eqref{subcase:b=4}, each of \eqref{eq:ex1pagina13}, 
\eqref{eq_nuova_strana_col_2} and \eqref{eq:ex2Bpagina13}
contradict Lemma~\ref{lemmaeqq5}.
Under \eqref{subcase:b=p-2}, each of \eqref{eq:ex1pagina13}, 
\eqref{eq_nuova_strana_col_2} and \eqref{eq:ex2Bpagina13}
contradict Lemma~\ref{lemmaeqq7}.
Let us now consider the case that \eqref{subcase:bless_p} holds:
For $q=p^{\alpha}$ with $p$ prime and $p > 2$ and 
$\alpha > 1$, each of 
\eqref{eq:ex1pagina13},  
\eqref{eq_nuova_strana_col_2} and \eqref{eq:ex2Bpagina13}
contradicts Lemma~\ref{lemmaeqqp}.
For $q=2^{\alpha}$ with $\alpha > 1$,
\eqref{eq:ex1pagina13} contradicts Lemma~\ref{lemmaeqqp}.
Moreover, in this case
since $2\cdot E_{b\cdot\vec{1}_n}^1=\vec{0}_{n\times m}$,
\eqref{eq_nuova_strana_col_2} and \eqref{eq:ex2Bpagina13}
are equivalent to 
\eqref{eq:ex3pagina13} ans~\eqref{eq:singoletto} and yield the already discussed 
contradiction. 
	Thus, we can conclude that 
	$c$ preserves $\rho(A,B)$,
	and therefore, $c$ is type-preserving. 
	
	Seeking a contradiction, let us 
	suppose that there exists $R\in D^{(n\times n)} $ such that
	for all $x\in T$ we have $c(x)=R\cdot x$.
	Then since $c(E_{i,1})=\vec{0}_{n\times m}$ 
	for each $i\in\finset{n}$, we infer that 
	$R=\vec{0}$. Since $c$ is not the constant zero function
	we derive a contradiction with the
	fact that every unary zero-preserving polynomial of the
	$\ab{M}_n(\ab{D})$-module $\ab{D}^{(n\times m)}$
	is of the form $x\mapsto M\cdot x$ for some 
	$M\in \ab{M}_n(\ab{D})$.
	
	\textbf{Case 2}: \emph{$(m>1,q \in\{2,3 \},n=2,k=1)$}:
	Let $T=\{\vec{0}_{n\times m}, E_{1,1}, E_{2,2}, 
	E_{\vec{1}_n}^1+ E_{\vec{1}_n}^2\}$
	and let $c\colon T\to D^{(n\times m)}$ be the constant 
	function with constant value 
	$\vec{0}_{n\times m}$ on $T'=\{\vec{0}_{n\times m}, 
	E_{1,1}, E_{2,2}\}$
	and the identity on $\{E_{\vec{1}_n}^1+ E_{\vec{1}_n}^2\}$.
	Next, we show that $c$ is congruence preserving. 
	To this end, we observe that it suffices to check 
	for each $x\in T'$
	whether $c(E_{\vec{1}_n}^1+ E_{\vec{1}_n}^2)-c(x)$
	belongs to the $\ab{M}_n(\ab{D}) $-submodule generated by 
	$E_{\vec{1}_n}^1+ E_{\vec{1}_n}^2-x$.
	Since $c(E_{\vec{1}_n}^1+ E_{\vec{1}_n}^2)-c(x)=
	E_{\vec{1}_n}^1+ E_{\vec{1}_n}^2$
	for all $x\in T'$, we need only prove that
	for all $x\in T'$ we have that
	$E_{\vec{1}_n}^1+ E_{\vec{1}_n}^2$ belongs to
	the $\ab{M}_n(\ab{D}) $-submodule generated by 
	$E_{\vec{1}_n}^1+ E_{\vec{1}_n}^2-x$. 
	The case $x=\vec{0}_{n\times m}$ is trivial. 
	For the case $x=E_{1,1}$, we observe that 
	$M_{\vec{1}_n}^2\in D^{(n\times n)} $, 
	and the claim follows from
	\[E_{\vec{1}_n}^1+ E_{\vec{1}_n}^2=
	M_{\vec{1}_n}^2\cdot (E_{\vec{1}_n}^1+ E_{\vec{1}_n}^2-E_{1,1}).\]
	For the case $x=E_{2,2}$, we observe that 
	$M_{\vec{1}_n}^1\in D^{(n\times n)} $, and the claim follows from
	\[E_{\vec{1}_n}^1+ E_{\vec{1}_n}^2=
	M_{\vec{1}_n}^1\cdot (E_{\vec{1}_n}^1+ E_{\vec{1}_n}^2-E_{2,2}).\]
	Next, we show that $c$ preserves the relations 
	of the form $\rho(A,B)$ for each pair 
	of $\ab{M}_n(\ab{D}) $-submodules $A,B$ 
	of $\ab{D}^{(n\times m)}$ with $A\prec B$.
	Let $A,B$ be two $\ab{M}_n(\ab{D}) $-submodules
	of $\ab{D}^{(n\times m)}$ with $A\prec B$, and let 
	$x_{1},x_{2}, x_{3},x_{4} \in T$; we show that
	\begin{equation} \label{impl}
		(x_{1},x_{2},x_{3},x_{4}) \in \rho(A,B) \Rightarrow
		(c(x_{1}),c(x_{2}),c(x_{3}),c(x_{4}))\in \rho (A,B).
	\end{equation}
	Clearly, if $x_{1},x_{2},x_{3},x_{4} \in T'$, then
	$c$ is the constant $\vec{0}_{n\times m}$ and 
	\eqref{impl} is a consequence of the fact that
	\[(\vec{0}_{n\times m},\vec{0}_{n\times m},
	\vec{0}_{n\times m},\vec{0}_{n\times m})\in \rho(A,B).\]
	Next, we consider the case that at least one of 
	$x_{1},x_{2}, x_{3},x_{4}$ is equal to 
	$E_{\vec{1}_n}^1+ E_{\vec{1}_n}^2$.
	The subcase in 
	which $E_{\vec{1}_n}^1+ E_{\vec{1}_n}^2\in A$
	is trivial, since \eqref{impl} follows
	from the fact that
	$c(E_{\vec{1}_n}^1+ E_{\vec{1}_n}^2)=
	E_{\vec{1}_n}^1+ E_{\vec{1}_n}^2\in A$. 
	Let us now consider the
	subcase in which 
	$E_{\vec{1}_n}^1+ E_{\vec{1}_n}^2\notin A$.
	Lemma~\ref{lemmaeqtwo}\eqref{item_item2oflemmaeqtwo}
	implies that
	if at least one but not all of
	$x_{1},x_{2},x_{3},x_{4}$ are equal to $E_{\vec{1}_n}^1+ E_{\vec{1}_n}^2$, 
	then $(x_{1},x_{2},x_{3},x_{4})\notin \rho(A,B)$. 
	On the other hand, if all of $x_{1},x_{2},x_{3},x_{4}$ are equal to
	$E_{\vec{1}_n}^1+ E_{\vec{1}_n}^2$, 
	then \eqref{impl} follows form the fact that
	\[(E_{\vec{1}_n}^1+ E_{\vec{1}_n}^2, E_{\vec{1}_n}^1+ 
	E_{\vec{1}_n}^2, E_{\vec{1}_n}^1+ E_{\vec{1}_n}^2,E_{\vec{1}_n}^1+ 
	E_{\vec{1}_n}^2)\in \rho(A,B).\] 
	Therefore, $c$ is type-preserving.
	Seeking a contradiction, 
	let us suppose that $c$ is the 
	restriction of a polynomial function.
	Since every unary zero-preserving 
	polynomial function is of the form
	$p(x)=R \cdot x $ for some $R\in D^{(n\times n)} $,
	there exists $M\in D^{(n\times n)} $ such that
	$M\cdot E_{1,1}=\vec{0}_{n\times m}= M\cdot E_{2,2}$.
	Thus $M=\vec{0}$ and since $c$ is not constant,
	we get a contradiction with the assumption that
	$x\mapsto M\cdot x$ interpolates $c$ on $T$.  
	Therefore, $\ab{D}^{(n\times m)}$ is not strictly $1$-polynomially rich.
	
	\textbf{Case 3}: \emph{$(m>1,q=2,n=3,k=1)$}:
	The proof is similar to the one of 
	Case 2, only setting
	$T= \{\vec{0}_{3\times m}, E_{1,1}, E_{2,1}, E_{3,2},E_{\vec{1}_3}^1+E_{\vec{1}_3}^2\}$
	and using
	Lemma~\ref{lemmaeqthree} instead of Lemma~\ref{lemmaeqtwo}.
	
	\textbf{Case 4}: \emph{$(m = 1,q=2,2\leq n \leq 3,k>1)$ or $(m = 1,q=3,n=2,k>1)$}:
	Let 
	\[
	\begin{split}
		T''=\{\matrixvec{x}\in (D^{n})^k\mid 
		&\exists j\in\finset{k},\, i\in\finset{n}\,
		\forall l\in\finset{k}\setminus\{j\}\colon\\
		&\matrixvec{x}(j)=\vec{e}_i^n, \, 
		\matrixvec{x}(l)=\vec{0}_n\},
	\end{split}
	\]
	let $\matrixvec{0}_n$ and 
	$\matrixvec{1}_n$ be the
	elements of $(\ab{D}^n)^k$
	whose entries are all $\vec{0}_n$ 
	and $\vec{1}_n$ respectively,
	let $T=T''\cup\{\matrixvec{0}_n, \matrixvec{1}_n\}$,
	and let $c\colon T\to D^{(n\times m)}$ be the 
	constant function with constant
	values $\vec{0}_n$ on $T'=T\setminus\{\matrixvec{1}_n\}$
	and be the first projection on $\{\matrixvec{1}_n\}$. 
	Lemma~\ref{lemmasm} implies that 
	$c$ is congruence preserving. 
	Next, we show that 
	$c$ preserves $\rho(A,B)$ for
	$A=\{ \vec{0}_{n}\}$ and $B=D^{(n\times 1)}$. 
	Since $\ab{D}^{(n\times 1)}$ is simple this implies that 
	$c$ is type-preserving.  
	Let $\matrixvec{x}, \matrixvec{y}, \matrixvec{z},
	\matrixvec{t}\in T$.
	If none of $\matrixvec{x}, \matrixvec{y}, \matrixvec{z},
	\matrixvec{t}$ are equal to $\matrixvec{1}_n$
	then 
	$(c(\matrixvec{x}),c(\matrixvec{y}),c(\matrixvec{z}),c(\matrixvec{t}))=
	(\vec{0}_n,\vec{0}_n, \vec{0}_n,\vec{0}_n)\in \rho(A,B)$.
	If at least one but not all of $\matrixvec{x}, \matrixvec{y}, 
	\matrixvec{z},\matrixvec{t}$ are equal to
	$\matrixvec{1}_n$, then 
	$\matrixvec{t}\neq \matrixvec{x}-\matrixvec{y}+\matrixvec{z}$. 
	Thus $c$ preserves $\rho(A,B)$.
	Seeking a contradiction,
	let us suppose that $c$ can be interpolated by
	a polynomial function.
	Since every $k$-ary zero-preserving 
	polynomial function $p$ of $\ab{D}^{(n\times m)}$ is of the form
	$p(x_{1},\dots , x_k)=R_{1}x_{1}+\cdots +R_{k}x_{k}$ for some 
	$R_{1},\dots, R_{k} \in D^{(n\times n)}$, since 
	$c(T')=\vec{0}_n$, we have
	$R_{j}=\vec{0}$ for all $j\in\finset{k}$. 
	This contradicts the fact that $c$ is not constant. 
	Therefore, $c$ is not the restriction of a polynomial
	of $\ab{D}^{(n\times m)}$, and $\ab{D}^{(n\times m)}$ 
	is not strictly $k$-polynomially rich.
	
	\textbf{Case 5}: \emph{$(m>1,q \in \{2,3\},n=1,k=2)$}:
	Let $t_{0}=\vec{0}_m^T$, let
	$t_{1}=(\vec{e}_m^1)^T$, and let 
	$t_{2}=(\vec{e}_m^2)^T$. 
	Let $c\colon T=\{(t_{0},t_{0}),(t_{1},t_{0}),(t_{2},t_{1})  \}\to D^{(n\times m)}$ 
	be defined by
	\[
	c(t_{0},t_{0})=c(t_{1},t_{0})=t_{0}\quad c(t_{2},t_{1})=t_{2}.
	\]
	First, we show that $c$ is a congruence preserving function.	
	We have that 
	$c(t_2,t_1)-c(t_0, t_0)=t_2\in \submodule{t_2, t_1}$.
	Moreover, 
	$c(t_2, t_1)-c(t_1, t_0)=t_2\in\submodule{t_2-t_1, t_1}$. 
	Thus, $c$ is congruence preserving (note that all the other 
	cases are trivial). 
	Next, we show that $c$ preserves $\rho(A,B)$ for all submodules 
	$A,B$ of $\ab{D}^{(n\times m)}$ with $A \prec B$.
	Let $A,B$ submodules of $\ab{D}^{(n\times m)}$ with $A
	\prec B$. If $t_{2} \in A$, then $c(t_{2},t_{1}) \equiv t_{0} \bmod A$;
	if $t_{2} \notin A$, then $a(t_{i},t_{0}) \not\equiv (t_{2},t_{1}) \bmod A$
	for $i \in \{0,1\}$ and for all $a \in \{1,2 \}$.
	If $t_{2} \in A$, it is obvious that $c$ preserves $\rho(A,B)$. If $t_{2}
	\notin A$ we have to show that for all $x=(x_{1},x_{2}),y=(y_{1},y_{2}),
	z=(z_{1},z_{2}),u=(u_{1},u_{2}) \in T$ we have
	\begin{equation} \label{eql2}
		\begin{split}
			(x_{1},y_{1},z_{1},u_{1}) \in \rho(A,B) \text{ and }
			(x_{2},y_{2},z_{2},u_{2}) \in \rho(A,B)
			\Rightarrow \\
			(c(x_{1},x_{2}),c(y_{1},y_{2}),c(z_{1},z_{2}),c(u_{1},u_{2})) 
			\in \rho(A,B).
		\end{split}
	\end{equation}			
	If none or all of $x,y,z,tu$ are equal to $(t_{2},t_{1})$, then the		
	implication holds trivially. But if at least one but not all of		
	$x,y,z,u$ are equal to $(t_{2},t_{1})$, then		
	the left hand-side of~\eqref{eql2} is not true, and 
	\eqref{eql2} holds. 
	Hence $c$ is type-preserving.
	Seeking a contradiction,
	let us suppose that 
	$c$ is the restriction of a polynomial function
	$p$. Then there exist $R_1, R_2\in D^{(n\times n)}$
	such that for all $(x_1, x_2)\in T$ we have 
	$c(x_1, x_2)=R_{1}\cdot x_{1}+R_{2}\cdot x_{2}$.
	Since $c(t_{1},t_{0})=t_{0}$ we infer that $R_{1}=0$.		
	Moreover, $c(t_{2},t_{1})=t_{2}$ implies that $R_{1}=1.$ 
	Contradiction. 
	Hence $c$ cannot be interpolated by a polynomial
	and $\ab{D}^{(n\times m)}$ is not strictly $2$-polynomially rich.	
	
	\textbf{Case 6}: \emph{$(m \geq 1, q=2,n=1,k>3)$}: 
	Let
	\[
	\begin{split}
		U'':=\{\matrixvec{x}\in (D^{(1\times m)})^k
		\mid & \exists i\in \finset{k} \, 
		\forall j\in \finset{k}\setminus\{i\}\colon\\
		&\matrixvec{x}(i)=(\vec{e}_m^1)^T,\, 
		\matrixvec{x}(j)=(\vec{0}_m)^T\},
	\end{split}
	\]
	let $\matrixvec{0}_m$ the element of
	$(D^{(1\times m)})^k$ whose all entries are $(\vec{0}_m)^T$, 
	let $\matrixvec{e}_m^1$ 
	the element of
	$(D^{(1\times m)})^k$ whose all entries are 
	$(\vec{e}_m^1)^T$, 
	let $U':=U''\cup \{\matrixvec{0}_m\}$, 
	let $U:= U'\cup \{\matrixvec{e}_m^1\}$, and let
	$c\colon U\to D^{(1\times m)}$ be the constant 
	function with constant value $(\vec{0}_m)^T$ on $U'$
	and the constant function with constant 
	value $(\vec{e}_m^1)^T$ on $U\setminus U'$. 
	Lemma~\ref{lemmasm} implies that $c$ is a
	congruence preserving function.
	Next, we show that $c$ is type-preserving.
	Let $A,B$ be submodules of
	$\ab{D}^{(1\times m)}$ with $A \prec B$. 
	Next, we show that for all
	$\matrixvec{x}, \matrixvec{y}, \matrixvec{z}, 
	\matrixvec{t}\in U$
	\begin{equation}\label{eq:caso10_objective}
		\begin{split}
			&\forall i\in\finset{k}\colon\\
			&(\matrixvec{x}(i), \matrixvec{y}(i), \matrixvec{z}(i), 
			\matrixvec{t}(i))\in \rho(A,B)	
			\Rightarrow	
			(c(\vec{x}),c(\vec{y}),c(\vec{z}), c(\vec{t}))\in\rho(A,B).
		\end{split}
	\end{equation}
	If $(\vec{e}_m^1)^T\in A$, then $c(\matrixvec{e}_m^1)
	\equiv (\vec{0}_m)^T \bmod A$
	and \eqref{eq:caso10_objective} follows. 	
	If $(\vec{e}_m^1)^T\notin A$, then
	all the elements of $U$ are pairwise different modulo $A$. 
	If none or all of $\matrixvec{x}, \matrixvec{y}, \matrixvec{z}, 
	\matrixvec{t}$ are equal 
	to $\matrixvec{e}_m^1$, then \eqref{eq:caso10_objective} holds since
	$\rho(A,B)$ contains all diagonal tuples. 
	Otherwise, the
	left-hand side of 
	\eqref{eq:caso10_objective} can never be true, 
	and therefore, $c$ preserves $\rho(A,B)$.
	Hence $c$ is a function which preserves types.
	Every $k$-ary zero-preserving polynomial function 
	$p$ of $\ab{D}^{(1\times m)}$ 
	is of the form
	\begin{equation*}
		p(x_{1},\dots ,x_{k})=r_{1}x_{1}+\dots +r_{k}x_{k}
		\text{  for some }r_{1},\dots ,r_{k}\in D .
	\end{equation*}
	Since $c(U')= (\vec{0}_m)^T$, we infer that 
	$ r_{1}=\dots=r_{k}=0$.
	Since $c$ is not the zero-function, $c$ is not the
	restriction of a polynomial function of 
	$\ab{D}^{(n\times m)}$,
	and therefore, $\ab{D}^{(n\times m)}$ is not
	strictly $k$-polynomially rich.
	
	The remaining two cases are similar to 
	Case 6. For the sake of brevity we only 
	provide the function
	that shows that $\ab{D}^{(n\times m)}$ is not
	strictly $k$-polynomially rich. 
	
	\textbf{Case 7}: \emph{$(m \geq 1,q=3,n=1,k>2)$}:
	Let $U''$ and $U'$ be defined as in Case 6
	and let $U=U'\cup \{2 \cdot \matrixvec{e}_m^1\}$,
	where $\matrixvec{e}_m^1$ is defined as in Case 6.
	Let $c$ be the constant function with constant 
	value $(\vec{0}_m)^T$ on $U'$ and with value 
	$2\cdot (\vec{e}_m^1)^T$ on $U\setminus U'$. 
	One can prove as in Case 6 that 
	$c$ cannot be interpolated by a 
	polynomial function of $\ab{D}^{(n\times m)}$.
	
	\textbf{Case 8}: \emph{$(m \geq 1,q=5,n=1,k >1)$}:
	Let $U''$ and $U'$ be defined as in Case 6
	and let $U=U'\cup \{4 \cdot \matrixvec{e}_m^1\}$,
	where $\matrixvec{e}_m^1$ is defined as in Case 6.
	Let $c$ be the constant function with constant 
	value $(\vec{0}_m)^T$ on $U'$ and with value 
	$4\cdot (\vec{e}_m^1)^T$ on $U\setminus U'$.
	One can prove as in Case 6 that 
	$c$ cannot be interpolated by a 
	polynomial function of $\ab{D}^{(n\times m)}$.
\end{proof}
\section{Preliminaries on Mal'cev algebras}\label{sec:preliminary_on_malcevalgebras}
In this section we recall some basic facts on the congruence lattice of a Mal'cev algebra. 
The fact that every abelian algebra in a congruence modular 
variety is affine was %
proved in \cite{Her79}. 
For more details on abelian algebras in congruence 
modular varieties we refer the reader to \cite[Chapter~5]{FM:CTFC}. 
The coordinatization of abelian congruences that we report can also 
be found in \cite{Fre83, Aic18, Ros24}.
We start with a preliminary result about commutators and 
centralizers in Mal'cev algebras. 
\begin{proposition}[{cf.~\cite[Proposition~4.3]{FM:CTFC}}]\label{prop:commutator_lattice}
Let $\ab{A}$ be a Mal'cev algebra, and let $[\cdot, \cdot  ]$ 
be the term condition commutator operation on $\Con \ab{A}$ 
as defined in \cite[Definition~4.150]{MMT:ALVV}. 
Then $(\Con \ab{A}; \wedge, \vee, [\cdot, \cdot])$ is 
a commutator lattice as defined in \cite[Definition~3.1]{Aic18}.
Moreover, the centralizer operation $(\cdot : \cdot)$,
as defined in \cite[Definition~4.150]{MMT:ALVV}, 
coincides with the residuation defined in \cite[Section~3]{Aic18}.
\end{proposition}
Let $\ab{A}$ be a Mal'cev algebra with Mal'cev 
polynomial $d$, and let $o\in A$.
Then we define two binary operations on $A$ as follows:
For all $x_1, x_2\in A$ we let
\begin{equation}\label{eq:equazione_che_definisce_il_gruppo_abeliano_del_termine_di_Malcev}
\begin{split}
x_1+_ox_2&:=d(x_1, o, x_2)\text{ and }\\
x_1 -_o x_2&:=d(x_1, x_2, o).
\end{split}
\end{equation}
For $x\in A$, we let $-_o x:=o-_ox$.
The following lemma goes back to \cite{Her79,Fre83}; a proof that relies only
on the definition of the commutator operation can be found in \cite[Section~3]{Aic19a}. 
\begin{lemma}\label{lemma:piu_emeno_fanno_gruppoo_abeliano_nella_classe_congruenza}
Let $\ab{A}$ be a Mal'cev algebra, let $\alpha\in \Con\ab{A}$ with
$[\alpha, \alpha]=\bottom{A}$, let $o\in A$, and let $+_o$ and $-_o$
be defined as in \eqref{eq:equazione_che_definisce_il_gruppo_abeliano_del_termine_di_Malcev}.
Then $(o/\alpha; +_o, -_o, o)$ is an abelian group.  
\end{lemma}
Let $\ab{A}$ be a Mal'cev algebra, let $\alpha\in \Con\ab{A}$ with
$[\alpha, \alpha]=\bottom{A}$, let $o\in A$, let $+_o$ and $-_o$
be defined as in \eqref{eq:equazione_che_definisce_il_gruppo_abeliano_del_termine_di_Malcev}.
We define
\begin{equation}\label{eq:equazione_che_definisce_l'universo_dell'anello_dei_polinomi_ristretti}
R_o=\{p\restrict{o/\alpha}\mid p\in\POL\ari{1}\ab{A}\text{ and }p(o)=o\}.
\end{equation}
Furthermore, on $R_o$ we define two operations as follows:
For all $p,q\in R_o$ we let 
\begin{equation}\label{eq:equazione_che_definisce_le_operazioni_dell'anello_dei_polinomi_ristretti}
\begin{split}
(p+q)(x)&:=p(x)+_o q(x)=d(p(x),o,q(x))\text{ and } \\
(p\circ q)(x)&:=p(q(x)), \text{ for all } x\in A.
\end{split}
\end{equation} 
We define $\ab{R}_o:=(R_o; +, \circ)$,
and its action on
$o/\alpha$ as follows: For all $r\in R_o$ and $x\in o/\alpha$ we let 
$r\cdot x:=r(x)$.   
\begin{lemma}\label{lemma:piu_meno_fanno_modulo_su-anello_polinomi_ristretto_nella_classe_congruenza}
Let $\ab{A}$ be a Mal'cev algebra, let $\alpha\in \Con\ab{A}$ with
$[\alpha, \alpha]=\bottom{A}$, let $o\in A$, let $+_o$ and $-_o$
be defined as in \eqref{eq:equazione_che_definisce_il_gruppo_abeliano_del_termine_di_Malcev},
and let $\ab{R}_o$ be defined as in
\eqref{eq:equazione_che_definisce_l'universo_dell'anello_dei_polinomi_ristretti}
and \eqref{eq:equazione_che_definisce_le_operazioni_dell'anello_dei_polinomi_ristretti}. Then 
$\ab{R}_o$ is a ring with unity and $(o/\alpha; +_o)$ is a module over $\ab{R}_o$. 
Moreover, the algebra $\ab{A}\restrict{o/\alpha}$ is polynomially equivalent to 
the $\ab{R}_o$-module $(o/\alpha; +_o)$. 
\end{lemma}
In the following we list some properties of projective intervals 
in Mal'cev algebras. 
Given a lattice $\ab{L}$ and two intervals 
$\interval{\alpha}{\beta}$ and $\interval{\gamma}{\delta}$
we say that $\interval{\alpha}{\beta}$ 
\emph{transposes up to} $\interval{\gamma}{\delta}$,
and we write $\interval{\alpha}{\beta}\nearrow\interval{\gamma}{\delta}$,
if $\beta\vee \gamma= \delta$ and $\beta\wedge \gamma=\alpha$. 
Analogously, we say that $\interval{\alpha}{\beta}$ 
\emph{transposes down to} $\interval{\gamma}{\delta}$,
and we write $\interval{\alpha}{\beta}\searrow\interval{\gamma}{\delta}$,
if $\beta=\alpha\vee \delta$ and $\gamma=\alpha\wedge \delta$. 
We say that $\interval{\alpha}{\beta}$
and $\interval{\gamma}{\delta}$ are 
\emph{projective}, and we write 
$\interval{\alpha}{\beta}\leftrightsquigarrow\interval{\gamma}{\delta}$,
if there are $n \in \N_0$ and a finite sequence 
\[\interval{\alpha}{\beta}=\interval{\eta_0}{\theta_0},
\interval{\eta_1}{\theta_1},\dots, 
\interval{\eta_n}{\theta_n}=\interval{\gamma}{\delta}
\]
such that $\interval{\eta_i}{\theta_i}$ transposes up or down 
to $\interval{\eta_{i+1}}{\theta_{i+1}}$ for each $i<n$. 
For 
further results
about projective intervals in modular lattices we refer the reader 
to \cite[Chapter~2]{MMT:ALVV}.
\begin{lemma}\label{lemma:exAic18Lemma3.4}
Let $\ab{A}$ be a Mal'cev algebra, let $\alpha, \beta, \gamma,\delta\in \Con\ab{A}$ 
with $\alpha\leq\beta$, $\gamma\leq\delta$, 
and $\interval{\alpha}{\beta}\leftrightsquigarrow\interval{\gamma}{\delta}$. 
Then we have:
\begin{enumerate}
\item $(\alpha:\beta)=(\gamma:\delta)$;\label{item:centralizerinprojectiveintervals}
\item $[\beta, \beta]\leq \alpha$ if and only if $[\delta, \delta]\leq \gamma$. \label{item:typesinprojactiveintervals}
\end{enumerate}
\end{lemma}
\begin{proof}
Proposition~\ref{prop:commutator_lattice} implies that the assumptions 
of \cite[Lemma~3.4]{Aic18} are satisfied, and the statement follows. 
\end{proof}
The following
lemmata from \cite[Section~3]{Ros24} 
can be viewed as a partial generalization 
of the fact that in a group $\ab{G}$, given two normal 
subgroups $A$ and $B$, we have $A/(A\cap B)\cong (A+B)/B$. 
\begin{lemma}[{\cite[Lemma~3.2]{Ros24}}]\label{lemma:conseguenze_permutabilita_intervalli_proiettivi_esistenza,d(b,o,c)}
Let $\ab{A}$ be a Mal'cev algebra, let $\alpha,\beta,\gamma,\delta\in \Con\ab{A}$
with $\alpha\leq\beta$, $\gamma\leq\delta$ and 
$\interval{\alpha}{\beta}\nearrow\interval{\gamma}{\delta}$. 
Then for all $o\in A$ and for each $x\in o/\delta$ there exist 
$b\in o/\beta$ and $c\in o/\gamma$
such that $x\mathrel{\alpha} d(b,o,c)$. 
\end{lemma}
\begin{lemma}[{\cite[Lemma~3.4]{Ros24}}]\label{lemma:isomorphisms_polynomial_restrictions_projective_intervals}
Let $\ab{A}$ be a Mal'cev algebra, let $n\in\N$,
let $o\in A$, let $\alpha,\beta,\gamma,\delta\in \Con\ab{A}$
with $\alpha\leq\beta$, $\gamma\leq\delta$ and 
$\interval{\alpha}{\beta}\nearrow\interval{\gamma}{\delta}$.

For each $p\in\POL\ari{n}\ab{A}$ with $p(o,\dots, o)=o$ we define 
\[
\begin{split}
f\colon ((o/\alpha)/(\beta/\alpha))^n\to (o/\alpha)/(\beta/\alpha) &\text{ by } 
(x_1/\alpha, \dots, x_n/\alpha)\mapsto p(\vec{x})/\alpha\\
g\colon ((o/\gamma)/(\delta/\gamma))^n\to (o/\gamma)/(\delta/\gamma) &\text{ by } 
(x_1/\gamma, \dots, x_n/\gamma)\mapsto p(\vec{x})/\gamma.
\end{split}
\]
Moreover, we let 
\[
	h_o:=\{(x/\gamma,b/\alpha)\in (o/\gamma)/(\delta/\gamma)\times
	(o/\alpha)/(\beta/\alpha)\mid \exists c\in o/\gamma\colon x\mathrel{\alpha}d(b,o,c)\}.
	\] 
Then $h_o$ is a function and an
isomorphism between the algebras $((o/\alpha)/(\beta/\alpha); f)$ 
and $((o/\gamma)/(\delta/\gamma); g)$. 
\end{lemma}
The following lemma is a consequence of the theory developed in 
\cite[Chapters~1-2]{HM:TSOF}:
\begin{lemma}\label{lemma:projectivity_andTCT}
Let $\ab{A}$ be a finite algebra, 
let $\quotienttct{\alpha}{\beta}$
and $\quotienttct{\gamma}{\delta}$ be two prime quotients
in $\Con\ab{A}$
with $\interval{\alpha}{\beta}
\leftrightsquigarrow \interval{\gamma}{\delta}$.
Then $\Mtct{\ab{A}}{\alpha}{\beta}=\Mtct{\ab{A}}{\gamma}{\delta}$. 
Furthermore, if $\mu\in\Con\ab{A}$ and 
$\interval{\bottom{A}}{\mu}$ is a simple 
complemented modular lattice,
then $\quotienttct{\bottom{A}}{\mu}$ is
tame and for each atom $\alpha$ of $\interval{\bottom{A}}{\mu}$
we have 
$\Mtct{\ab{A}}{\bottom{A}}{\mu}=\Mtct{\ab{A}}{\bottom{A}}{\alpha}$. 
\end{lemma}
\begin{proof}
The first statement is an
easy consequence of 
\cite[Exercises~2.19(3)]{HM:TSOF}.
Furthermore, it easy to see that 
if $\interval{\bottom{A}}{\mu}$
is a simple complemented modular lattice, 
then it is tight \cite[Exercises~1.14(5)]{HM:TSOF},
and therefore, the  quotient $\quotienttct{\bottom{A}}{\mu}$
is tame \cite[Theorem~2.11]{HM:TSOF}. 
Let $\alpha$ be an atom of $\interval{\bottom{A}}{\mu}$
and let $\beta$ be its complement in $\interval{\bottom{A}}{\mu}$.
Then $\interval{\bottom{A}}{\alpha}\nearrow
\interval{\beta}{\mu}$, and by the first statement of this 
lemma, we have that $\Mtct{\ab{A}}{\bottom{A}}{\alpha}=
\Mtct{\ab{A}}{\beta}{\mu}$. Thus, 
\cite[Exercises~2.19(2)]{HM:TSOF} imply that
$\Mtct{\ab{A}}{\beta}{\mu}=
\Mtct{\ab{A}}{\bottom{A}}{\mu}$, and therefore,
$\Mtct{\ab{A}}{\bottom{A}}{\alpha}=\Mtct{\ab{A}}{\bottom{A}}{\mu}$.
\end{proof}
In the following we investigate the properties of 
the algebra $\ab{A}\restrict{o/\mu}$ for $o\in A$
and $\mu\in \Con\ab{A}$ when 
$\ab{A}$ is a finite Mal'cev algebra and $\mu$ is an abelian
congruence of $\ab{A}$ such that
$\interval{\bottom{A}}{\mu}$ is a simple 
complemented modular lattice. 
For the theory of simple complemented modular lattices 
(or projective geometries) we refer the reader to
\cite[Section~4.8]{MMT:ALVV}.
We start by listing some results from \cite[Section~4]{Ros24}
that that will be often referred to in this paper.
\begin{proposition}[{\cite[Proposition~4.3]{Ros24}}]\label{prop:identification_with_matrices}
Let $\ab{A}$ be a finite Mal'cev algebra, let $\mu$ be an abelian
element of $\Con{\ab{A}}$ such that $\interval{\bottom{A}}{\mu}$ 
is a
simple complemented modular lattice of height $h$, let $o\in A$
with $\card{o/\mu}>1$, and
let $\ab{R}_o$ be defined as in 
\eqref{eq:equazione_che_definisce_l'universo_dell'anello_dei_polinomi_ristretti}
and \eqref{eq:equazione_che_definisce_le_operazioni_dell'anello_dei_polinomi_ristretti}.
Then for each atom $\alpha$ of $\interval{\bottom{A}}{\mu}$
the ring of $\ab{R}_o$-endomorphisms of the $\ab{R}_o$-module 
$(o/\alpha;+_o)$ is a field $\ab{D}$, and $(o/\alpha;+_o)$
is a $\ab{D}$-vector space. Let $n$ be the
dimension of $(o/\alpha;+_o)$ over $\ab{D}$.
Then there exist a ring isomorphism 
$\epsilon_{\ab{R}_o}\colon \ab{R}_o\to \ab{M}_n(\ab{D})$,
and
a group isomorphism $\epsilon_\mu^o\colon 
(o/\mu; +_o)\to\ab{D}^{(n\times h)}$ such
that for all $r\in R_o$ and $v\in o/\mu$ we have 
$\epsilon_\mu^o (r(v))=\epsilon_{\ab{R}_o}(r)\cdot \epsilon_\mu^o(v).$
\end{proposition}

\begin{lemma}\label{lemma:on_cardinality_of-classes_intervals_simple_com_mod_latt}
Let $\ab{A}$ be a finite Mal'cev 
algebra, let $\mu$ be an abelian congruence of
$\ab{A}$ such that $\interval{\bottom{A}}{\mu}$ 
is a simple complemented modular lattice, and let 
$o\in A$. Then for all prime quotients
$\quotienttct{\gamma}{\delta}$ and 
$\quotienttct{\eta}{\theta}$ in $\interval{\bottom{A}}{\mu}$
we have 
$\card{o/\gamma/(\delta/\gamma)}=\card{o/\eta/(\theta/\eta)}$.
Furthermore, $\card{o/\mu}=1$ if and only if 
for each atom $\alpha$ of $\interval{\bottom{A}}{\mu}$
we have that
$\card{o/\alpha}=1$. 
\end{lemma}
\begin{proof}
Since $\interval{\bottom{A}}{\mu}$ is 
a simple complemented modular lattice, 
all prime quotients are projective to each other, and so the 
first part of the statement 
follows from 
Lemma~\ref{lemma:isomorphisms_polynomial_restrictions_projective_intervals}.
Clearly, if $\card{o/\mu}=1$, then $\card{o/\beta}=1$ for each $\beta\leq\mu$.
The converse implication follows from the fact that 
$\ab{A}$ is congruence permutable and $\mu$ is the join of the
atoms of $\interval{\bottom{A}}{\mu}$
(cf.~\cite[Lemma~4.83]{MMT:ALVV}).  
\end{proof}
\begin{definition}
Let $\ab{A}$ be a Mal'cev algebra and let 
$B\subseteq A$. We say that $B$ is \emph{strictly 
$k$-polynomially rich with respect to $\ab{A}$}
if for all $T\subseteq B^k$ and for each 
type-preserving $f\colon T\to A$ with image contained in $B$,
there exists $p\in\POL\ari{k}\ab{A}$  that preserves $B$ and
interpolates $f$ on $T$. 
\end{definition}
Obviously, isomorphic copies of type-preserving functions are type-preserving:
\begin{lemma}\label{lemma:basic_fact}
Let $\ab{B}$ and $\ab{C}$ be two algebras 
and let $\phi$ be an isomorphism from $\ab{B}$ to $\ab{C}$. 
Let $l\in \N$, let $T\subseteq B^l$ and let $f\colon T\to C$.
Let $S$ be the image obtained by the component-wise action of $\phi$ on $T$,
and let $g\colon S\to C$ be defined for each $\vec{s}\in S$ by 
\[
\vec{s}\mapsto \phi(f(\phi^{-1}(\vec{s}(1)), \dots, \phi^{-1}(\vec{s}(l))).
\]
Then $f$ is type-preserving if and only if $g$ is type-preserving. 
\end{lemma}
The following Lemma allows us to consider partial functions locally
as functions on modules.
\begin{lemma}\label{lemma:restrizione_classe_k_poly_rich_sse:modulo}
	Let $\ab{A}$ be a finite algebra with a Mal'cev polynomial,
	let $\mu$ be an abelian congruence of $\ab{A}$ such that $\interval{\bottom{A}}{\mu}$
	is a simple complemented modular lattice of height $h$,
	let $o\in A$ with $\card{o/\mu}>1$, let $k\in \N$, 
	let $n$, $\ab{D}$, and $\epsilon_\mu^o$ be the natural number, 
	the field and the
	group homomorphism built in
	Proposition~\ref{prop:identification_with_matrices}. 
	Then the $\ab{M}_n(\ab{D})$-module $\ab{D}^{(n\times h)}$
	is strictly $k$-polynomially rich if and only if $o/\mu$
	is strictly $k$-polynomially rich with respect to $\ab{A}$. 
\end{lemma}
\begin{proof}
First, we show that for all 
$x,y,z\in A$ with $\card{\{x,y,z\}/\mu}=1$,
and for all $\gamma\in \interval{\bottom{A}}{\mu}$ we have 
\begin{equation}\label{eq:ex:Lemma4.5}
(x,y)\in \gamma \Leftrightarrow (d(x,y,z), z)\in \gamma.
\end{equation}
Clearly, if $(x,y)\in \gamma$, then 
$d(x,y,z)\mathrel{\gamma} d(x,x,z)=z$. 
For the opposite implication we assume that 
$(d(x,y,z), z)\in \gamma$, and show $(x,y)\in \gamma$. 
The abelianity of $\mu$ and \cite[Proposition~2.6]{Aic06} yield
\[
\begin{split}
x= &d(x,y,y)=d(d(x,y,y),d(y,y,y),d(z,z,y))=\\
&d(d(x,y,z),d(y,y,z),d(y,y,y))=d(d(x,y,z),z,y)\mathrel{\gamma}\\
&d(z,z,y)=y.
\end{split}\]
This concludes the proof of \eqref{eq:ex:Lemma4.5}.

By Lemma \ref{lemma:piu_meno_fanno_modulo_su-anello_polinomi_ristretto_nella_classe_congruenza},
$\ab{A}\restrict{o/\mu}$ and the $\ab{R}_o$-module
$(o/\mu;+_o)$ are polynomially equivalent.
Hence these two algebras have the same congruences,
and the same TCT-types of prime quotients. Hence 
$\ab{A}\restrict{o/\mu}$ is strictly $k$-polynomially rich if and only if 
the $\ab{R}_o$-module $(o/\mu;+_o)$ is 
is strictly $k$-polynomially rich.
Furthermore, by Proposition \ref{prop:identification_with_matrices}, the 
$\ab{M}_n(\ab{D})$-module $\ab{D}^{(n\times h)}$ is isomorphic 
to the $\ab{R}_o$-module
$(o/\mu;+_o)$. Thus, Lemma \ref{lemma:basic_fact} implies that the 
$\ab{M}_n(\ab{D})$-module $\ab{D}^{(n\times h)}$ is strictly $k$-polynomially rich if and only if
$\ab{A}\restrict{o/\mu}$ is strictly $k$-polynomially rich. 
In the following we will show that 
$\ab{A}\restrict{o/\mu}$ is strictly $k$-polynomially rich if 
and only if $o/\mu$ is strictly $k$-polynomially rich with respect to $\ab{A}$.

Let us denote the lattice of $\ab{R}_o$-submodules of $(o/\mu;+_o)$
by $\ab{L}$.
Lemma~\ref{lemma:piu_emeno_fanno_gruppoo_abeliano_nella_classe_congruenza} and
Lemma~\ref{lemma:piu_meno_fanno_modulo_su-anello_polinomi_ristretto_nella_classe_congruenza}
imply that for each $\gamma\in \interval{\bottom{A}}{\mu}$ the set
$o/\gamma$ is an  $\ab{R}_o$-submodule of $(o/\mu; +_o)$. 
Next, we show that 
the map 
$\Psi\colon \Interval{\bottom{A}}{\mu}\to \ab{L}$
defined by $\gamma \mapsto o/\gamma$
is a lattice 
isomorphism.
We first prove that $\Psi$ is a homomorphism.
Clearly, $\Psi$ is a $\wedge$-homomorphism since for all 
$\gamma, \delta\in \interval{\bottom{A}}{\mu}$ we have 
$o/(\gamma \wedge \delta)=o/\gamma \cap o/\delta$. 
Next, we prove that $\Psi$ is also a $\vee$-homomorphism.
To this end, let $\gamma, \delta\in \interval{\bottom{A}}{\mu}$; we show that 
$o/(\gamma\vee\delta)=o/\gamma +_o o/\delta$. 
For each $x\in o/\mu$ 
\begin{equation}\label{eq:join}
x\in o/(\gamma \vee \delta) \Longleftrightarrow \exists y\in o/\mu \colon x\mathrel{\gamma}
y\mathrel{\delta} o \Longleftrightarrow d(x,y,o)\mathrel{\gamma} o \text{ and } y\in o/\delta;
\end{equation}
where the first equivalence follows from the fact that $\ab{A}$ is congruence permutable and the
second equivalence follows from \eqref{eq:ex:Lemma4.5}. 
By \cite[Lemma~3.1(3)]{Aic18} we have that 
\[(x-_o y)+_o y=d(d(x,y,o),o,y)=x;
\]
and therefore \eqref{eq:join} implies that \[
x\in o/(\gamma \vee \delta)\Longrightarrow  x\in \{u+_o v\mid u\in o/\gamma \text{ and } v\in o/\delta\},
\] 
which shows $o/(\gamma\vee\delta)\subseteq o/\gamma +_o o/\delta$. 
To prove $o/(\gamma\vee\delta)\supseteq o/\gamma +_o o/\delta$, let 
$x\in o/\gamma +_o o/\delta$. Then there exists $u\in o/\gamma$ and $v\in o/\delta$
such that $x=d(u, o,v)$, and therefore, 
\[
x= d(u, o,v)\mathrel{\gamma} d(o,o,v) =v\mathrel{\delta} o. 
\]
Thus, $\Psi$ is also a $\vee$-homomorphism, and hence a lattice homomorphism. 

Since $\Interval{\bottom{A}}{\mu}$ is simple and $\Psi$ is not constant, $\Psi$ is injective. 
Next, we show that $\Psi$ is surjective. To this end, let 
$N$ be an atom of $\ab{L}$. By Proposition~\ref{prop:identification_with_matrices},
there exists $x\in o/\mu$ such that $N=R_o \cdot x$. 
Let $\eta:=\Cg[\ab{A}]{\{(o,x)\}}$. Next, we show that 
\begin{equation}\label{eq:etaequalsN}
o/\eta =N. 
\end{equation}
Let $v\in o/\eta$. By \cite[Theorem~4.70(iii)]{MMT:ALVV} there exists $p\in \POL_1\ab{A}$ with
$p(x)=v$ and $p(o)=o$. Thus, $p\restrict{o/\mu}\in R_o$ and $v=p\restrict{o/\mu}\cdot x\in N$. 
This proves that $o/\eta\subseteq N$. For the other inclusion of \eqref{eq:etaequalsN}, let
$n\in N$. Then there exists $r\in R_o$ with $n=r\cdot x$. Thus, there exists $p\in \POL_1\ab{A}$ such that 
$p\restrict{o/\mu}=r$, $p(o)=o$, and $p(x)=n$. Then by \cite[Theorem~4.70(iii)]{MMT:ALVV},
$(n, x)\in \eta$. This concludes the proof of~\eqref{eq:etaequalsN}. 

Since each element of $\ab{L}$ is join of atoms, that fact that $\Psi$ is a homomorphism 
and~\eqref{eq:etaequalsN} imply
that $\ab{L}$ is contained in the image of $\Psi$.  
This concludes the proof that $\Psi$ is a lattice isomorphism. 
 
The fact that $\Psi$ is a lattice isomorphism implies that the congruences of 
$\ab{A}\restrict{o/\mu}$ are exactly the sets of the form 
$\gamma \cap o/\mu^2$ with $\gamma\in\interval{\bottom{A}}{\mu}$. 
Moreover, it is easy to verify that for each prime quotient 
$\quotienttct{\alpha}{\beta}$ of 
$\interval{\bottom{A}}{\mu}$ we have
\[
\rho(\alpha\cap o/\mu^2, \beta\cap o/\mu^2)=
\rho(\alpha,\beta)\cap o/\mu^4.
\]
Therefore, for each $k\in \N$, for each $T\subseteq o/\mu^k$ and for each $f\colon T \to o/\mu$ 
we have that $f$ preserve the types of $\ab{A}$ if and only if $f$ preserves the types of $\ab{A}\restrict{o/\mu}$. 
This immediately implies that $\ab{A}\restrict{o/\mu}$ is strictly $k$ polynomially rich if 
and only if $o/\mu$ is strictly $k$-polynomially rich with respect to $\ab{A}$.   
\end{proof}

The following theorem relates subtypes of abelian tame quotients to the
modules associated to these quotients by commutator theory. 
\begin{theorem}\label{teor:from_coordinatization_to_TCT}
Let $\ab{A}$ be a finite Mal'cev algebra,
let $\mu$ be an abelian congruence such that
$\interval{\bottom{A}}{\mu}$ is 
a simple complemented modular lattice of height $h$. 
For each $o\in A$ with $\card{o/\mu}>1$
let $\ab{D}_o$ and $n_o$ be the 
field and the natural number built in 
Proposition~\ref{prop:identification_with_matrices},
let $\ab{R}_o$ be defined as in
\eqref{eq:equazione_che_definisce_l'universo_dell'anello_dei_polinomi_ristretti},
and let $\epsilon_\mu^o$, and $\epsilon_{\ab{R}_o}$
be the group and ring isomorphisms built in 
Proposition~\ref{prop:identification_with_matrices}. 
Let $\alpha$ be an atom of $\interval{\bottom{A}}{\mu}$.
Then $\quotienttct{\bottom{A}}{\mu}$ is tame, 
$\Mtct{\ab{A}}{\bottom{A}}{\mu}=\Mtct{\ab{A}}{\bottom{A}}{\alpha}$, 
and
$\card{\ab{D}_o}$ is the subtype of $\quotienttct{\bottom{A}}{\mu}$.
\end{theorem}
 \begin{proof}
 Since $\interval{\bottom{A}}{\mu}$ is 
 a simple complemented modular lattice,
 Lemma~\ref{lemma:projectivity_andTCT} implies that
 $\quotienttct{\bottom{A}}{\mu}$ is tame
 and 
 $\Mtct{\ab{A}}{\bottom{A}}{\mu}=\Mtct{\ab{A}}{\bottom{A}}{\alpha}$. 
  Moreover, by 
 Proposition~\ref{prop:identification_with_matrices},
 $\ab{A}\restrict{o/\mu}$ is polynomially 
 equivalent to the $\ab{M}_{n_o}(\ab{D}_o)$-module
 $\ab{D}_o^{(n_o\times h)}$.
 Since $\card{o/\mu}>1$ and $\interval{\bottom{A}}{\mu}$
 is a simple complemented modular lattice, 
 Lemma~\ref{lemma:on_cardinality_of-classes_intervals_simple_com_mod_latt}
 implies that $\card{o/\alpha}>1$. Thus, there exists 
 $a\in A$ such that $(a,o)\in \alpha\setminus\bottom{A}$, 
 and \cite[Exercise~8.8(1)]{HM:TSOF} implies that
 there exist an $\quotienttct{\bottom{A}}{\alpha}$-trace $N$
 such that $\{o,a \}\subseteq N$ and  
 a $\quotienttct{\bottom{A}}{\alpha}$-minimal set $U$ with
 $U \cap a/\alpha = N$.
 Let $e_U$ be the unary idempotent
 polynomial function of $\ab{A}$ constructed in 
 \cite[Theorem~2.8]{HM:TSOF} that satisfies $e_U(A)=U$. 
 Since $o\in N$ we have that $e_U(o)=o$, and therefore,
 $e_U\restrict{o/\mu}\in R_o$. The functions in $R_o$ can be
 represented by matrices, and so
 by Proposition~\ref{prop:identification_with_matrices},
 there exist $M\in \ab{M}_{n_o}(\ab{D}_o)$ 
  such that for all $x\in o/\mu$, we have
 \[ \epsilon_\mu^o(e_U(x))=M\cdot \epsilon_\mu^o(x).\]
 In particular, since $N\subseteq o/\mu \cap U$ and $e_U$ acts as the identity map
 on $U$, we have that $\epsilon_\mu^o(x)=M\cdot \epsilon_\mu^o(x)$
 for each $x\in N$.
 Thus, \[
 N = o/\alpha \cap U = 
    \{x\in o/\alpha\mid  e_U (x) = x \} =
 \{x\in o/\alpha\mid M\cdot \epsilon_\mu^o(x)=\epsilon_\mu^o(x)\}.
 \]
 Let $N' := \epsilon_\mu^o(N)$. 
 Since 
 $N' = \{ v \in \ab{D}_o^{n_o \times h} \mid v \in 
 \epsilon_{\mu}^o (o/\alpha) \text{ and } M \cdot v = v \}$, $N'$ 
 is a vector space over $\ab{D}_o$.
 Next we prove that $\epsilon_\mu^o(N)$ has dimension one.
 Seeking a contradiction, let us assume that
 there exist $a,b\in N$ such that 
 $\epsilon_\mu^o(a)$ and $\epsilon_\mu^o(b)$ are linearly 
 independent. Let $L$ be the matrix associated to
 the linear function that maps $\epsilon_\mu^o(a)$ to
 $\epsilon_\mu^o(a)$ and $\epsilon_\mu^o(b)$ to $\vec{0}$.
 Then the restriction of $(\epsilon_{\ab{R}_o})^{-1}(M\cdot L)$
 to $N$ is neither constant nor injective, in contradiction with
 \cite[Lemma~2.16(7)]{HM:TSOF}.  Let $q$ be the subtype of 
 $\quotienttct{\bottom{A}}{\mu}$. 
 Since $N$ is a subspace of dimension~$1$ with respect to  both
 fields
 $\ab{D}_o$ and $\GF{q}$,
 we have $\card{\ab{D}_o}=q$. 
 \end{proof}
\begin{corollary}\label{cor:projectivity_preserves_ext_types_and_dimension}
Let $\ab{A}$ be a finite Mal'cev algebra 
and let $\mu_1, \mu_2,\nu_1, \nu_2,\alpha, \beta\in\Con\ab{A}$.
We assume that
$\interval{\mu_1}{\mu_2}$ and $\interval{\nu_1}{\nu_2}$
are simple complemented modular lattices,
that $\quotienttct{\mu_1}{\mu_2}$ and $\quotienttct{\nu_1}{\nu_2}$
are abelian, that $\alpha$ is an atom of 
$\interval{\mu_1}{\mu_2}$, that $\beta$ is an atom of
$\interval{\nu_1}{\nu_2}$, and that
$\interval{\mu_1}{\alpha}\leftrightsquigarrow
\interval{\nu_1}{\beta}$. 
Then $\quotienttct{\mu_1}{\mu_2}$ and $\quotienttct{\nu_1}{\nu_2}$
are both type $\tp{2}$ and have the same 
subtype. Let $q$ be the subtype of both 
$\quotienttct{\mu_1}{\mu_2}$ and $\quotienttct{\nu_1}{\nu_2}$.
Then for each $o\in A$ such that
$\card{(o/\mu_1)/(\mu_2/\mu_1)}>1$
we have that 
$((o/\mu_1)/(\alpha/\mu_1); +_{o/\mu_1})$
and 
$((o/\nu_1)/(\beta/\nu_1); +_{o/\nu_1})$
are vector spaces over $\GF{q}$
of the same dimension. 
\end{corollary}
\begin{proof}
Let $o\in A$ such that
$\card{(o/\mu_1)/(\mu_2/\mu_1)}>1$.
By Lemma~\ref{lemma:on_cardinality_of-classes_intervals_simple_com_mod_latt},
we have that
$\card{(o/\mu_1)/(\alpha/\mu_1)}=\card{(o/\nu_1)/(\beta/\nu_1)}>1$.
By Lemma~\ref{lemma:exAic18Lemma3.4}, 
$\quotienttct{\mu_1}{\mu_2}$ and $\quotienttct{\nu_1}{\nu_2}$
are both type $\tp{2}$. 
By Lemma~\ref{lemma:isomorphisms_polynomial_restrictions_projective_intervals}
$((o/\mu_1)/(\alpha/\mu_1); \{+_{o/\mu_1}\}\cup R_{o/\mu_1})\cong
((o/\nu_1)/(\beta/\nu_1); \{+_{o/\nu_1}\}\cup R_{o/\nu_1})$. 
Thus, Theorem~\ref{teor:from_coordinatization_to_TCT}
implies that $\quotienttct{\mu_1}{\mu_2}$ and $\quotienttct{\nu_1}{\nu_2}$
have the same subtype.
Proposition~\ref{prop:identification_with_matrices} 
implies that 
$((o/\mu_1)/(\alpha/\mu_1); +_{o/\mu_1})$
and 
$((o/\nu_1)/(\beta/\nu_1); +_{o/\nu_1})$
are vector spaces over $\GF{q}$, and 
since $\card{(o/\mu_1)/(\alpha/\mu_1)}=\card{(o/\nu_1)/(\beta/\nu_1)}$,
they have the same dimension.
\end{proof}
\section{Sufficient conditions for (SC1)}
In this section we prove that 
certain algebras satisfy property (SC1).
The following generalization of \cite[Proposition~3.3]{AI:PIIE}
from finite expanded groups to finite Mal'cev algebras
first appeared in \cite{Ros24}. 
\begin{lemma}[{\cite[Lemma~9.3]{Ros24}}]\label{lemma:prop3.3AicIdz04implicazione_da_due_a_uno}
Let $\ab{A}$ be a finite Mal'cev algebra.
Then the following are equivalent:
\begin{enumerate}
\item $\ab{A}$ does not satisfy (SC1);\label{item:A_doesn_satisfy_SC1}
\item there exist two join irreducible congruences $\alpha, \beta$ \label{item:the_failure_exists}
such that 
\begin{equation}\label{eq:failure_of_SC1}
[\alpha,\beta]\leq \alpha^-\prec \alpha\leq\beta^-\prec\beta.
\end{equation}
\end{enumerate}
\end{lemma}
\begin{definition}
Let $\ab{A}$ be a finite Mal'cev algebra.
A pair $(\alpha, \beta)$ of join irreducible 
elements of $\Con\ab{A}$ is called 
\emph{a failure of (SC1)} if $\alpha$ and $\beta$
satisfy \eqref{eq:failure_of_SC1}.
\end{definition}
\begin{proposition}\label{prop:condition_on_sc1_failurs}
Let $\ab{A}$ be a finite strictly $1$-polynomially rich
Mal'cev algebra and let $(\alpha, \beta)$ be 
failure of (SC1). Then
for all $(a_1, a_2, b)\in A^3$ we have that 
\begin{equation} \label{eq:fcons}
\Cg{\{(a_1, a_2)\}}=\alpha\Rightarrow 
\Cg{\{(b, a_1)\}}\neq \beta.
\end{equation}
\end{proposition}
\begin{proof}
First, we remark that $[\beta, \alpha]=[\alpha, \beta]$ by
\cite[Lemma~2.5]{Aic06}. We will use this fact about commutators
implicitly throughout the proof.
Seeking a contradiction, let us suppose that
$(\alpha, \beta)$ is a failure of (SC1) and
there exist $(o,a, b)\in A^3$ such that
$\Cg{\{(o,a)\}}=\alpha$ and $\beta=\Cg{\{(o,b)\}}$. 
Then 
\cite[Proposition~2.6]{Aic06} implies that
\[
\begin{split}
o&=d(o,b,b)\equiv_{\Cg{\{(b, d(a,o,b)\}}} d(d(o,b,b),d(o,o,b),d(a,o,b))\\
&\equiv_{[\beta, \alpha]}d(d(o,o,a),d(b,o,o),d(b,b,b))\\
&=d(a,b,b)=a.
\end{split}
\]
Thus
\(
(0,a) \in \Cg{\{b,d(a,o,b)\}} \vee [\beta, \alpha].
\)
Moreover, we have $b=d(o,o,b)\mathrel{\alpha}d(a,o,b)$.
Therefore, we have 
\[
\alpha=\Cg{\{(o,a)\}}\leq \Cg{\{b, d(a,o,b)\}}\vee [\beta, \alpha]\leq \alpha.
\]
Since $\alpha$ is join irreducible we infer that 
\begin{equation}\label{eq:alpha=congruenzagenerata_d_b}
\alpha=\Cg{\{(b, d(a,o,b)\}}. 
\end{equation}
Moreover, we have 
\[
b=d(o,o,b)\mathrel{\alpha}d(a,o,b)\equiv_{\Cg{\{(o,d(a,o,b))\}}}o.
\]
Thus, since $o=d(a,a,o)\mathrel{\beta} d(a,o,b)$ we have 
\[
\beta=\Cg{\{(o,b)\}}\leq \alpha\vee \Cg{\{(o,d(a,o,b))\}}\leq \beta.
\]
Since $\beta$ is join irreducible we have  
\begin{equation}\label{eq:beta=congruenzagenerata_d_o}
\beta=\Cg{\{(o, d(a,o,b)\}}. 
\end{equation}
Moreover, we have 
\[
o\mathrel{\alpha} a\equiv_{\Cg{\{(a, d(a,o,b))\}}}d(a,o,b)\mathrel{\alpha} d(a,a,b)=b.
\]
Thus, since $a=d(a,o,o)\mathrel{\beta} d(a,o,b)$ we have 
\[
\beta=\Cg{\{(o,b)\}}\leq \alpha\vee \Cg{\{(a,d(a,o,b))\}}\leq \beta.
\]
Since $\beta$ is join irreducible we have  
\begin{equation}\label{eq:beta=congruenzagenerata_d_a}
\beta=\Cg{\{(a, d(a,o,b)\}}. 
\end{equation}
Putting together 
\eqref{eq:alpha=congruenzagenerata_d_b},
\eqref{eq:beta=congruenzagenerata_d_o},
\eqref{eq:beta=congruenzagenerata_d_a},
and \eqref{eq:failure_of_SC1} we infer that
\[
(o,a)\in\Cg{\{(a, d(a,o,b)\}}\cap \Cg{\{(o, d(a,o,b)\}}\cap \Cg{\{(b, d(a,o,b)\}}.
\]
Thus, the partial function $f\colon \{a,o,b,d(a,o,b)\}\to \{a,o\}$ defined by 
$\{a,o,b\}\mapsto o$ and $d(a,o,b)\mapsto a$ is 
congruence preserving. 

Next, we show that $f$ is type-preserving. 
We first prove the following identities:
\begin{align}
\beta&=\Cg{\{(d(o,a,o),d(a,o,b))\}} \label{eq:gen_beta_doao_daob}\\
\beta&=\Cg{\{(d(a,o,a), d(a,o,b))\}} \label{eq:gen_beta_daoa_daob}\\
\beta&=\Cg{\{(d(b, d(a,o,b),b),a)\}} \label{eq:gen_beta_bdaobb_a}\\
\beta&=\Cg{\{(d(b, d(a,o,b),b),o)\}} \label{eq:gen_beta_bdaobb_o}\\
\beta&=\Cg{\{(d(d(a,o,b),b,d(a,o,b)),a)\}} \label{eq:gen_beta_daobbdaob_a}\\
\beta&=\Cg{\{(d(d(a,o,b),b,d(a,o,b)),o)\}} \label{eq:gen_beta_daobbdaob_o}
\end{align} 
We have that
\[
o=d(o,o,o)\mathrel{\alpha}d(o,a,o)\equiv_{\Cg{\{(d(o,a,o),d(a,o,b))\}}} d(a,o,b)\mathrel{\alpha} d(o,o,b)=b.
\]
Thus, $\beta\leq \Cg{\{(d(o,a,o),d(a,o,b))\}} \vee \alpha\leq \beta$. 
Thus since $\beta$ is join irreducible and $\alpha < \beta$ we infer 
$\beta=\Cg{\{(d(o,a,o),d(a,o,b))\}}$.
This proves \eqref{eq:gen_beta_doao_daob}.
Similarly, we have that
\[
o=d(o,o,o)\mathrel{\alpha}d(a,o,a)\equiv_{\Cg{\{(d(a,o,a),d(a,o,b))\}}} d(a,o,b)\mathrel{\alpha} d(o,o,b)=b.
\]
Therefore, arguing as above, we infer that
$\beta=\Cg{\{(d(a,o,a), d(a,o,b))\}}$ and \eqref{eq:gen_beta_daoa_daob} follows. 
Moreover, \eqref{eq:alpha=congruenzagenerata_d_b} implies
\[
a\equiv_{\Cg{\{(d(b, d(a,o,b),b),a)\}}} d(b, d(a,o,b),b) \mathrel{\alpha} b.
\]
Thus, arguing as above, we infer that 
$\beta\leq\alpha\vee \Cg{\{(d(b, d(a,o,b),b),a)\}}\leq \beta$, and \eqref{eq:gen_beta_bdaobb_a}
follows from the fact that
$\alpha<\beta$ and $\beta$ is join irreducible.
Arguing as above, from \eqref{eq:alpha=congruenzagenerata_d_b} we infer that
\[
o\equiv_{\Cg{\{(d(b, d(a,o,b),b),o)\}}} d(b, d(a,o,b),b)\mathrel{\alpha} b.
\]
Thus, $\beta\leq \alpha \vee \Cg{\{(d(b, d(a,o,b),b),o)\}}\leq \beta$, and
\eqref{eq:gen_beta_bdaobb_o} follows from the fact that
$\alpha< \beta$ and $\beta$ is join irreducible. 
Moreover, \eqref{eq:alpha=congruenzagenerata_d_b} implies that
\[a\equiv_{ \Cg{\{(d(d(a,o,b),b,d(a,o,b)),a)\}}} d(d(a,o,b),b,d(a,o,b)) \mathrel{\alpha} b.\]
Thus, $\beta\leq \alpha \vee \Cg{\{(d(d(a,o,b),b,d(a,o,b)),a)\}}\leq \beta$, and
\eqref{eq:gen_beta_daobbdaob_a} follows from the fact that
$\alpha< \beta$ and $\beta$ is join irreducible.
Moreover,  \eqref{eq:alpha=congruenzagenerata_d_b} implies that
\[o\equiv_{\Cg{\{(d(d(a,o,b),b,d(a,o,b)),o)\}}}d(d(a,o,b),b,d(a,o,b))\mathrel{\alpha} b \]
Thus, $\beta\leq \alpha \vee \Cg{\{(d(d(a,o,b),b,d(a,o,b)),o)\}}\leq \beta$, and
\eqref{eq:gen_beta_daobbdaob_o} follows from the fact that
$\alpha< \beta$ and $\beta$ is join irreducible.

Let $\gamma\in\Con\ab{A}$ be such that $\gamma\geq \alpha^-$,
then we have that
\begin{align}
(d(b, d(a,o,b),b),b)&\in\gamma \Rightarrow (d(o,a,o),o)\in\gamma \label{eq:bdaobbimpliesdoaoo}\\
(d(b, d(a,o,b),b),d(a,o,b))&\in\gamma \Rightarrow (d(o,a,o),a)\in\gamma\label{eq:bdaobbdaobimpliesdoaoa}\\
(d(d(a,o,b),b, d(a,o,b)),b)&\in\gamma\Rightarrow (d(a,o,a),o)\in \gamma\label{eq:daobbdaobbimpliesdaoao}\\
(d(d(a,o,b),b,d(a,o,b)),d(a,o,b))&\in\gamma\Rightarrow (d(a,o,a),a)\in\gamma \label{eq:daobbdaobdaobimpliesdaoaa}
\end{align}
Note that $[\alpha,\beta]=[\beta, \alpha]\leq \gamma$ by assumption. 
Thus, \eqref{eq:bdaobbimpliesdoaoo} follows from
\[
\begin{split}
d(o,a,o)=&d(d(o,a,a),d(o,o,a),d(o,b,b))\equiv_{[\alpha,\beta]}d(o, d(a,o,b),b)\equiv_\gamma \\
&d(o, d(d(a,a,a),d(a,a,o),d(b,d(a,o,b),b)),b)\equiv_{[\alpha, \beta]}\\
&d(o, d(b, d(a,o,b),d(a,o,b)),b)=d(o,b,b)=o,
\end{split}
\]
where all the commutator identities are a consequence of
\cite[Proposition~2.6]{Aic06}. 
In a similar fashion we can derive \eqref{eq:bdaobbdaobimpliesdoaoa} from
\[
\begin{split}
a=&d(a,b,b)=d(d(a,o,o), d(d(a,o,b),d(a,o,b),b),d(b,b,b))\equiv_{[\beta,\alpha]}\\
& d(d(a, d(a,o,b),b),d(o,d(a,o,b),b),d(o,b,b)=\\
& d(d(a,d(a,o,b),b),d(o,d(a,o,b),b),d(b,b,o))\equiv_{[\alpha,\beta]}\\
& d(d(a,o,b),b,o)\equiv_\gamma d(d(b, d(a,o,b),b),b,o)=\\
& d(d(b, d(a,o,b),b),d(b,b,b),d(o,o,o))\equiv_{[\alpha, \beta]}\\
& d(o,d(d(a,o,b),b,o),o)\equiv_{[\alpha, \beta]} d(o,d(d(a,o,o),d(o,o,o),d(b,b,o)),o)=\\
&d(o,d(a,o,o),o)=d(o,a,o).
\end{split}
\]
Next, we show \eqref{eq:daobbdaobbimpliesdaoao}. To this end we first observe that
\begin{equation}\label{eq:prima_identita_prova_daobbdaobbimpliesdaoao}
\begin{split}
o=&d(o,b,b)=d(d(o,a,a),d(d(a,o,b),d(a,o,b),b),d(b,b,b))\equiv_{[\beta, \alpha]}\\
&d(d(o,d(a,o,b),b),d(a,d(a,o,b),b),a);
\end{split}
\end{equation}
and
\begin{equation}\label{eq:seconda_identita_prova_daobbdaobbimpliesdaoao}
d(a,d(a,o,b),b)=d(d(a,a,a),d(a,o,b),d(o,o,b))\equiv_{[\beta,\alpha]}d(o,a,a)=o.
\end{equation}
Moreover, under the assumption that 
$(d(d(a,o,b),b, d(a,o,b)),b)\in\gamma$ we have 
\begin{equation}\label{eq:terza_identita_prova_daobbdaobbimpliesdaoao}
\begin{split}
&d(o,d(a,o,b),b)\equiv_\gamma \\
& d(d(a,a,o),d(d(a,o,b),d(a,o,b),d(a,o,b)),d(d(a,o,b),b,d(a,o,b)))\equiv_{[\beta,\alpha]}\\
&d(a, d(a, d(a,o,b),b),o)=\\
&d(d(a,b,b),d(a,d(a,o,b),b),d(o,d(a,o,b),d(a,o,b)))\equiv_{[\alpha, \beta]}\\
&d(o,b,d(a,o,b))=d(d(a,a,o),d(a,a,b),d(a,o,b))\equiv_{[\beta, \alpha]}\\
&d(a,o,o)=a
\end{split}
\end{equation}
Then since $[\beta, \alpha]=[\alpha, \beta]\leq \gamma$,
\eqref{eq:daobbdaobbimpliesdaoao} follows from 
the fact that 
\[
o\equiv_{\gamma} d(d(o,d(a,o,b),b),d(a,d(a,o,b),b),a)\equiv_{\gamma} d(a,o,a),
\]
where the first equality follows form \eqref{eq:prima_identita_prova_daobbdaobbimpliesdaoao},
and the second equality follows from
a combination of \eqref{eq:seconda_identita_prova_daobbdaobbimpliesdaoao}
and \eqref{eq:terza_identita_prova_daobbdaobbimpliesdaoao}. 

Since $[\alpha, \beta]=[\beta, \alpha]\leq\gamma$, \eqref{eq:daobbdaobdaobimpliesdaoaa}
follows from 
\[
\begin{split}
a=&d(d(a,a,a), d(a,o,b),d(a,o,b))\equiv_{\Cg{\{(d(d(a,o,b),b,d(a,o,b)),d(a,o,b))\}}} \\
&d(d(a,a,a),d(a,o,b),d(d(a,o,b),d(o,o,b),d(a,o,b)))\equiv_{[\alpha, \beta]}\\
&d(d(a,a,a),d(a,o,b),d(d(a,o,a), o, b))\equiv{[\beta, \alpha]}\\
&d(d(a,a,d(a,o,a)),d(a,o,o),d(a,b,b))=d(a,o,a).
\end{split}
\]

We are now ready to prove that $f$ is type-preserving.
To this end, let $\quotienttct{\gamma}{\delta}$ be a 
prime quotient of TCT-type $\tp{2}$, and let
$x_1,x_2, x_3,x_4\in\dom f$ with
$\vec{x}\in\rho(\gamma,\delta)$.
We show that $(f(x_1),f(x_2), f(x_3), f(x_4))\in\rho(\gamma, \delta)$. 
We split the proof into three cases:

\textbf{Case 1}: \emph{$\gamma\geq \alpha$}: Since $f$ 
is constant modulo $\alpha$, $f$ is constant 
modulo $\gamma$, and therefore we have 
$d(f(x_1),f(x_2),f(x_3))\mathrel{\gamma} f(x_4)$. 

\textbf{Case 2}: \emph{$\delta\ngeq \alpha$}: 
Since each pair of distinct elements from $\dom f$
generates $\alpha$ or $\beta$
(cf.~\eqref{eq:alpha=congruenzagenerata_d_b}, 
\eqref{eq:beta=congruenzagenerata_d_o}, and
\eqref{eq:beta=congruenzagenerata_d_a}),
we infer that $x_1=x_2=x_3$, and therefore,
by the compatibility of $f$ stated after
\eqref{eq:beta=congruenzagenerata_d_a}, we have
\[d(f(x_1), f(x_2), f(x_3))=f(x_1)=f(d(x_1,x_2,x_3))\mathrel{\gamma} f(x_4).\]

\textbf{Case 3}: \emph{$\gamma\ngeq \alpha$ and $\delta\geq\alpha$}:
In this case $\interval{\alpha^-}{\alpha}\nearrow \interval{\gamma}{\delta}$.
In fact $\gamma\leq \gamma\vee \alpha\leq \delta$, and since $\gamma\prec\delta$
and $\gamma\ngeq \alpha$ we have $\gamma\vee\alpha=\delta$.
Moreover, $\gamma\wedge \alpha=\gamma\wedge \alpha^-$ since 
$\alpha$ is join irreducible, and since $\gamma\prec\delta$, we
also have $\gamma\vee \alpha^-=\delta$. 
Thus, by modularity, we infer that $\gamma\wedge \alpha=\alpha^-$. 
This also proves that
$[\alpha, \beta]\leq \gamma$.
Finally, $\delta\ngeq \beta$. In fact if 
$\delta\geq \beta$, then arguing as above
with $\beta$ in place of $\alpha$
(note that $\beta$ is also join irreducible),
we infer that $\gamma\geq \beta^-$. Since $\beta^-\geq \alpha$,
we have a contradiction with the case assumption.

In order to prove that 
$(f(x_1), f(x_2), f(x_3), f(x_4))\in\rho(\gamma, \delta)$
we need to prove that $d(f(x_1),f(x_2), f(x_3))\mathrel{\gamma} f(x_4)$.
Clearly, since $\delta\ngeq \beta$ and $\beta$
is join irreducible,
\eqref{eq:beta=congruenzagenerata_d_o}
and 
\eqref{eq:beta=congruenzagenerata_d_a},
imply that the triplet $(x_1,x_2, x_3)$ belongs to
$\{a,o\}^3\cup \{d(a,o,b),b\}^3$. 
Let us first consider the case that
$(x_1,x_2, x_3)\in \{a,o\}^3$.
Then since $d(x_1, x_2,x_3)\mathrel{\gamma}x_4$, 
we infer from \eqref{eq:beta=congruenzagenerata_d_o},
\eqref{eq:beta=congruenzagenerata_d_a},
\eqref{eq:gen_beta_doao_daob}, and
\eqref{eq:gen_beta_daoa_daob}
that either $\gamma\geq \alpha$, in contradiction with the
case assumption, 
or $f(x_4)=o$, and then the result follows from 
$(f(x_1), f(x_2), f(x_3), f(x_4))=(o,o,o,o)\in\rho(\gamma, \delta)$.
Let us next consider the case that 
$(x_1,x_2, x_3)\in \{d(a,o,b),b\}^3$.
If $d(x_1, x_2, x_3)\in \{b, d(a,o,b)\}$, then 
since $\gamma\ngeq \alpha$ we have that $x_4=d(x_1,x_2,x_3)$
and $ (f(x_1), f(x_2), f(x_3), f(x_4))\in\rho(\gamma, \delta)$.
This leaves us with eight possibilities that we list below.
\begin{description}
\item [$\vec{x}=(b, d(a,o,b),b, a)$] Then $d(b, d(a,o,b),b)\mathrel{\gamma} a$,
and by \eqref{eq:gen_beta_bdaobb_a}, $\gamma\geq \beta$ in contradiction with the
case assumption.
\item [$\vec{x}=(b, d(a,o,b),b, o)$] Then $d(b, d(a,o,b),b)\mathrel{\gamma} o$,
and by \eqref{eq:gen_beta_bdaobb_o}, $\gamma\geq \beta$ in contradiction with the
case assumption.
\item [$\vec{x}=(d(a,o,b),b,d(a,o,b),a) $] Then $d(d(a,o,b),b,d(a,o,b))\mathrel{\gamma} a$,
and by \eqref{eq:gen_beta_daobbdaob_a}, $\gamma\geq \beta$ 
in contradiction with the case assumption.
\item [$\vec{x}=(d(a,o,b),b,d(a,o,b),o)$] Then $d(d(a,o,b),b,d(a,o,b))\mathrel{\gamma} o$,
and by \eqref{eq:gen_beta_daobbdaob_o}, $\gamma\geq \beta$ 
in contradiction with the case assumption.
\item [$\vec{x}=( b, d(a,o,b),b, b)$] Then $(d(b, d(a,o,b),b),b)\in\gamma$,
and by \eqref{eq:bdaobbimpliesdoaoo}, we have
$d(f(x_1), f(x_2), f(x_3))=d(o,a,o)\mathrel{\gamma} o=f(x_4)$.
\item [$\vec{x}=( b, d(a,o,b),b, d(a,o,b))$] Then $d(b, d(a,o,b),b)\mathrel{\gamma} d(a,o,b)$,
and by \eqref{eq:bdaobbdaobimpliesdoaoa}, 
$d(f(x_1), f(x_2), f(x_3))=d(o,a,o)\mathrel{\gamma} a=f(x_4)$.
\item [$\vec{x}=( d(a,o,b),b, d(a,o,b),b)$] Then $d(d(a,o,b),b, d(a,o,b))\mathrel{\gamma} b$,
and by \eqref{eq:daobbdaobbimpliesdaoao}, 
$d(f(x_1), f(x_2), f(x_3))=d(a,o,a)\mathrel{\gamma} o=f(x_4)$.
\item [$\vec{x}=( d(a,o,b),b, d(a,o,b),d(a,o,b))$] 
Then \[(d(d(a,o,b),b, d(a,o,b)),d(a,o,b))\in\gamma,\]
and by \eqref{eq:daobbdaobdaobimpliesdaoaa},
$d(f(x_1), f(x_2), f(x_3))=d(a,o,a)\mathrel{\gamma} a=f(x_4)$.
\end{description}
This concludes the proof that $f$ is type preserving.

Since $\ab{A}$ is strictly 1-polynomially rich, there exists $p\in\POL\ari{1}\ab{A}$ 
such that for all $x\in\{a,o,b,d(a,o,b)\}$ we have $p(x)=f(x)$.
Thus, \cite[Proposition~2.6]{Aic06} yields 
\[
o=d(o,o,o)=d(p(a),p(o),p(b))\mathrel{[\alpha,\beta]} p(d(a,o,b))=a.
\]
Thus, $(a,o)\in[\alpha, \beta]$, and hence 
$\alpha=\Cg{\{(a,o)\}}\leq [\alpha,\beta]\leq\alpha^-$.
This contradicition $\alpha \le \alpha^-$ shows that~\eqref{eq:fcons}
holds.
\end{proof}

\begin{theorem}\label{teor:sci_necessary}
Let $\ab{A}$ be a finite Mal'cev algebra and let $k\geq 2$. 
Then the following statements are true:
\begin{enumerate}
\item If $\ab{A}$ is strictly $k$-polynomially
rich, then $\ab{A}$ satisfies (SC1).\label{item:k_in_sc1_necessary}
\item If $\ab{A}$ is strictly 1-polynomially
rich and congruence regular, then $\ab{A}$ satisfies (SC1). \label{item:regular_in_SC1_nce}
\end{enumerate}
\end{theorem}
\begin{proof}
For proving \eqref{item:k_in_sc1_necessary},
we observe that the function constructed in 
the proof of \cite[Lemma~12]{IS:PRA} is binary. 
Thus, (SC1) is a necessary condition for 
$k$-polynomially richness. Hence the result follows
from the fact that if $\ab{A}$ is strictly 
$k$-polynomially rich, then $\ab{A}$ is 
$k$-polynomially rich. 

Next, we show \eqref{item:regular_in_SC1_nce}.
Seeking a contradiction let us assume that
there exists a pair of congruences $(\alpha, \beta)$
that is a failure of (SC1).
Let $\alpha=\Cg{\{(a_1, a_2)\}}$. Then
by regularity, there exists $b\in a_1/\beta\setminus a_1/\beta^-$
such that the triple $(a_1, a_2, b)\in A^3$ satisfies 
$\alpha=\Cg{\{(a_1, a_2)\}}$ and $\beta=\Cg{\{(b,a_1)\}}$
in contradiction with Proposition~\ref{prop:condition_on_sc1_failurs}.
Thus, Lemma~\ref{lemma:prop3.3AicIdz04implicazione_da_due_a_uno} implies
that $\ab{A}$ satisfies (SC1). 
\end{proof}
\section{Homogeneous series in Mal'cev algebras with (SC1)} \label{sec:homegeneous_seq}
In this section we develop the notion of
homogeneous congruence (cf.~Definition~\ref{def:homogeneous}) and 
we describe the properties of homogeneous series
(cf.~Definition~\ref{def:hom_series}) in Mal'cev 
algebras with (SC1). 
For an  element $\mu$ of a complete lattice $\ab{L}$ we define 
\[
\begin{split}
\Phi(\mu)&:=\mu\wedge \bigwedge\{\alpha\in L\mid \alpha\prec \mu\};\\
\mu^*&:=\bigvee\{\alpha\in L\mid \alpha\wedge \mu=0\}.
\end{split}
\]
\begin{lemma}[cf. {\cite[Lemma~6.4]{Ros24}}]\label{lemma:whymumeetmustariszero}
Let $\ab{L}$ be a complete modular lattice of finite height,
let $0$ be the bottom element of $\ab{L}$, and let
$\mu$ be a homogeneous element of $\ab{L}$. Then $\mu\wedge\mu^*=0$.  
\end{lemma}
\begin{proposition}[cf. {\cite[Proposition~6.6]{Ros24}}]\label{prop:typesandhomogenuity}
Let $\ab{A}$ be a Mal'cev algebra whose congruence lattice has finite height, let $\mu$ be a homogeneous 
element of $ \Con\ab{A}$, and let $\alpha,\beta\in \Con\ab{A}$ 
with $\alpha\prec\beta\leq\mu$. Then we have:
\begin{enumerate}
\item If $[\beta,\beta]\nleq\alpha$, then 
$\alpha=\bottom{A}$, $\beta=\mu$, and $(\bottom{A}:\mu)=\mu^*$. 
\label{item:typethreecasehomogeneous}
\item If $[\beta,\beta]\leq\alpha$, then
$(\alpha :\beta)=(\Phi(\mu):\mu)\geq \mu\vee\mu^*$. \label{item:typetwocasehomogeneous}
\end{enumerate}
\end{proposition}
Putting together Propositions~6.7, 6.8 and 6.9 from \cite{Ros24}
we infer the following results on 
homogeneous congruences in finite Mal'cev algebras
with (SC1):
\begin{proposition}\label{prop:existence_of_hom_congruences_under_SC1}
Let $\ab{A}$ be a finite Mal'cev algebra with (SC1). 
Then $\Con \ab{A}$ has a homogeneous element.
\end{proposition}
\begin{proposition}\label{prop:homogenuityandcolpementation}
Let $\ab{A}$ be a finite Mal'cev algebra with (SC1)
and let $\mu$ be an homogeneous element of $\Con\ab{A}$.
Then we have:
\begin{enumerate}
\item $\interval{\bottom{A}}{\mu}$ is 
a simple complemented modular lattice;\label{item:simplecomplmodlatticebelowhom}
\item $\Phi(\mu)=\bottom{A}$;\label{item:phiofhomogenous}
\item for all $\alpha\in \Con\ab{A}$ we have $\alpha\geq \mu$ or $\alpha\leq \mu\vee \mu^*$; \label{item:splithomogeous}
\item $(\Phi(\mu):\mu)\leq\mu\vee\mu^*$;\label{item_Prop_6.8_Ros24}
\item if $[\mu,\mu]\neq \mu$ then $(\Phi(\mu):\mu)=
(\bottom{A}:\mu)=\mu\vee\mu^*$ and
$[\mu,\mu]=\bottom{A}$. \label{item:ex_if_not_neutral_then_abelian_for_hom}
\end{enumerate}
\end{proposition}
\begin{lemma}\label{lemma:projecting_belowo_mu_star}
	Let $\ab{A}$ be a finite Mal'cev algebra with (SC1), let $\mu$ be a homogeneous
	abelian congruence of $\ab{A}$, and let $\quotienttct{\alpha}{\beta}$ be 
	a type $\tp{2}$ prime quotient below $\mu$. 
	Then $\alpha\vee \mu^*\prec \beta\vee \mu^*$ and 
	$\quotienttct{\alpha\vee\mu^*}{\beta\vee\mu^*}$
	has type $\tp{2}$.
\end{lemma}
\begin{proof}
	Since $\alpha\leq \beta$ and $\Con\ab{A}$ is modular,
	we have that $\beta\wedge (\alpha \vee \mu^*)=\alpha\vee(\mu^*\wedge \beta)\leq \alpha\vee (\mu^* \wedge \mu)$.
	Hence, Lemma~\ref{lemma:whymumeetmustariszero}
	yields that $\beta\wedge (\alpha \vee \mu^*)\leq \alpha$.
	Since $\alpha\leq \beta$, we have that 
	$\beta\wedge (\alpha \vee \mu^*)=\alpha$.
	Moreover, it is clear that $\beta\vee(\alpha \vee \mu^*)=\beta\vee \mu^*$.
	Thus 
	$\interval{\alpha}{\beta}\nearrow\interval{\alpha\vee\mu^*}{\beta\vee\mu^*}$.
	Hence, the statement follows from Lemma~\ref{lemma:exAic18Lemma3.4}
	and from the fact that in a modular lattice projective intervals are isomorphic. 
\end{proof}
\begin{theorem}\label{teor:main_result_on_homogeneous_sequences_on_quotients}
Let $\ab{A}$ be a finite Mal'cev algebra with (SC1),
let $n\in\N$,
let $(\mu_0,\dots, \mu_{n})$ be a homogeneous series 
in $\Con\ab{A}$, let $k\in\{1,2,3\}$,
and let us assume
that either $\ab{A}$ is congruence neutral
or the homogeneous series 
$(\mu_0,\dots, \mu_{n})$ is (CT$k$).
Let $\alpha\in\Con\ab{A}$. Then 
$\ab{A}/\alpha$ is (SC1). Moreover, 
either $\ab{A}/\alpha$ is congruence neutral 
or each homogeneous series $(\nu_0, \dots, \nu_{l})$ 
in $\Con(\ab{A}/\alpha)$
is (CT$k$).
\end{theorem}
\begin{proof}
If $\ab{A}$ is (SC1), then so is $\ab{A}/\alpha$, and
if $\ab{A}$ is congruence neutral, then 
so is $\ab{A}/\alpha$. 
Next, let $k\in\{1,2,3\}$,
and let us assume that $(\mu_0, \dots, \mu_{n})$
is (CT$k$), and
that $\ab{A}/\alpha$ is not congruence neutral, and 
let us fix a homogeneous series $(\nu_0', \dots, \nu_{l}')$
for $\Con(\ab{A}/\alpha)$. By the Correspondence Theorem
(cf.~\cite[Theorem~4.12]{MMT:ALVV})
there exist $\nu_1, \dots, \nu_l\in\Con\ab{A}$
such that for each $j\leq l$ we have 
$\nu_j'=\nu_j/\alpha$. Let $N$ be the height of $\Con\ab{A}$,
and let $\beta_0\prec\dots\prec\beta_{N}$ be a maximal chain 
that contains $\alpha, \nu_1, \dots, \nu_l$, and
let $\delta_0\prec\dots\prec\delta_{N}$ 
be a maximal chain that contains $\mu_0,\dots, \mu_{n}$. 
By the Dedekind Birkhoff Theorem 
(cf.~\cite[Theorem~2.37]{MMT:ALVV})
there exists $\sigma\in \Sym{N}$ such that
for each $i\in\finset{N}$ 
\[
\interval{\beta_{i-1}}{\beta_i}\leftrightsquigarrow
\interval{\delta_{\sigma(i)-1}}{\delta_{\sigma(i)}}.
\]
Thus, for all
$j\in\finset{l}$ and for all
$r\in\finset{N}$ with $\nu_{j-1}\prec\beta_{r}\leq\nu_j$,
there exists $i\in\finset{n}$ such that
$\mu_{i-1}\leq \delta_{\sigma(r)-1}\prec
\delta_{\sigma(r)}\leq \mu_i$
and $\interval{\nu_{j-1}}{\beta_r}\leftrightsquigarrow
\interval{\delta_{\sigma(r)-1}}{\delta_{\sigma(r)}}$.
Furthermore, since $\mu_i$ is homogeneous 
in $\interval{\mu_{i-1}}{\uno{A}}$,
there exists $s\in\finset{N}$ such that
$\delta_{s-1}=\mu_{i-1}$ and 
$\interval{\delta_{\sigma(r)-1}}{\delta_{\sigma(r)}}
\leftrightsquigarrow
\interval{\mu_{i-1}}{\delta_s}$. 
Thus, for all
$j\in\finset{l}$ and for all
$r\in\finset{N}$ with $\nu_{j-1}\prec\beta_{r}\leq\nu_j$,
there exists $s\in\finset{N}$ and $i\in\finset{n}$
such that
\[\mu_{i-1}=\delta_{s-1} \text{ and }
\interval{\nu_{j-1}}{\beta_r}\leftrightsquigarrow
\interval{\mu_{i-1}}{\delta_s}.
\]
Thus, Corollary~\ref{cor:projectivity_preserves_ext_types_and_dimension}
implies that $\quotienttct{\nu_{j-1}}{\nu_j}$
and $\quotienttct{\mu_{i-1}}{\mu_i}$ have the same extended
type. Moreover, if $[\mu_i,\mu_i]\leq \mu_{i-1}$ and the 
extended type of $\quotienttct{\mu_{i-1}}{\mu_i}$ is 
$(\tp{2}, q)$, then for each $a\in A$ we have that 
$((a/\mu_{i-1})/(\delta_s/\mu_{i-1}); +_{a/\mu_{i-1}})$
and 
$((a/\nu_{j-1})/(\beta_r/\nu_{j-1}); +_{a/\nu_{j-1}})$
are both 
vector spaces over $\GF{q}$ of the same dimension.
Next, we show that 
if $h_j$ is the height of $\interval{\nu_{j-1}}{\nu_j}$
and $h_i$ is the height of $\interval{\mu_{i-1}}{\mu_i}$,
then $h_j\leq h_i$. Seeking a contradiction, 
let us assume that there exists $t\in\finset{N}$ such that
$\nu_{j-1}\leq \beta_{t-1}\prec\beta_t\leq \nu_j$ and
$\quotienttct{\delta_{\sigma(t)-1}}{\delta_{\sigma(t)}}$ 
is not contained in $\interval{\mu_{i-1}}{\mu_i}$.
Then there exists $f\in\finset{n}$ such that
$\quotienttct{\delta_{\sigma(t)-1}}{\delta_{\sigma(t)}}$
is contained in $\interval{\mu_{f-1}}{\mu_f}$.
Let $\gamma$ be an atom of $\interval{\nu_{j-1}}{\nu_j}$.
Then there exists $s\in\finset{N}$ such that
\[
\interval{\delta_{\sigma(t)-1}}{\delta_{\sigma(t)}}\leftrightsquigarrow
\interval{\beta_{t-1}}{\beta_t}\leftrightsquigarrow
\interval{\nu_j}{\gamma}\leftrightsquigarrow
\interval{\mu_{i-1}}{\delta_s}.\]
Let $g$ be the minimum between $i$ and $f$. 
Then by \cite[Proposition~7.6]{AM:TOPC},
we have that
$\interval{\delta_{\sigma(t)-1}}{\delta_{\sigma(t)}}
\leftrightsquigarrow
\interval{\mu_{i-1}}{\delta_s}$ in 
$\interval{\mu_{g-1}}{\uno{A}}$.
This contradicts the definition of homogeneity. 
Hence we can conclude that $h_j\leq h_i$.

Thus, we can conclude 
that for all $j\in\finset{l}$ there exists $i\in\finset{n}$
such that if the quotient 
$\quotienttct{\mu_{i-1}}{\mu_i}$ satisfies one 
of the conditions 
\eqref{item:main_theorem_poly_rich_item_neutral},
\eqref{item:main_theorem_poly_rich_item_subtype2},
\eqref{item:main_theorem_poly_rich_item_subtype3},
\eqref{item:main_theorem_poly_rich_item_ABp},
\eqref{item:main_theorem_poly_rich_item_neutral_2},
\eqref{item:main_theorem_poly_rich_item_ABp_2},
\eqref{item:main_theorem_poly_rich_item_neutral_3},
or
\eqref{item:main_theorem_poly_rich_item_ABp_3},
then so does 
$\quotienttct{\nu_{j-1}}{\nu_j}$.
Therefore,
since the homogeneous series 
$(\mu_0, \dots, \mu_{n})$ is (CT$k$),
we infer that 
$(\nu_0', \dots, \nu_{l}')$ is (CT$k$). 
\end{proof}

\section{Remarks on a well known strategy for interpolating partial functions with polynomials}\label{sec:interpolation}
Let $\ab{A}$ be a Mal'cev algebra, 
let $k\in\N$, let $D\subseteq A^k$
let $T\subseteq D$, let $f\colon D\to A$,
and let $\mu\in \Con\ab{A}$.
We say that a polynomial function $p\in\POL\ari{k}\ab{A}$
\emph{interpolates $f$ on $T$ modulo $\mu$} if for all 
$\vec{t}\in T$ we have $p(\vec{t})\mathrel{\mu}f(\vec{t})$.
Moreover, we say that $p$ \emph{interpolates $f$ on $T$}
if $p$ interpolates $f$ on $T$ modulo $\bottom{A}$.
Finally, we say that \emph{$p$ interpolates $f$} if $p$
interpolates $f$ on $D$. 
The following lemma tells us how to combine  
``interpolation of functions that are constant modulo $\mu$''
and
``interpolation modulo $\mu$''
to interpolate a compatible function. 
This has been done in~\cite[Lemma~6]{KaaMay10} for total functions 
on finite set and in~\cite[Lemma~7.1]{Ros24} for partial functions of
finite domain on possibly infinite sets. 
\begin{lemma}[{\cite[Lemma~7.1]{Ros24}}]\label{lemma:fromcosetstotheinfinirtyandbeyond}
Let $\ab{A}$ be a Mal'cev algebra,
let $\mu\in \Con\ab{A}$, and
let $\strset{R}$ be a set of relations on $A$ such that
$\Con\ab{A}\subseteq \strset{R}\subseteq \Inv\POL\ab{A}$.
We assume that
every partial function with finite domain
that preserves $\strset{R}$ and 
has its image contained in one $\mu$-equivalence class can be interpolated by a polynomial 
function of $\ab{A}$.  
Let $k\in \N$, let $D\subseteq A^k$, let $T$ be a finite
subset of $D$, and let $f\colon D\to A$ be a
partial function that preserves $\strset{R}$. We assume 
that there exists $p_0\in \POL\ari{k}\ab{A}$ 
that interpolates $f$ on $T$ modulo $\mu$. 
Then there exists $p\in \POL\ari{k}\ab{A}$ that interpolates $f$ on $T$. 
\end{lemma}
Lemma~\ref{lemma:fromcosetstotheinfinirtyandbeyond} motivates the study of those 
partial functions whose image is contained inside one equivalence class of a congruence.
In this section we focus on 
interpolating functions that are constant modulo a 
homogeneous congruence. 
The following proposition from 
\cite{Ros24} is a generalization
to Mal'cev algebras of the results from
\cite[Section~7]{AI:PIIE} on homogeneous non-abelian ideals of
expanded groups, and a generalization 
to the case of homogeneous non-abelian atoms of \cite[Proposition~4.4]{AicBehRos24} on 
subdirectly irreducible Mal'cev algebras with a non-abelian monolith. 
\begin{proposition}[{\cite[Proposition~7.3]{Ros24}}]\label{prop:interpolation_in_a_class_nonAbelia_case}
Let $\ab{A}$ be a Mal'cev algebra whose congruence lattice has finite
height,
let $\mu$ be a 
homogeneous non-abelian element of $\Con\ab{A}$,
let $o\in A$, and let $k\in\N$. Then for every finite subset $T$
of $A^k$,
and for all $l\colon T\to o/\mu$ the following are equivalent:
\begin{enumerate}
\item $l$ is compatible;\label{item:compatibel:in:prop:interpolation_in_a_class_nonAbelia_case}
\item there exists $p_T\in\POL\ari{k}\ab{A}$ such that for all $\vec{t}\in T$
we have $p_T(\vec{t})=l(\vec{t})$ and $p_T(A^k)\subseteq o/\mu$. \label{item:interpolable:in:prop:interpolation_in_a_class_nonAbelia_case}
\end{enumerate}
\end{proposition}
For a congruence $\alpha$ of a (Mal'cev) algebra $\ab{A}$ and for 
$k\in\N$ we define
\[
\alpha^k:=\Biggl\{\begin{pmatrix}
\vec{a}\\
\vec{b}
\end{pmatrix}\in (A^k)^2\mid 
\forall i\leq k\colon (\vec{a}(i),\vec{b}(i))\in\alpha
\Biggr\}.
\]
The following proposition from 
\cite{Ros24} is a generalization
to Mal'cev algebras of the results from
\cite[Section~7]{AI:PIIE} on abelian homogeneous ideals of expanded groups.
\begin{proposition}[{\cite[Proposition~7.4]{Ros24}}]\label{prop:interpolating_on_cosets_abelian_case}
Let $\ab{A}$ be a finite Mal'cev algebra with (SC1), let $\mu$ be a
homogeneous abelian element of $\Con\ab{A}$ 
let $o\in A$, let $k\in\N$, let $T$ be a
finite subset of $A^k$, 
and let $f\colon T\to o/\mu$. Then the following are equivalent:
\begin{enumerate}
\item There exists $p_T\in\POL\ari{k}\ab{A}$ with $p_T(\vec{t})=f(\vec{t})$ for all 
$\vec{t}\in T$, and $p_T(A^k)\subseteq o/\mu$;\label{item:interpolating_polynomial_prop:interpolating_on_cosets_abelian_case}
\item for each $(\Phi(\mu):\mu)^k$-equivalence class of the form 
$\vec{v}/(\Phi(\mu):\mu)^k$ there exists $p_{\vec{v}}\in\POL\ari{k}\ab{A}$
such that $p_{\vec{v}}(\vec{t})=f(\vec{t})$ for all 
$\vec{t}\in T\cap (\vec{v}/(\Phi(\mu):\mu)^k)$. \label{item:interpolation-on_cosets_prop:interpolating_on_cosets_abelian_case}
\end{enumerate}   
\end{proposition}
\begin{proposition}\label{prop:laproposizione_dell_f_uno}
	Let $\ab{A}$ be a finite Mal'cev algebra with (SC1),
	let $\mu$ be a homogeneous abelian congruence of $\ab{A}$,
	let $o\in A$, let $k\in\N$, let $\vec{v}\in A^k$,
	let $T\subseteq \vec{v}/(\Phi(\mu):\mu)^k$, 
	let $f\colon T\to o/\mu$ be a congruence preserving (type-preserving)
	function. 
	Let $f_1\subseteq \vec{v}/\mu \times o/\mu$ be
	defined by 
		\[
	\begin{split}
		f_1:=\{&(\vec{u}, f(d(u_1,v_1, u_1^*), \dots, d(u_k, v_k, u_k^*)))
		\mid \\
		&\vec{u}\in \vec{v}/\mu^k \text{ and } \exists 
		\vec{u}^*\in \vec{v}/(\mu^*)^k\colon (d(u_1,v_1, u_1^*), \dots, d(u_k, v_k, u_k^*))\in T\}.
	\end{split}
	\]
	Then $f_1$ is functional and congruence preserving (type-preserving). 
	Furthermore, if there exists $p\in\POL\ari{k}\ab{A}$ that
	interpolates $f_1$ on its domain and has its image contained in $o/\mu$, then 
	for all $\vec{t}\in T$ we have $p(\vec{t})=f(\vec{t})$. 
\end{proposition}
\begin{proof}
First, we remark that
Proposition~\ref{prop:homogenuityandcolpementation}\eqref{item:ex_if_not_neutral_then_abelian_for_hom}
implies that
$(\Phi(\mu):\mu)=\mu\vee \mu^*$.
Next, we prove that $f_1$ is functional. To this end, let
	$\vec{u}\in \vec{v}/\mu^k$ and let $\vec{a}^*, \vec{b}^*\in \vec{v}/(\mu^*)^k$ be such that 
	\[
	\begin{pmatrix}
		d(u_1, v_1, a_1^*)\\
		\vdots \\
		d(u_k, v_k, a_k^*)
	\end{pmatrix},\,
	\begin{pmatrix}
		d(u_1, v_1, b_1^*)\\
		\vdots \\
		d(u_k, v_k, b_k^*)
	\end{pmatrix}\in T.\]
	Then by compatibility, we have that 
	\[
	f(d(u_1, v_1, a_1^*), \dots d(u_k,v_k, a_k^*))\mathrel{\mu^*}
	f(d(u_1, v_1, b_1^*), \dots d(u_k,v_k, b_k^*)).\]
	Since $f(T)\subseteq o/\mu$, we have that 
	\[
	f(d(u_1, v_1, a_1^*), \dots d(u_k,v_k, a_k^*))\mathrel{\mu\wedge \mu^*}
	f(d(u_1, v_1, b_1^*), \dots d(u_k,v_k, b_k^*)), \]
	and therefore, Lemma~\ref{lemma:whymumeetmustariszero} yields
	\[
	f(d(u_1, v_1, a_1^*), \dots d(u_k,v_k, a_k^*))=
	f(d(u_1, v_1, b_1^*), \dots d(u_k,v_k, b_k^*)). \]
	
	Next, we show that $f_1$ is compatible. 
	To this end, let $\theta\in\Con\ab{A}$, and 
	let $\vec{u}\equiv_\theta \vec{w}$ with $\vec{u}, \vec{w}$ in
	the domain of $f_1$. 
	We show that $f_1(\vec{u})\mathrel{\theta}f_1(\vec{w})$.
	Let us define \[
	\tilde{\theta}:=\Cg{\{(u_1, w_1), \dots, (u_k, w_k)\}}.
	\]
	Since $\vec{u}$ and $\vec{w}$ are in the domain of $f_1$,
	there exist $\vec{u}^*, \vec{w}^*\in \vec{v}/(\mu^*)^k$ such that 
	\[
	\begin{pmatrix}
		d(u_1, v_1, u_1^*)\\
		\vdots\\
		d(u_k, v_k, u_k^*)
	\end{pmatrix},\, 
	\begin{pmatrix}
		d(w_1, v_1, w_1^*)\\
		\vdots\\
		d(w_k, v_k, w_k^*)
	\end{pmatrix}\in T.
	\]
	Furthermore, we have 
	\[
	\begin{pmatrix}
		d(u_1, v_1, u_1^*)\\
		\vdots\\
		d(u_k, v_k, u_k^*)
	\end{pmatrix} \equiv_{\tilde{\theta}} 
	\begin{pmatrix}
		d(w_1, v_1, u_1^*)\\
		\vdots\\
		d(w_k, v_k, u_k^*)
	\end{pmatrix} \equiv_{\mu^*}
	\begin{pmatrix}
		d(w_1, v_1, w_1^*)\\
		\vdots\\
		d(w_k, v_k, w_k^*)
	\end{pmatrix}
	\]
	Thus, since $f$ is compatible and $f(T)\subseteq o/\mu$
	\[
	\begin{split}
		f_1(\vec{u})=&f(d(u_1, v_1, u_1^*), \dots, d(u_k, v_k, u_k^*))
		\equiv_{(\tilde{\theta}\vee \mu^*)\wedge \mu}\\
		&f(d(w_1, v_1, w_1^*), \dots, d(w_k, v_k, w_k^*))=f_1(\vec{w}).
	\end{split}
	\]
	Since $\tilde{\theta}\leq\mu$ and $\Con\ab{A}$ is modular,
	Lemma~\ref{lemma:whymumeetmustariszero} yields
	\[(\tilde{\theta}\vee\mu^*)\wedge\mu=\tilde{\theta}\vee(\mu^*\wedge \mu)=
	\tilde{\theta}\vee\bottom{A}=\tilde{\theta}\subseteq \theta.\]
	Thus, $f_1(\vec{u})\mathrel{\theta}f_1(\vec{w})$. 
	
	Next, we show that $f_1$ is type-preserving (under the assumption that
	$f$ is type-preserving). 
	To this end, let us fix $\alpha, \beta\in \Con\ab{A}$ with
	$\alpha\prec\beta\leq \mu$ and $[\beta, \beta]\leq \alpha$,
	and let us fix $\vec{x}, \vec{y}, \vec{z}, \vec{w}$ in the 
	domain of $f_1$ such that 
	$\vec{x}\equiv_\beta \vec{y}\equiv_\beta \vec{z}$
	and 
	\[
	\begin{pmatrix}
		d(x_1,y_1, z_1)\\
		\vdots \\
		d(x_k, y_k, z_k)
	\end{pmatrix}\equiv_\alpha \vec{w}.
	\]
	The fact that $\card{\{f_1(\vec{x}), f_1(\vec{y}), f_1(\vec{z})\}/\beta}=1$
	follows from the already proved compatibility of 
	$f_1$. 
	We show that $d(f_1(\vec{x}), f_1(\vec{y}), f_1(\vec{z}))\mathrel{\alpha} f_1(\vec{w})$. 
	Since 
	$(x_i, y_i, z_i, w_i)\in \rho(\alpha, \beta)$ for all $i\in\finset{k}$ and 
	since $\vec{x}, \vec{y}, \vec{z}, \vec{w}\in \dom f_1$,
	there exist $\vec{x}^*,\vec{y}^*,\vec{z}^*,\vec{w}^* \in \vec{v}/(\mu^*)^k$
	such that for all $i\in\finset{k}$ we have 
	\[
	\begin{pmatrix}
		d(x_i, v_i, x_i^*)\\d(y_i, v_i, y_i^*)\\d(z_i, v_i, z_i^*)\\
		d(w_i, v_i, w_i^*)
	\end{pmatrix}\in \rho(\alpha\vee \mu^*, \beta\vee\mu^*)
	\]
	and 
	\[
	\begin{pmatrix}
		d(x_1, v_1, x_1^*)\\
		\vdots \\
		d(x_k, v_k, x_k^*) 
	\end{pmatrix},\,
	\begin{pmatrix}
		d(y_1, v_1, y_1^*)\\
		\vdots \\
		d(y_k, v_k, y_k^*) 
	\end{pmatrix},\,
	\begin{pmatrix}
		d(z_1, v_1, z_1^*)\\
		\vdots \\
		d(z_k, v_k, z_k^*) 
	\end{pmatrix},\,
	\begin{pmatrix}
		d(w_1, v_1, w_1^*)\\
		\vdots \\
		d(w_k, v_k, w_k^*) 
	\end{pmatrix}\in \dom f
	\]
Therefore, since $f$ is type-preserving,
Lemma~\ref{lemma:projecting_belowo_mu_star} 
implies that
	\[
	\begin{pmatrix}
		f(d(x_1, v_1, x_1^*), \dots, d(x_k, v_k, x_k^*))\\
		f(d(y_1, v_1, y_1^*), \dots, d(y_k, v_k, y_k^*))\\
		f(d(z_1, v_1, z_1^*), \dots, d(z_k, v_k, z_k^*))\\
		f(d(w_1, v_1, w_1^*), \dots, d(w_k, v_k, w_k^*))
	\end{pmatrix}\in \rho(\alpha\vee \mu^*, \beta\vee\mu^*).
	\]
	Thus, we have that 
	\[\begin{split}
		d\bigl( &f (d(x_1, v_1, x_1^*), \dots, d(x_k, v_k,x_k^*)), \\
		&f (d(y_1, v_1, y_1^*), \dots, d(y_k, v_k, y_k^*)),\\
		&f (d(z_1, v_1, z_1^*), \dots, d(z_k, v_k, z_k^*))
		\bigr)\equiv_{\alpha\vee \mu^*}\\
		f&(d(w_1, v_1, w_1^*), \dots, d(w_k, v_k, w_k^*)).
	\end{split}
	\]
Since $\alpha\leq \mu$ and $\Con\ab{A}$ is modular, Lemma~\ref{lemma:whymumeetmustariszero} implies that
	\[(\alpha \vee \mu^*)\wedge \mu=\alpha.
	\]
Since $f$ is constant modulo $\mu$ we have that
	\[
	\begin{split}
		d(&f_1(\vec{x}), f_1(\vec{y}), f_1(\vec{z})) =\\
		d\bigl( &f (d(x_1, v_1, x_1^*), \dots, d(x_k, v_k,x_k^*)), \\
		&f (d(y_1, v_1, y_1^*), \dots, d(y_k, v_k, y_k^*)),\\
		&f (d(z_1, v_1, z_1^*), \dots, d(z_k, v_k, z_k^*))
		\bigr)\equiv_{\alpha}\\
		f&(d(w_1, v_1, w_1^*), \dots, d(w_k, v_k, w_k^*))=
		f_1(\vec{w}).
	\end{split}
	\]
Thus, $f_1$ preserves $\rho(\alpha, \beta)$.
In order to prove that $f_1$ is type-preserving, 
we need to show that $f_1$ preserves the relations 
of the form $\rho(\alpha, \beta)$ when 
$\quotienttct{\alpha}{\beta}$ is a type $\tp{2}$
quotient that is not contained in 
$\interval{\bottom{A}}{\mu}$. 
To this end, let us fix $\alpha, \beta\in \Con\ab{A}$ 
with the above discussed properties, and let 
$\vec{x}, \vec{y}, \vec{z}, \vec{w}\in \dom f_1$ 
such that for all $i\in \finset{k}$ we have
$(x_i, y_i, z_i, w_i)\in \rho(\alpha, \beta)$. 
Then since $\dom f_1\subseteq (o/\mu)^k$, 
for each $i\in \finset{k}$, we have 
$x_i\mathrel{\mu} y_i\mathrel{\mu} z_i \mathrel{\mu} w_i$.
Thus, 
either $\mu\leq \alpha\prec \beta$ or 
$\alpha\wedge \mu\prec \beta\wedge \mu$ and
$(x_i, y_i, z_i, w_i)\in \rho(\alpha\wedge \mu, \beta\wedge \mu)$
for each $i\in\finset{k}$. 
If $\mu\leq \alpha\prec \beta$, then 
$\mu\times \mu\subseteq \rho(\alpha, \beta)$ 
and since for each $i\in \finset{k}$ we have 
$x_i\mathrel{\mu} y_i\mathrel{\mu} z_i \mathrel{\mu} w_i$,
the fact that
\[(f_1(\vec{x}), f(\vec{y}), 
f(\vec{z}), f(\vec{w}))\in \rho(\alpha, \beta)\]
follows from the already proved compatibility of $f_1$.
If $\alpha\wedge \mu\prec \beta\wedge \mu$ and
$(x_i, y_i, z_i, w_i)\in \rho(\alpha\wedge \mu, \beta\wedge \mu)$
for each $i\in\finset{k}$,
then since $f_1$ preserves $\rho(\gamma, \delta)$ 
for each prime quotient $\quotienttct{\gamma}{\delta}$
contained in $\interval{\bottom{A}}{\mu}$, we have
\[(f_1(\vec{x}), f(\vec{y}), 
f(\vec{z}), f(\vec{w}))\in \rho(\alpha\wedge \mu, \beta\wedge \mu)
\subseteq \rho(\alpha, \beta).
\]
This concludes the proof that $f_1$ is type-preserving. 
	
Next, we assume that
there exists $p\in\POL\ari{k}\ab{A}$
that interpolates $f_1$ on its domain and 
has its image contained in $o/\mu$, and 
we show that for all $\vec{t}\in T$ 
we have $f(\vec{t})=p(\vec{t})$. 
To this end, let us fix $\vec{t}\in T$. 
Since $\interval{\bottom{A}}{\mu}\nearrow\interval{\mu^*}{\mu\vee\mu^*}$,
Lemma~\ref{lemma:conseguenze_permutabilita_intervalli_proiettivi_esistenza,d(b,o,c)} implies that there exist
$\vec{u}\in \vec{v}/\mu^k$ and $\vec{u}^*\in \vec{v}/(\mu^*)^k$
such that 
\[
\vec{t}=\begin{pmatrix}
	d(u_1, v_1, u_1^*)\\
	\vdots \\
	d(u_k, v_k, u_k^*)
\end{pmatrix}.
\]
Thus, $f(\vec{t})=f_1(\vec{u})=p(\vec{u})$.
Since for each $i\in\finset{k}$ we have $u_i=d(u_i, v_i, v_i)$ and 
$u_i^*\mathrel{\mu^*}v_i$ we have that 
\[p(\vec{u})\equiv_{\mu^*} p(d(u_1, v_1, u_1^*), \dots, d(u_k, v_k, u_k^*))= p(\vec{t}),
\]
and therefore $f(\vec{t})\mathrel{\mu^*}p(\vec{t})$. 
Since $f(\vec{t})\mathrel{\mu} o$ and $p(A^k)\subseteq o/\mu$, 
we have that $f(\vec{t})\mathrel{\mu} p(\vec{t})$. 
Thus, Lemma~\ref{lemma:whymumeetmustariszero} yields 
$p(\vec{t})=f(\vec{t})$. 	
\end{proof}

\section{Partial type-preserving functions that are constant modulo an homogeneous abelian congruence with (AB$p$)}
\begin{proposition}\label{prop:interpolazione_type_preserving_arity_k_caso_ABP}
Let $\ab{A}$ be a finite Mal'cev algebra with (SC1), 
let $o\in A$, let $k\in \N$, let $p$ be a prime number, and let $\mu$
be a homogeneous abelian congruence of $\ab{A}$. 
Let us assume that $(\ab{A}, \mu)$ satisfies 
(AB$p$) and that $o/\mu$ is 
strictly $k$-polynomially rich with respect to $\ab{A}$. 
Let $\vec{v}\in A^k$, let $T\subseteq \vec{v}/(\Phi(\mu):\mu)^k$, and
let $f\colon T\to o/\mu$ be a type-preserving function.
Then $f$ can be interpolated on $T$ by a polynomial whose image is 
contained in $o/\mu$.
\end{proposition}
\begin{proof}
First, we observe that if $\card{o/\mu}=1$, then the constant 
	polynomial function with constant
	value $o$ interpolates $f$ on $T$ and has its image contained 
	in $o/\mu$. 
	
	In the remainder of the proof 
	we investigate the case $\card{o/\mu}>1$. 
	To this end, let us define $f_1\subseteq \vec{v}/\mu^k \times o/\mu$ as
	\[
	\begin{split}
	f_1:=\{&(\vec{u}, f(d(u_1,v_1, u_1^*), \dots, d(u_k, v_k, u_k^*)))
    \mid \\
	&\vec{u}\in \vec{v}/\mu^k \text{ and } \exists 
	\vec{u}^*\in \vec{v}/(\mu^*)^k\colon (d(u_1,v_1, u_1^*), \dots, d(u_k, v_k, u_k^*))\in T\}.
	\end{split}
	\]
By Proposition~\ref{prop:laproposizione_dell_f_uno}
$f_1$ is functional and type-preserving. 

Next, we show that there exists $p\in\POL\ari{k}\ab{A}$
that interpolates $f_1$ on its domain and 
has its image contained in $o/\mu$.
We split the proof into two cases:\\
\textbf{Case 1}: \emph{$\card{\vec{v}/\mu^k}=1$}: 
In this case the constant function with constant value $f_1(\vec{v})$
interpolates $f_1$ on its domain and has its image contained in $o/\mu$. \\
\textbf{Case 2}: \emph{$\card{\vec{v}/\mu^k}>1$}: 
In this case \cite[Theorem~6.13]{Ros24}
implies that for each $i\in\finset{k}$ if $\card{v_i/\mu}>1$, then 
$o/\mu$ and $v_i/\mu$ are polynomially isomorphic.
Hence for each such $i$ there exist $t_{v_i \rightarrow o}, t_{o \rightarrow v_i}\in \POL\ari{1}\ab{A}$,
such that $t_{v_i \rightarrow o}\restrict{v_i/\mu}$ is an isomorphism  
between
$\ab{A}\restrict{v_i/\mu}$  
and $\ab{A}\restrict{o/\mu}$, and
$(t_{v_i \rightarrow o}\restrict{v_i/\mu}))^{-1}=t_{o \rightarrow v_i}\restrict{o/\mu}$. 
Thus, since $\card{\vec{v}/\mu^k}>1$, 
there exists $m\in \finset{k}$ and 
an injective function $\tau\colon \finset{m}\to\finset{k}$ such
for each $j\in\finset{k}$ the cardinality of
$v_j/\mu$ is a singleton if and only if 
$j\notin \{\tau(1), \dots, \tau(m)\}$. 
For each $j\in \finset{k}$ we define a map 
$t_j\colon (o/\mu)^m\to v_j/\mu$ as follows.
\[
t_j(x_1, \dots, x_m)=\begin{cases}
	v_j & \text{ if } \card{v_j/\mu}=1\\
	t_{o\rightarrow v_j}(x_{\tau^{-1}(j)}) & \text{ if } \card{v_j/\mu}>1.
\end{cases}
\]
We let 
\[T'=\{\vec{a}\in (o/\mu)^m \mid (t_1(\vec{a}), \dots, t_k(\vec{a}))\in 
\dom f_1\},
\]
and we define $f_2\colon T'\to o/\mu$ as $f_1(t_1, \dots, t_k)$. 
The function $f_2$ is a partial $m$-ary type-preserving function as it is 
a composition of type-reserving functions. 
Thus, since $o/\mu$ is
strictly $k$-polynomially rich 
with respect to $\ab{A}$,
and $m\leq k$, 
there exists $q\in \POL\ari{m}\ab{A}$
such that $q$ interpolates $f_2$ on its domain and 
$q(\vec{x})\in o/\mu$ for each $\vec{x}\in (o/\mu)^m$. Let 
$e_\mu^o$ be the unary idempotent polynomial 
with target set $o/\mu$ built in 
\cite[Theorem~6.12]{Ros24}, and
let us define $p\colon A^k\to A$ by
\[p(\vec{x}):=e_\mu^o (q(t_{v_{\tau(1)}\rightarrow o}(x_{\tau(1)}), \dots,
t_{v_{\tau(m)\rightarrow o}}(x_{\tau(m)}))).
\] 
The fact that the image of $p$
is contained in $o/\mu$ its an easy consequence of its definition.
Next, we show that for each $\vec{x}\in \dom f_1$ we have 
$f_1(\vec{x})= p(\vec{x})$. 
To this end, let $\vec{x}\in \dom f_1$. 
For each $j\in\finset{k}$ we have that
\[
t_j(t_{v_{\tau(1)}\rightarrow o}(x_{\tau(1)}), 
\dots, t_{v_{\tau(m)}\rightarrow o}(x_{\tau(m)}))=
\begin{cases}
	v_j &\text{ if } \card{v_j/\mu}=1\\
	t_{o\rightarrow v_j}(t_{v_j\rightarrow o}(x_j))=x_j &\text{ if } \card{v_j/\mu}>1.
\end{cases}
\]
Thus, since $\vec{x}\in \dom f_1\subseteq \vec{v}/\mu^k$, we have that
for each $j\in\finset{k}$ 
\begin{equation}\label{eq:equazione_valore_dei_t_j_composti_t_freccia}
t_j(t_{v_{\tau(1)}\rightarrow o}(x_{\tau(1)}), 
\dots, t_{v_{\tau(m)}\rightarrow o}(x_{\tau(m)}))=x_j.
\end{equation}
Then we have 
\[
\begin{split}
&p(x_1, \dots, x_k)=\\
&e_\mu^o (q(t_{v_{\tau(1)}\rightarrow o}(x_{\tau(1)}), \dots, 
t_{v_{\tau(m)\rightarrow o}}(x_{\tau(m)})))=\\
&e_\mu^o (f_2(t_{v_{\tau(1)}\rightarrow o}(x_{\tau(1)}), \dots,
t_{v_{\tau(m)\rightarrow o}}(x_{\tau(m)})))=\\
&e_\mu^o(f_1(t_1(t_{v_{\tau(1)}\rightarrow o}(x_{\tau(1)}), \dots,
t_{v_{\tau(m)\rightarrow o}}(x_{\tau(m)})),
\dots, \\
&t_k(t_{v_{\tau(1)}\rightarrow o}(x_{\tau(1)}), \dots,
t_{v_{\tau(m)\rightarrow o}}(x_{\tau(m)}))
)).
\end{split}
\]
Hence \eqref{eq:equazione_valore_dei_t_j_composti_t_freccia}
implies that $p(\vec{x})=e_\mu^o(f_1(x_1, \dots, x_k))$, which equals 
$f_1(\vec{x})$ as $e_\mu^o$ is the identity on $o/\mu$. 

Thus, we can conclude that 
there exists $p\in\POL\ari{k}\ab{A}$
that interpolates $f_1$ on its domain and 
has its image contained in $o/\mu$.
Hence the result follows from Proposition~\ref{prop:laproposizione_dell_f_uno}.
\end{proof}

\begin{proposition}\label{prop:interpolating_partial_functions_type_preserving_inside_a_congruence_class_abelian_case_with_ABp}
	Let $\ab{A}$ be a finite Mal'cev algebra with (SC1),
	let $o\in A$, let $p$ be a prime, 
	let $k\in\N$,
	let $\mu$ be a homogeneous abelian congruence of $\ab{A}$
	such that $(\ab{A},\mu)$ satisfies (AB$p$),
	and $o/\mu$
	is strictly $k$-polynomially rich with respect to $\ab{A}$.
	Let $T\subseteq A^k$, and let $f\colon T\to o/\mu$
	be a type-preserving function. 
	Then there exists $p\in\POL\ari{1}\ab{A}$
	whose image is contained in $o/\mu$
	that interpolates $f$ on $T$. 
\end{proposition}
\begin{proof}	
For each class $\vec{v}/(\Phi(\mu):\mu)^k$ of 
$T/(\Phi(\mu):\mu)^k$
Proposition~\ref{prop:interpolazione_type_preserving_arity_k_caso_ABP}
yields a polynomial with image contained in $o/\mu$
that interpolates $f$ on $\vec{v}/(\Phi(\mu):\mu)^k$.
Thus, Proposition~\ref{prop:interpolating_on_cosets_abelian_case}
yields the desired polynomial function. 
\end{proof}

\section{Partial type-preserving functions that are constant modulo an homogeneous abelian atom}

\begin{proposition}\label{prop:interpolation_in_the_simple_type_two_case}
	Let $\ab{A}$ be a finite Mal'cev algebra with (SC1), 
	let $\mu$ be a homogeneous abelian atom of $\Con\ab{A}$,
	let $l$ be the number of $\mu$-classes,
	let $v_1\dots, v_l$ be a transversal through the $\mu$-classes,
	let $n_1, \dots, n_l, q\in \N$ be such that 
	\[(\forall i\leq l\colon n_i\in\{0,1,2,3\}\text{ and }q=2) \text{ or } 
	(\forall i\leq l\colon n_i\in\{0,1,2\}\text{ and } q=3),\]
	$q$ is the subtype of $\quotienttct{\bottom{A}}{\mu}$,
	and for each $i\leq l$ we have $\card{v_i/\mu}=q^{n_i}$. 
	Then for all $i, j\in \finset{l}$ and for each $T\subseteq v_i/\mu$
	and for each type-preserving $f\colon T\to v_j/\mu$
	such that $f(v_i)=v_j$ there exists $p\in\POL\ari{1}\ab{A}$
	that interpolates $f$ on $T$. 
\end{proposition}
\begin{proof}
Let $i,j\in\finset{l}$ 
and let $\epsilon_\mu^{v_i}$ and $\epsilon_\mu^{v_j}$ be the 
	group homomorphisms built in Proposition
	\ref{prop:identification_with_matrices},
	and let $\hat{f}:=\epsilon_\mu^{v_j}\circ f \circ (\epsilon_\mu^{v_i})^{-1}$.
	We show that $\hat{f}$ can be interpolated by a linear map 
	from $\GF{q}^{n_i}$ to $\GF{q}^{n_j}$.
	
	First, we observe that the claim holds trivially 
	if one among $n_j$ and $n_i$ is $0$.
	Next, we assume that both are strictly greater than $0$. 
	Then the assumptions of Lemma 
	\ref{lemma:technical_lemma_kapl} are satisfied and 
	there exists a set $S\subseteq \epsilon_\mu^{v_i}(T)$ 
	that consists of linearly independent vectors, such that
	for each $t\in T$ there exists $x_1, x_2, x_3\in S\cup \{\epsilon_\mu^{v_i}(v_i)\}$
	such that $\epsilon_\mu^{v_i}(t)=x_1-x_2+x_3$. 
	
	Let $M$ be the matrix associated to the linear map from the subspace generated by $S$ into 
	$\GF{q}^{n_j}$ defined naturally by the values of $\hat{f}$ on $S$. 
	We show that for each $t\in T$ we have 
	$\hat{f}(\epsilon_\mu^{v_i}(t))=M\cdot \epsilon_\mu^{v_i}(t)$.
	To this end, let $x_1, x_2, x_3\in S\cup \{\epsilon_\mu^{v_i}(v_i)\}$
	such that $\epsilon_\mu^{v_i}(t)=x_1-x_2+x_3$. 
	Then 
	\[\begin{pmatrix}
		(\epsilon_\mu^{v_i})^{-1}(x_1)\\
		(\epsilon_\mu^{v_i})^{-1}(x_2)\\
		(\epsilon_\mu^{v_i})^{-1}(x_3)\\
		t
	\end{pmatrix}\in \rho(\bottom{A}, \mu).
	\]
	Thus 
	\[\begin{pmatrix}
		f((\epsilon_\mu^{v_i})^{-1}(x_1))\\
		f((\epsilon_\mu^{v_i})^{-1}(x_1))\\
		f((\epsilon_\mu^{v_i})^{-1}(x_3))\\
		f(t)
	\end{pmatrix}\in \rho(\bottom{A}, \mu).
	\]
	Hence we have that 
	\[
	f(t)=d(f((\epsilon_\mu^{v_i})^{-1}(x_1)), f((\epsilon_\mu^{v_i})^{-1}(x_1)), f((\epsilon_\mu^{v_i})^{-1}(x_3))).
	\]
Moreover, by \cite[Proposition 2.6]{Aic06} we have that
\[
\begin{split}
&d(f((\epsilon_\mu^{v_i})^{-1}(x_1)), f((\epsilon_\mu^{v_i})^{-1}(x_1)),
 f((\epsilon_\mu^{v_i})^{-1}(x_3)))=\\
&d(d(f((\epsilon_\mu^{v_i})^{-1}(x_1)), f((\epsilon_\mu^{v_i})^{-1}(x_1)),
 v_i), v_i, f((\epsilon_\mu^{v_i})^{-1}(x_3))),
\end{split}
\]
and therefore,
	\[
	\begin{split}
		\hat{f}(\epsilon_\mu^{v_i}(t))=&\epsilon_\mu^{v_j}(f(t))=\\
		&\epsilon_\mu^{v_j}( d(f((\epsilon_\mu^{v_i})^{-1}(x_1)), f((\epsilon_\mu^{v_i})^{-1}(x_1)), f((\epsilon_\mu^{v_i})^{-1}(x_3))))=\\
		&\epsilon_\mu^{v_j}(f((\epsilon_\mu^{v_i})^{-1}(x_1)))- \epsilon_\mu^{v_j}(f((\epsilon_\mu^{v_i})^{-1}(x_2)))
		+\epsilon_\mu^{v_j}(f((\epsilon_\mu^{v_i})^{-1}(x_3)))=\\
		&M\cdot x_1- M\cdot x_2 +M\cdot x_3= M\cdot (x_1-x_2+x_3)=\\
		&M\cdot \epsilon_\mu^{v_i}(t). 
	\end{split}
	\]	
	By the argument reported in \cite[Section~4]{Fre83}
	and Theorem~\ref{teor:from_coordinatization_to_TCT}
	we have that the set 
	\[
	\text{Hom}(v_i, v_j, \mu):=\{p\in \POL\ari{1}\ab{A}\mid p(v_i)=v_j\}
	\]
	is a dense set of linear transformations 
	from $\GF{q}^{n_i}$ to $\GF{q}^{n_j}$. 
	Hence,
	there exists a polynomial $p\in \text{Hom}(v_i, v_j, \mu)$
	that interpolates $f$ on its domain. 
\end{proof}
\begin{proposition}\label{prop:interpolazione_coset_centralizzatore_caso_type_preserving_simple}
	Let $\ab{A}$ be a finite Mal'cev algebra with (SC1), 
	let $\mu$ be a homogeneous abelian atom of $\Con\ab{A}$,
	let $q$ be the subtype of $\quotienttct{\bottom{A}}{\mu}$,
	let us assume that $q\in\{2,3\}$ and that
	if $q=2$, then the cardinality of each $\mu$-class belongs to
	$\{1,2,4,8\}$ and if 
	$q=3$, then the cardinality of each $\mu$-class belongs to
	$\{1,3,9\}$,
	let $o\in A$, let $T\subseteq A$ with $\card{T/(\Phi(\mu):\mu)}=1$,
	and let $f\colon T\to o/\mu$ be a type-preserving partial 
	function. Then there exists $p\in\POL\ari{1}\ab{A}$
	that interpolates $f$ on $T$ and is constant modulo $\mu$.
\end{proposition}
\begin{proof}
	Let $v\in T$ be fixed. 
	Let us define $f_1\subseteq v/\mu\times o/\mu$ by
	\[
	f_1:=\{(u, f(d(u, v, u^*)))\in A^2 \mid u\in v/\mu \text{ and } \exists u^*\in v/\mu^* \colon d(u, v, u^*)\in T\}.
	\]
	By Proposition~\ref{prop:laproposizione_dell_f_uno}
	$f_1$ is functional and type-preserving. 
	Let $f_2\colon \dom f_1\to o/\mu$ be defined by 
	$x\mapsto d(f_1(x), f_1(v), o)$. 
	The function $f_2$ satisfies the assumptions of 
	Proposition~\ref{prop:interpolation_in_the_simple_type_two_case}.
	Therefore, there exists $p_2\in \POL\ari{1}\ab{A}$ such that
	$p_2\restrict{\dom f_1}=f_2$.
	Let us define $p_1$ as the function that maps each $x\in A$ to
	$d(p_2(x), o,f_1(v))$. Then \cite[Proposition~2.6]{Aic06} implies 
	that for all $x\in\dom f_1$:
	\[
	\begin{split}
		p_1(x)=& d(p_2(x), o, f_1(v))=\\
		&d(d(f_1(x), f_1(v), o), d(o,o,o), d(o,o,f_1(v))\mathrel{[\mu, \mu]}\\
		&d(d(f_1(x), o,o), d(f_1(v), o,o), d(o,o,f_1(v))=\\
		&f_1(x).
	\end{split}
	\] 
	Thus, since $[\mu, \mu]=\bottom{A}$
	we have that $p_1\restrict{\dom f_1}=f_1$. 
	
	Next, let us define $p:=e_\mu^o\circ p_1$, where
	$e_\mu^o$ is the idempotent polynomial constructed in 
	\cite[Theorem~6.12]{Ros24}.
	Then $p$ interpolates $f_1$ on its domain and has its image contained
	in $o/\mu$. Hence the result follows from
	Proposition~\ref{prop:laproposizione_dell_f_uno}. 
\end{proof}
\begin{proposition}\label{prop:interpolation_of_partial_type_preserving_in_class_of_atoms}
	Let $\ab{A}$ be a finite Mal'cev algebra with (SC1),
	let $\mu$ be a homogeneous abelian atom of $\Con\ab{A}$.
	Let us assume that
	the subtype $q$ of $\quotienttct{\bottom{A}}{\mu}$
	is either $2$ or $3$, and if $q=2$, then each $\mu$-class 
	has cardinality $1$, $2$, $4$, or $8$, and 
	if $q=3$, then each $\mu$-class has cardinality $1$, $3$, or $9$. 
	Let $T\subseteq A$, let $o\in A$, and let $f\colon T\to o/\mu$ be a
	type-preserving function. 
	Then there exists $p\in\POL\ari{1}\ab{A}$ constant modulo $\mu$
	and that interpolates $f$. 
\end{proposition}
\begin{proof}
For each class $v/(\Phi(\mu):\mu)$ of $T/(\Phi(\mu):\mu)$
Proposition~\ref{prop:interpolazione_coset_centralizzatore_caso_type_preserving_simple}
yields a polynomial with image contained in $o/\mu$
that interpolates $f$ on $v/(\Phi(\mu):\mu)$.
Thus, Proposition~\ref{prop:interpolating_on_cosets_abelian_case}
yields the desired polynomial function. 
\end{proof}

\section{Interpolation of partial functions that preserve types}
\begin{proposition}\label{prop:interpolation_on_quotient_yields_interpolation_case_neutral}
	Let $\ab{A}$ be a finite Mal'cev algebra with (SC1),
	let $\mu$ be an homogeneous congruence of $\ab{A}$
	with $[\mu, \mu]=\mu$, let $k\in\N$,
	let $T$ be a finite subset of $A^k$ and let $f\colon T\to A$
	be a partial type-preserving function that can be interpolated
	modulo $\mu$ by a polynomial function of $\ab{A}$. 
	Then $f$ can be interpolated by a polynomial 
	function of $\ab{A}$. 
\end{proposition}
\begin{proof}
	Since $[\mu, \mu]=\mu$
	Proposition~\ref{prop:interpolation_in_a_class_nonAbelia_case}
	implies that every partial function that preserves
	congruences and types and has its image contained in one $\mu$.class
	can be interpolated by a polynomial function of $\ab{A}$
	whose image is contained in the same $\mu$-class.
	Thus, the assumptions of Lemma~\ref{lemma:fromcosetstotheinfinirtyandbeyond}
	are satisfied, and $f$ can be interpolated by a polynomial function. 
\end{proof}
\begin{proposition}\label{prop:interpolation_on_quotient_yields_interpolation_case_type_preserving_ABp}
	Let $\ab{A}$ be a finite Mal'cev algebra  with (SC1),
	let $\mu$ be an homogeneous abelian congruence of $\ab{A}$
	such that $(\ab{A}, \mu)$ satisfies (AB$p$) for $p\in\{2,3,5\}$
	and $\interval{\bottom{A}}{\mu}$ has height $m$.
	Let $T\subseteq A$ and let $f\colon T\to A$
	be a partial type-preserving function that can be interpolated
	modulo $\mu$ by a polynomial function of $\ab{A}$. 
	Then $f$ can be interpolated by a polynomial 
	function of $\ab{A}$. 
\end{proposition}
\begin{proof}
	We first prove that every partial function that preserves
	congruences and types and has its image contained in one $\mu$-class
	can be interpolated by a polynomial function of $\ab{A}$
	whose image is contained in the same $\mu$-class.
	Since  $[\mu,\mu]\neq\mu$ 
	Proposition~\ref{prop:homogenuityandcolpementation}\eqref{item:ex_if_not_neutral_then_abelian_for_hom}
	yields that $\mu$ is abelian. 
By Lemma~\ref{lemma:piu_meno_fanno_modulo_su-anello_polinomi_ristretto_nella_classe_congruenza},
$\ab{A}\restrict{o/\mu}$ is polynomially equivalent to 
the $\ab{R}_o$-module $(o/\mu; +)$. 
By Proposition~\ref{prop:identification_with_matrices}, 
the $\ab{R}_o$-module $(o/\mu; +)$ is isomorphic 
to the $\GF{p}$-module $\Z_p^m$.
Hence it satisfies condition \eqref{item:kapl_m23511}
from Theorem~\ref{teor:result_Kapl_thesis}, and therefore, it
is strictly 1-polynomially rich. Thus,
Lemma~\ref{lemma:restrizione_classe_k_poly_rich_sse:modulo}
yields that
$o/\mu$ is strictly 1-polynomially rich with respect to $\ab{A}$.
Hence
Proposition~\ref{prop:interpolating_partial_functions_type_preserving_inside_a_congruence_class_abelian_case_with_ABp}
yields that any unary type-preserving 
partial function whose domain is 
contained in one $\mu$-class
can be interpolated by a polynomial function
whose image is contained in the
same $\mu$-class. 
Thus, the assumptions of
Lemma~\ref{lemma:fromcosetstotheinfinirtyandbeyond}
are satisfied, and $f$ can be interpolated by a polynomial function. 
\end{proof}
\begin{proposition}\label{prop:interpolation_on_quotient_yields_interpolation_case_type_preserving_AB2,3_binary}
	Let $\ab{A}$ be a finite Mal'cev algebra with (SC1),
	let $\mu$ be an homogeneous abelian atom of $\Con\ab{A}$
	such that $(\ab{A}, \mu)$ satisfies (AB$p$) for $p\in\{2,3\}$,
	let $T$ be a finite subset of $A^2$ and let $f\colon T\to A$
	be a partial type-preserving function that can be interpolated
	modulo $\mu$ by a polynomial function of $\ab{A}$. 
	Then $f$ can be interpolated by a polynomial 
	function of $\ab{A}$. 
\end{proposition}
\begin{proof}
	We first prove that every partial binary function that preserves
	congruences and types and has its image contained in one $\mu$-class
	can be interpolated by a polynomial function of $\ab{A}$
	whose image is contained in the same $\mu$-class.
	Since  $[\mu,\mu]\neq\mu$ 
	Proposition~\ref{prop:homogenuityandcolpementation}\eqref{item:ex_if_not_neutral_then_abelian_for_hom}
	yields that $\mu$ is abelian. 
By Lemma~\ref{lemma:piu_meno_fanno_modulo_su-anello_polinomi_ristretto_nella_classe_congruenza},
$\ab{A}\restrict{o/\mu}$ 
is polynomially equivalent to 
the $\ab{R}_o$-module $(o/\mu; +)$. By Proposition~\ref{prop:identification_with_matrices},
the $\ab{R}_o$-module $(o/\mu; +)$ is isomorphic 
to the $\GF{p}$-module $\Z_p$, hence it satisfies
either condition~\eqref{item:kapl_12123}
or condition~\eqref{item:kapl_1312}
from Theorem~\ref{teor:result_Kapl_thesis}, and therefore,
it is strictly 2-polynomially rich. Thus, Lemma~\ref{lemma:restrizione_classe_k_poly_rich_sse:modulo}
	yields that
	$o/\mu$ is strictly 2-polynomially rich with respect to $\ab{A}$. 
	Therefore,
Proposition~\ref{prop:interpolating_partial_functions_type_preserving_inside_a_congruence_class_abelian_case_with_ABp}
	yields that any binary type-preserving 
	partial function whose domain is 
	contained in one $\mu$-class
	can be interpolated by a polynomial function
	whose image is contained in the
	same $\mu$-class. 
	Thus, the assumptions of Lemma~\ref{lemma:fromcosetstotheinfinirtyandbeyond}
	are satisfied, and $f$ can be interpolated by a polynomial function. 
\end{proof}
\begin{proposition}\label{prop:interpolation_on_quotient_yields_interpolation_case_type_preserving_AB2_atom_arity_3}
	Let $\ab{A}$ be a finite Mal'cev algebra with (SC1),
	let $\mu$ be an homogeneous abelian atom of $\Con\ab{A}$
	such that $(\ab{A}, \mu)$ satisfies (AB$2$),
	let $T$ be a finite subset of $A^3$ and let $f\colon T\to A$
	be a partial type-preserving function that can be interpolated
	modulo $\mu$ by a polynomial function of $\ab{A}$. 
	Then $f$ can be interpolated by a polynomial 
	function of $\ab{A}$. 
\end{proposition}
\begin{proof}
	We first prove that every partial ternary function that preserves
	congruences and types and has its image contained in one $\mu$-class
	can be interpolated by a polynomial function of $\ab{A}$
	whose image is contained in the same $\mu$-class.
	Since  $[\mu,\mu]\neq\mu$ 
	Proposition~\ref{prop:homogenuityandcolpementation}\eqref{item:ex_if_not_neutral_then_abelian_for_hom}
	yields that $\mu$ is abelian. 
	By Lemma~\ref{lemma:piu_meno_fanno_modulo_su-anello_polinomi_ristretto_nella_classe_congruenza},
	 $\ab{A}\restrict{o/\mu}$ 
	is polynomially equivalent to 
	the $\ab{R}_o$-module $(o/\mu; +)$. By Proposition~\ref{prop:identification_with_matrices}, 
	the $\ab{R}_o$-module $(o/\mu; +)$ is isomorphic 
	to the $\GF{2}$-module $\Z_2$. 
	Hence it satisfies condition~\eqref{item:kapl_12123}
	from Theorem~\ref{teor:result_Kapl_thesis}, and therefore, it 
	is strictly 3-polynomially rich.
	Hence Lemma~\ref{lemma:restrizione_classe_k_poly_rich_sse:modulo}
	yields that
	$o/\mu$ is strictly 3-polynomially rich with respect to $\ab{A}$. 
	Thus,	Proposition~\ref{prop:interpolating_partial_functions_type_preserving_inside_a_congruence_class_abelian_case_with_ABp}
	yields that any ternary type-preserving 
	partial function whose domain is 
	contained in one $\mu$-class
	can be interpolated by a polynomial function
	whose image is contained in the
	same $\mu$-class. 
	Thus, the assumptions of Lemma~\ref{lemma:fromcosetstotheinfinirtyandbeyond}
	are satisfied, and $f$ can be interpolated by a polynomial function. 
\end{proof}
\begin{proposition}\label{prop:interpolation_on_quotients_yields_interpolation_case_height_1}
	Let $\ab{A}$ be a finite Mal'cev algebra with (SC1),
	let $\mu$ be a homogeneous abelian atom of $\Con\ab{A}$.
	Let us assume that
	the subtype $q$ of $\quotienttct{\bottom{A}}{\mu}$
	is either $2$ or $3$, and if $q=2$, then each $\mu$-class 
	has cardinality $1$, $2$, $4$, or $8$, and 
	if $q=3$, then each $\mu$-class has cardinality $1$, $3$, or $9$. 
	Let $T\subseteq A$, let $f\colon T\to A$ be a type-preserving
	function that can be interpolated by a polynomial modulo $\mu$.
	Then $f$ can be interpolated by a polynomial on $T$. 
\end{proposition}
\begin{proof}	Proposition~\ref{prop:interpolation_of_partial_type_preserving_in_class_of_atoms}
	implies that
	every partial function that preserves
	congruences and types and has its image contained in one $\mu$-class
	can be interpolated by a polynomial function of $\ab{A}$
	whose image is contained in the same $\mu$-class.
	Thus, the assumptions of Lemma~\ref{lemma:fromcosetstotheinfinirtyandbeyond}
	are satisfied, and $f$ can be interpolated by a polynomial function. 
\end{proof}
\begin{lemma}\label{lemma:if_first_interval_not_complianta_then_no_bueno}
Let $\ab{A}$ be a finite Mal'cev algebra with (SC1),
let $k\in\{1,2,3\}$, let $l\in \N\setminus\{1,2,3\}$,
let $\mu$ be a homogeneous congruence 
of $\ab{A}$. If $\quotienttct{\bottom{A}}{\mu}$ 
is not (CT$k$),
then $\ab{A}$ is not strictly $k$-polynomially rich. 
Moreover, if $\mu$ is abelian, then 
$\ab{A}$ is not strictly $l$-polynomially rich. 
\end{lemma}
\begin{proof}
First, we observe that if $\mu$ is not abelian,
then by Proposition~\ref{prop:typesandhomogenuity}
$\mu$ is a neutral atom, and therefore,
$\quotienttct{\bottom{A}}{\mu}$
is (CT$k$).  
Next, we assume that $\mu$ is not neutral.
Then by
Proposition~\ref{prop:homogenuityandcolpementation}\eqref{item:ex_if_not_neutral_then_abelian_for_hom},
$\mu$ is abelian. 
If $\ab{A}$ is strictly $k$ polynomially rich, then 
for each choice of $o\in A$ we have that every $k$-ary 
partial type-preserving function from $(o/\mu)^k$ to $o/\mu$
can be interpolated by a basic operation of $\ab{A}\restrict{o/\mu}$.
In fact let $f$ be a partial $k$-ary function such that
$\dom f\subseteq (o/\mu)^k$ and let $p\in \POL\ari{k}\ab{A}$
be a polynomial that interpolates $f$. 
Let $e_\mu^o$ be the idempotent polynomial 
constructed in 
\cite[Theorem~6.12]{Ros24}.
Then $e_\mu^o\circ p$ interpolates $f$ on its domain and it
is a basic operation of $\ab{A}\restrict{o/\mu}$. 
Thus, Lemma~\ref{lemma:restrizione_classe_k_poly_rich_sse:modulo}
implies that the $\ab{M}_n(\ab{D})$-module $\ab{D}^{(n\times m)}$ 
constructed in Proposition~\ref{prop:identification_with_matrices}
is strictly $k$-polynomially rich. Thus, it satisfies 
one of the conditions from Theorem~\ref{teor:result_Kapl_thesis}. 
Thus, one immediately verifies that $\quotienttct{\bottom{A}}{\mu}$
is (CT$k$).
Finally, let us consider the case that $\mu$ is 
abelian and $\ab{A}$ is strictly $l$-polynomially rich.
By the above argument, the $\ab{M}_n(\ab{D})$-module $\ab{D}^{(n\times m)}$ 
constructed in Proposition~\ref{prop:identification_with_matrices}
is strictly $l$-polynomially rich. This contradicts Theorem~\ref{teor:result_Kapl_thesis}. 
\end{proof}
\section{Proofs of the results from Section~\ref{sec:content}}\label{sec:proof_main_results}
\begin{proof}[Proof of Theorem \ref{teor:main_result_partial_type_preserving}]
	We proceed by induction on $\card{A}$. If $\card{A}=1$, then 
	$\ab{A}$ is strictly k-polynomially rich for all $k\in\N$,
	hence there is nothing to prove. 
	Let $l\in\{1,2,3\}$.
	For the induction step, we assume that  $\card{A}>1$ and that
	the statement is true for 
	each algebra of cardinality strictly less than $\card{A}$,
	and prove it for $\ab{A}$.
	If $l=1$
	we assume that for each $i\in \finset{n}$
	one of the conditions 
	\eqref{item:main_theorem_poly_rich_item_neutral}-\eqref{item:main_theorem_poly_rich_item_ABp} from 
	Definition~\ref{def:ex_main_theorem}
	is satisfied, and we prove that
	$\ab{A}$ is strictly 1-polynomially rich;
	if $l=2$ we assume that for each $i\in \finset{n}$
	one of the conditions 
	\eqref{item:main_theorem_poly_rich_item_neutral_2},
	\eqref{item:main_theorem_poly_rich_item_ABp_2}
	from 
	Definition~\ref{def:ex_main_theorem}
	is satisfied, and we prove that
	$\ab{A}$ is strictly 2-polynomially rich;
	if $l=3$
	we assume that for each $i\in \finset{n}$
	one of the conditions 
	\eqref{item:main_theorem_poly_rich_item_neutral_3},
	\eqref{item:main_theorem_poly_rich_item_ABp_3}
	from 
	Definition~\ref{def:ex_main_theorem}
	is satisfied, and we prove that
	$\ab{A}$ is strictly 3-polynomially rich.
	Let us consider the algebra $\ab{A}/\mu_1$. 
	By the induction hypothesis $\ab{A}/\mu_1$ 
	is strictly $l$-polynomially rich.
	Let $f$ be a partial type-preserving function with
	domain $T\subseteq A^l$. 
	Since $\ab{A}/\mu_1$ is strictly $l$-polynomially rich,
	$f$ can be interpolated modulo $\mu_1$ by a polynomial of $\ab{A}$.
	
	Next we assume that $l=1$ and 
	split the proof into three cases according to 
	which among 
	\eqref{item:main_theorem_poly_rich_item_neutral}-\eqref{item:main_theorem_poly_rich_item_ABp}
	is satisfied by $i=1$. 
	If \eqref{item:main_theorem_poly_rich_item_neutral}
	holds for $i=1$, then the result follows from 
	Proposition~\ref{prop:interpolation_on_quotient_yields_interpolation_case_neutral};
	if \eqref{item:main_theorem_poly_rich_item_subtype2} or
	\eqref{item:main_theorem_poly_rich_item_subtype3}
	hold for $i=1$, then the result follows from
	Proposition~\ref{prop:interpolation_on_quotients_yields_interpolation_case_height_1};
	if \eqref{item:main_theorem_poly_rich_item_ABp}
	holds for $i=1$, then the result follows from 
	Proposition~\ref{prop:interpolation_on_quotient_yields_interpolation_case_type_preserving_ABp}.
	
	Next, we assume $l=2$ and split the proof into two cases according to
	which among \eqref{item:main_theorem_poly_rich_item_neutral_2},
	\eqref{item:main_theorem_poly_rich_item_ABp_2}
	is satisfied by $i=1$. 
	If \eqref{item:main_theorem_poly_rich_item_neutral_2} holds for $i=1$,
	then the result follows from Proposition~\ref{prop:interpolation_on_quotient_yields_interpolation_case_neutral}.
	If \eqref{item:main_theorem_poly_rich_item_ABp_2} holds for $i=1$,
	the the result follows from 
	Proposition~\ref{prop:interpolation_on_quotient_yields_interpolation_case_type_preserving_AB2,3_binary}.
	
	Next, we assume $l=3$ and split the proof into two cases according to
	which among \eqref{item:main_theorem_poly_rich_item_neutral_3},
	\eqref{item:main_theorem_poly_rich_item_ABp_3}
	is satisfied by $i=1$.
	If \eqref{item:main_theorem_poly_rich_item_neutral_3} holds for $i=1$,
	then the result follows from Proposition~\ref{prop:interpolation_on_quotient_yields_interpolation_case_neutral}.
	If \eqref{item:main_theorem_poly_rich_item_ABp_3} holds for $i=1$,
	the the result follows from 
Proposition~\ref{prop:interpolation_on_quotient_yields_interpolation_case_type_preserving_AB2_atom_arity_3}.
\end{proof}
\begin{proof}[Proof of 
Theorem~\ref{teor:char_hered_scr_1_p_r}]
We first assume that each algebra in $\Hh\ab{A}$
is strictly 1-polynomially rich
and show that $\ab{A}$ satisfies (SC1) and each homogeneous 
series in $\Con\ab{A}$ is (CT1). 
Since $\ab{A}\in\Hh\ab{A}$,
$\ab{A}$ is strictly 1-polynomially rich. 
Hence Theorem~\ref{teor:sci_necessary}\eqref{item:regular_in_SC1_nce}
implies that $\ab{A}$ is (SC1). 
Let $(\mu_0, \dots, \mu_{n})$ be a 
homogeneous series for $\Con\ab{A}$, and let 
$i\in\finset{n}$. 
By definition of homogeneous series,
$\mu_i/\mu_{i-1}$ is a homogeneous congruence 
of $\ab{A}/\mu_{i-1}$. Since $\ab{A}/\mu_{i-1}$ is 
strictly 1-polynomially rich, 
Lemma~\ref{lemma:if_first_interval_not_complianta_then_no_bueno}
implies that $\quotienttct{\mu_{i-1}}{\mu_i}$ 
is (CT1). 

Next, we assume that $\ab{A}$ satisfies (SC1) and each homogeneous 
series in $\Con\ab{A}$ is (CT1)
and prove that each algebra in $\Hh\ab{A}$
is strictly 1-polynomially rich.
To this end, let $\alpha\in \Con\ab{A}$, we show
that $\ab{A}/\alpha$ is strictly 1-polynomially rich. 
Theorem~\ref{teor:main_result_on_homogeneous_sequences_on_quotients}
implies that $\ab{A}/\alpha$ is (SC1) and
either $\ab{A}$ is 
congruence neutral, or each 
homogeneous series in $\ab{A}/\alpha$ 
is (CT1). 
If $\ab{A}/\alpha$ is congruence neutral, 
then by \cite[Proposition~5.2]{Ai:OHAH},
$\ab{A}/\alpha$ is strictly 2-affine complete 
and hence strictly 1-polynomially rich.
If $\ab{A}/\alpha$ is not congruence neutral, 
then each homogeneous series for $\Con(\ab{A}/\alpha)$
is (CT1),
and $\ab{A}/\alpha$ is strictly 1-polynomially
rich by Theorem~\ref{teor:main_result_partial_type_preserving}.
\end{proof}

\begin{proof}[Proof of Theorem~\ref{teor:char_hered_scr_23_p_r}]
The proof follows the same pattern as the proof of 
Theorem~\ref{teor:char_hered_scr_1_p_r} with $k$ in
place of 1, with the only notable 
difference that in order to prove that 
\eqref{item:HA_in_char_hered_scr_23_p_r}
implies (SC1) we use 
Theorem~\ref{teor:sci_necessary}\eqref{item:k_in_sc1_necessary}.
\end{proof}

\begin{proof}[Proof of Theorem~\ref{teor:char_hered_scr_4_p_r}]
The equivalence of \eqref{item:2-affine_in_char_hered_scr_4_p_r},
\eqref{item:l-affine_in_char_hered_scr_4_p_r}
and \eqref{item:neutral_in_char_hered_scr_4_p_r} 
follows from \cite[Proposition~5.2]{Ai:OHAH}. 
Moreover, it is a straightforward consequence 
of the definitions that
\eqref{item:l-affine_in_char_hered_scr_4_p_r}
implies \eqref{item:HAk_in_char_hered_scr_4_p_r}
and that \eqref{item:HAk_in_char_hered_scr_4_p_r}
implies \eqref{item:HA_in_char_hered_scr_4_p_r}.
Next, we show that \eqref{item:HA_in_char_hered_scr_4_p_r}
implies \eqref{item:neutral_in_char_hered_scr_4_p_r}. 
Theorem~\ref{teor:sci_necessary}\eqref{item:k_in_sc1_necessary}
implies that $\ab{A}$ is (SC1), hence
Proposition~\ref{prop:existence_of_hom_congruences_under_SC1}
implies that there exists a homogeneous series
$(\mu_0, \dots, \mu_{n})$ for $\Con\ab{A}$. 
Lemma~\ref{lemma:if_first_interval_not_complianta_then_no_bueno}
implies that $\mu_1$ is not abelian.
Hence Proposition~\ref{prop:typesandhomogenuity}
implies that $\mu_1$ is an atom. 
By the definition of homogeneous series 
$\mu_i/\mu_{i-1}$ is homogeneous in $\ab{A}/\mu_{i-1} $
for each $i\in\finset{n}$. Hence, since 
$\ab{A}/\mu_{i-1}$ is strictly 4-polynomially rich,
Lemma~\ref{lemma:if_first_interval_not_complianta_then_no_bueno}
implies that $\mu_{i}/\mu_{i-1}$ is not abelian 
and Proposition~\ref{prop:typesandhomogenuity}
implies that $\mu_i/\mu_{i-1}$ is an atom. 
Thus, $(\mu_0, \dots, \mu_{n})$ is a
maximal chain in $\Con\ab{A}$ and each interval is not abelian. 
Let $\quotienttct{\alpha}{\beta}$ be a prime quotient in 
$\Con\ab{A}$, and let $\gamma_0, \dots, \gamma_{n}$
be a maximal chain of $\Con\ab{A}$ such that
there exists $i\in\finset{n}$ with 
$\gamma_{i-1}=\alpha$ and $\gamma_i=\beta$. 
By the Dedekind-Birkhoff Theorem 
(cf.~\cite[Theorem~2.37]{MMT:ALVV})
there exists $j\in\finset{n}$ such that
$\interval{\gamma_{i-1}}{\gamma_i}
\leftrightsquigarrow\interval{\mu_{j-1}}{\mu_j}$,
and therefore, by Lemma~\ref{lemma:exAic18Lemma3.4},
$[\beta, \beta]=\beta$. 
Thus, we can conclude that each prime quotient 
in $\Con\ab{A}$ is not abelian, and  
\eqref{item:neutral_in_char_hered_scr_4_p_r} follows.
\end{proof}
\bibliographystyle{amsalpha}
\bibliography{pint-partial1a}

\end{document}